\documentclass[a4paper,notitlepage,fleqn,11pt]{article}
\usepackage[tbtags]{amsmath}
\usepackage{amsfonts}
\usepackage{amsthm}
\usepackage{color}
\usepackage{graphicx,amssymb}
\usepackage{multirow}
\usepackage{rotating}
\usepackage{mathrsfs}
\usepackage{epstopdf}
\usepackage{enumerate}
\usepackage{pdflscape}
\usepackage{geometry}
\usepackage{natbib}
\usepackage{changepage}
\usepackage{dcolumn}
\usepackage{tabularx}
\usepackage{subfig}
\usepackage{mathabx}
\numberwithin{equation}{section}
\newtheorem{thm}{Theorem}
\newtheorem{cor}{Corollary}
\newtheorem{lemma}{Lemma}
\newtheorem{assum}{Assumption}
\newtheorem{remark}{Remark}

\DeclareMathOperator*{\argmin}{argmin}

\DeclareMathOperator*{\sgn}{sgn}

\DeclareMathOperator*{\var}{var}
\newcommand{\bm}[1]{\mbox{\boldmath{$#1$}}}

\renewcommand{\baselinestretch}{1.5}
\usepackage{xr}
\title{Quantile double autoregression}
\author{Qianqian Zhu and Guodong Li\\
\textit{Shanghai University of Finance and Economics and}\\
\textit{University of Hong Kong}}

\begin{document}
\maketitle
\begin{abstract}
Many financial time series have varying structures at different quantile levels, and also exhibit the phenomenon of conditional heteroscedasticity at the same time.
In the meanwhile, it is still lack of a time series model to accommodate both of the above features simultaneously.
This paper fills the gap by proposing a novel conditional heteroscedastic model, which is called the quantile double autoregression.
The strict stationarity of the new model is derived, and a self-weighted conditional quantile estimation is suggested.
Two promising properties of the original double autoregressive model are shown to be preserved.
Based on the quantile autocorrelation function and self-weighting concept, three portmanteau tests are constructed to check the adequacy of fitted conditional quantiles.
The finite-sample performance of the proposed inference tools is examined by simulation studies, and the necessity of the new model is further demonstrated by analyzing the S\&P500 Index.
\end{abstract}

\newpage
\section{Introduction}

The conditional heteroscedastic models have become a standard family of nonlinear time series models since the introduction of \citeauthor{Engle1982}'s \citeyearpar{Engle1982} autoregressive conditional heteroscedastic (ARCH) and \citeauthor{Bollerslev1986}'s \citeyearpar{Bollerslev1986} generalized autoregressive conditional heteroscedastic (GARCH) models.
Among existing conditional heteroscedastic models, the double autoregressive (AR) model recently has attracted more and more attentions; see \cite{Ling2004,Ling2007}, \cite{Ling_Li2008}, \cite{Zhu_Ling2013}, \cite{Li_Ling_Zhang2016}, \cite{Li_Zhu_Liu_Li2017}, \cite{Zhu_Zheng_Li2018} and references therein.
This model has the form of
\begin{equation}\label{dar0}
y_t=\sum_{i=1}^p\phi_{i}y_{t-i} +\varepsilon_t\sqrt{\omega+\sum_{j=1}^p\beta_{j}y_{t-j}^{2}},
\end{equation}
where $\omega>0$, $\beta_j\geq 0$ with $1\leq j\leq p$, and $\{\varepsilon_t\}$ are identically and independently distributed ($i.i.d.$) random variables with mean zero and variance one.
It is a special case of AR-ARCH models in \cite{Weiss1986}, and will reduce to \citeauthor{Engle1982}'s \citeyearpar{Engle1982} ARCH model when all $\phi_{i}$'s are zero.
The double AR model has two novel properties.
First, it has a larger parameter space than that of the commonly used AR model. For example, when $p=1$, the double AR model may still be stationary even as $|\phi_1|\geq 1$ \citep{Ling2004}, whereas this is impossible for AR-ARCH models.
Secondly, no moment condition on $y_t$ is needed to derive the asymptotic normality of the Gaussian quasi-maximum likelihood estimator (QMLE) \citep{Ling2007}. This is in contrast to the ARMA-GARCH model, for which the finite fourth moment of the process is unavoidable in deriving the asymptotic distribution of the Gaussian QMLE \citep{Francq_Zakoian2004}, resulting in a much narrower parameter space \citep{Li_Li2009}.

In the meanwhile, conditional heteroscedastic models are considered mainly for modeling volatility and financial risk. Some quantile-based measures, such as the value-at-risk (VaR), expected shortfall and limited expected loss, are intimately related to quantile estimation, see, e.g., \cite{Artzner_Delbaen_Eber_Heath1999}, \cite{Wu_Xiao2002}, \cite{Bassett_Koenker_Kordas2004} and \cite{Francq_Zakoian2015}.
Therefore, it is natural to consider the quantile estimation for conditional heteroscedastic models.
Many researchers have investigated the conditional quantile estimation (CQE) for the (G)ARCH models; see \cite{Koenker_Zhao1996} for linear ARCH models, \cite{Xiao_Koenker2009} for linear GARCH models, \cite{Lee_Noh2013} and \cite{Zheng_Zhu_Li_Xiao2016} for quadratic GARCH models.
\cite{Chan_Peng2005} considered a weighted least absolute deviation estimation for double AR models with the order of $p=1$, and \cite{Zhu_Ling2013} studied the quasi-maximum exponential likelihood estimation for a general double AR model.
\cite{Cai_Montes_Olmo2013} considered a Bayesian estimation of model \eqref{dar0} with $\varepsilon_t$ following a generalized lambda distribution.
However, it is still open to perform the CQE for double AR models.

In addition, for the double AR process generated by model \eqref{dar0}, its $\tau$th conditional quantile has the form of
\begin{equation}\label{darquantile}
Q_{\tau}(y_t|\mathcal{F}_{t-1})=\sum_{i=1}^p\phi_{i}y_{t-i} +b_{\tau}\sqrt{\omega+\sum_{j=1}^p\beta_{j}y_{t-j}^{2}}, \hspace{5mm}0<\tau<1,
\end{equation}
and the coefficients of $\phi_i$'s are all $\tau$-independent, where $\mathcal{F}_t$ is the $\sigma$-field generated by $\{y_s,s\leq t\}$, and $b_{\tau}$ is the $\tau$th quantile of $\varepsilon_t$; see also \cite{Cai_Montes_Olmo2013}.
However, when modeling the closing prices of S\&P500 Index by the CQE, we found that the estimated coefficients of $\phi_i$'s depend on the quantile level significantly, while those of $\beta_j$'s also slightly depend on the quantile level; see Figures \ref{IntroFunction} and \ref{FunctionPlots} in Section 6 for empirical evidences. Actually this phenomenon can also be found in many other stock indices.
\cite{Koenker_Xiao2006} proposed a quantile AR model by extending the common autoregression, and the corresponding coefficients are defined as functions of quantile levels.
By adapting the method in \cite{Koenker_Xiao2006}, this paper attempts to introduce a new conditional heteroscedastic model, called the quantile double autoregression, to better interpret the financial time series with the above phenomenon. We state our first main contribution below in details.
\begin{itemize}
\item[(a)] A direct extension of Koenker and Xiao's method will result in a strong constraint that the coefficients of $y_{t-j}^2$'s will be zero at a certain quantile level simultaneously; see Section \ref{Section_process}. A novel transformation of $S_{\textrm{Q}}(x)=\sqrt{|x|}\sgn(x)$ is first introduced to the conditional scale structure in Section \ref{Section_process}, and hence the drawback can be removed from the derived model. 
Moreover, there is no nonnegative constraint on coefficient functions, and this makes numerical optimization much flexible.
Section \ref{Section_process} also establishes the strict stationarity and ergodicity of the newly proposed model, and the first novel property of double AR models is preserved.
\end{itemize}

In the meanwhile, many financial time series are heavy-tailed, and the least absolute deviation (LAD) estimator of infinite variance AR models usually has a faster convergence rate, but a more complicated asymptotic distribution, than the commonly encountered asymptotic normality; see, e.g., \cite{Gross_Steiger1979}, \cite{AN_Chen1982} and \cite{Davis_Knight_Liu1992}. 
\cite{Ling2005} proposed a self-weighted LAD estimation, and the weights are used to reduce the moment condition on processes. As a result, the asymptotic normality can be obtained although the super-consistency is sacrificed.
The self-weighting approach has also been adopted to the quasi-maximum likelihood estimation for ARMA-GARCH models \citep{Zhu_Ling2011} and augmented double AR models \citep{Jiang2020}, and to the conditional quantile estimation for linear double AR models \citep{Zhu_Zheng_Li2018}.
The CQE may suffer from the similar problem. Actually the finite third moment on $\{y_t\}$ will be inevitable for the asymptotic normality of the CQE \citep{Zhu_Ling2011}, and this will make the resulting parameter space narrower; see Section \ref{Section_process}. 
\begin{itemize}
	\item[(b)] Motivated by \cite{Ling2005}, this paper considers a self-weighted CQE for the newly proposed model in Section \ref{Section_estimation}. Only a fractional moment on $\{y_t\}$ is needed for the asymptotic normality, and the second novel property of double AR models is hence preserved. As a result, the proposed methodology can be applied to very heavy-tailed data.
	Moreover, the objective function in this paper is non-differentiable and non-convex, and this causes great challenges in asymptotic derivations. Section \ref{Section_estimation} overcomes the difficulty by adopting the bracketing method in \cite{Pollard1985}, and this is another important contribution of this paper. 
\end{itemize}

According to the Box-Jenkins' three-stage modeling strategy, it is important to check the adequacy of fitted conditional quantiles, and the autocorrelation function of residuals is supposed to play an important role in this literature; see \cite{Box_Jenkins1970} and \cite{Li2004}.
\cite{Li_Li_Tsai2015} introduced a quantile correlation to measure the linear relationship between any two random variables for the given quantile in quantile regression settings, and proposed the quantile autocorrelation function (QACF) to assess the adequacy of fitted quantile AR models \citep{Koenker_Xiao2006}.
In the meanwhile, the diagnostic checking for the conditional location usually uses fitted residuals \citep{Ljung_Box1978}, while that for the conditional scale employs the absolute residuals \citep{Li_Li2005}. 
It is noteworthy that the proposed quantile double AR model consists of both the conditional location and scale parts.
\begin{itemize}
	\item[(c)] In line with the estimating procedure in Section \ref{Section_estimation}, Section \ref{Section_checking} first introduces a self-weighted QACF of residuals to measure the quantile correlation, and constructs two portmanteau tests to check the adequacy of the fitted conditional quantile in the conditional location and scale components, respectively. A combined test is also considered. 
	This is the third main contribution of this paper.
\end{itemize}

In addition, Section \ref{Section_simulation} conducts simulation studies to evaluate the finite-sample performance of self-weighted estimators and portmanteau tests, and Section \ref{Section_realdata} presents an empirical example to illustrate the usefulness of the new model and its inference tools. Conclusion and discussion are made in Section \ref{Section_conclusion}. All the technical details are relegated to the Appendix.
Throughout the paper, $\rightarrow_d$ denotes the convergence in distribution, and $o_p(1)$ denotes a sequence of random variables converging to zero in probability. We denote by $\|\cdot\|$ the norm of a matrix or column vector, defined as $\|A\|=\sqrt{\text{tr}(AA^\prime)}=\sqrt{\sum_{i,j}a_{ij}^2}$.

\section{Quantile double autoregression}\label{Section_process}

This section introduces a new conditional heteroscedastic model by generalizing the double AR model at \eqref{darquantile} or \eqref{dar0}, and the new model can also be used to accommodate the phenomenon that the financial time series may have varying structures at different quantile levels. 

We first consider a direct extension of model \eqref{darquantile}, which, as in \cite{Zhu_Zheng_Li2018}, can be reparametrized into 
$Q_{\tau}(y_t|\mathcal{F}_{t-1})=\sum_{i=1}^p\phi_{i}y_{t-i} +b_{\tau}^*\sqrt{1+\sum_{j=1}^p\beta_{j}^*y_{t-j}^{2}}$, where $b_{\tau}^*=b_{\tau}\sqrt{\omega}$ and $\beta_{j}^*=\beta_{j}/\omega$ with $1\leq j\leq p$.
Along the line of quantile autoregression \citep{Koenker_Xiao2006}, by letting the corresponding coefficients depend on the quantile level $\tau$, we then have
\begin{equation}\label{addeq1}
Q_{\tau}(y_t|\mathcal{F}_{t-1}) =\sum_{i=1}^p\phi_{i}(\tau)y_{t-i} +b^*(\tau)\sqrt{1+\sum_{j=1}^p\beta_j^*(\tau)y_{t-j}^{2}},\hspace{5mm} 0<\tau<1.
\end{equation}
However, if there exists $\tau$ such that $b^*(\tau)=0$, all $y_{t-j}^2$'s will then disappear from the quantile structure at this level, i.e. the contributions of all $y_{t-j}^2$'s are zero at a certain quantile level simultaneously. 
In addition, both coefficient functions $b^{*}(\tau)$ and $\beta_j^*(\tau)$ are related to the same term of $y_{t-j}^2$, and it may not be a good idea to separate the influence of $y_{t-j}^2$ on $Q_{\tau}(y_t|\mathcal{F}_{t-1})$ into two $\tau$-dependent functions. As a result, the extension at \eqref{addeq1} may not be a good choice.

This paper attempts to tackle the drawback of model \eqref{addeq1} by looking for a way to move $b_{\tau}\in \mathbb{R}$ at \eqref{darquantile} inside the square root. It then leads to a transformation of $S_{\textrm{Q}}(x)=\sqrt{|x|}\sgn(x)$, which is an extension of the square root function with the support from $\mathbb{R}_+=[0,+\infty)$ to $\mathbb{R}$, where $\sgn(\cdot)$ is the sign function.
Specifically, model \eqref{darquantile} can be reparameterized into
$Q_{\tau}(y_t|\mathcal{F}_{t-1})=\sum_{i=1}^p\phi_{i}y_{t-i} +S_{\textrm{Q}} (b_{\tau}^{\star}+\sum_{j=1}^p\beta_{j}^{\star}y_{t-j}^{2})$ alternatively, where $b_{\tau}^{\star}=b_{\tau}^2\omega$ and $\beta_{j}^{\star}=b_{\tau}^2\beta_{j}$ with $1\leq j\leq p$. By letting $\phi_i$'s, $b_{\tau}^{\star}$ and $\beta_j^{\star}$'s depend on the quantile level $\tau$, we then propose the new conditional heteroscedastic model below,
\begin{equation}\label{qdar}
Q_{\tau}(y_t|\mathcal{F}_{t-1}) =\sum_{i=1}^p\phi_{i}(\tau)y_{t-i} +S_{\textrm{Q}}\left(b(\tau)+\sum_{j=1}^p\beta_{j}(\tau)y_{t-j}^{2}\right), \quad \text{for all }0<\tau<1,
\end{equation}
where $b(\cdot)$, $\phi_{i}(\cdot)$'s and $\beta_{j}(\cdot)$'s with $1\leq i,j\leq p$ are continuous functions $(0,1)\rightarrow \mathbb{R}$, and it will reduce to model \eqref{darquantile} when $b(\tau)=b_{\tau}^{\star}$, $\phi_{j}(\tau)=\phi_{j}$ and $\beta_{j}(\tau)=\beta_{j}^{\star}$ with $1\leq j\leq p$. 
It is not necessary to assume that $b(\cdot)$ and $\beta_{j}(\cdot)$'s are all equal to zero at a certain quantile level, and the drawback of the definition at \eqref{addeq1} is hence removed. 

Assume that, as in \cite{Koenker_Xiao2006}, the right hand side of \eqref{qdar} is an increasing  function with respect to $\tau$.
Let $\{u_t\}$ be a sequence of $i.i.d.$ standard uniform random variables, and then model \eqref{qdar} has an equivalent form below,
\begin{equation}\label{qdarprocess}
y_t=\sum_{i=1}^p\phi_{i}(u_t)y_{t-i} + S_{\textrm{Q}}\left(b(u_t)+\sum_{j=1}^p\beta_{j}(u_t)y_{t-j}^{2}\right).
\end{equation}
It will reduce to the quantile AR model in \cite{Koenker_Xiao2006} when $\beta_j(u_t)$'s are zero with probability one.
We call model \eqref{qdar} or \eqref{qdarprocess} the quantile double autoregression (AR) for simplicity.

\begin{remark}\label{remark1-special cases}
	The proposed quantile double AR is a generalization of double AR models at \eqref{dar0}. 
	Specifically, let $\omega=E(|b(u_t)|)=\int_0^1|b(\tau)|d\tau$ and $\varepsilon_t=S_Q[b(u_t)]/\sqrt{\omega}$. Assume that $\phi_i(u_t)=\phi_i$ with $1\leq i\leq p$, and $\beta_j(u_t)=b(u_t)\beta_j/\omega$ with $1\leq j\leq p$ with probability one, and then 
	\eqref{qdarprocess} will have the same form as in \eqref{dar0}.
	Moreover, we may further assume that $E(\varepsilon_t)=\int_0^1S_Q[b(\tau)]d\tau/\sqrt{\omega}=0$.
\end{remark}

\begin{remark}
As mentioned by \cite{Koenker_Xiao2006}, it is very hard to give an explicit condition such that the right hand side of \eqref{qdar} is an increasing random function with respect to $\tau$.
We may assume that all $b(\cdot)$ and $\beta_{j}(\cdot)$'s with $1\leq j\leq p$ are all increasing functions.
Note that $S_{\textrm{Q}}(\cdot)$ is continuous and monotonically increasing, and hence the second term at the right hand side of \eqref{qdar} will be increasing with respect to $\tau$.
Following \cite{Koenker_Xiao2006}, we further assume that $\sum_{i=1}^p\phi_{i}(\tau)y_{t-i}$ is an increasing function with respect to $\tau$, and it is then guaranteed that \eqref{qdar} defines a qualified conditional quantile function.
\end{remark}

\begin{remark}
The AR-ARCH model is another popular conditional heteroscedastic model in the literature, and it is certainly of interest to consider a similar generalization for it.
The AR-ARCH model has the form of,
\[
y_t=\sum_{i=1}^p\phi_i y_{t-i}+e_t, \hspace{5mm} e_t=\varepsilon_t\sigma_t, \hspace{5mm} \sigma_t^2=\omega+\sum_{j=1}^p\beta_j e_{t-j}^2,
\]
where the coefficients and $\{\varepsilon_t\}$ are defined as in \eqref{dar0}, and its $\tau$th conditional quantile is given below,
\[
Q_{\tau}(y_t|\mathcal{F}_{t-1})=\sum_{i=1}^p\phi_{i}y_{t-i} +b_{\tau}\sqrt{\omega+\sum_{j=1}^p\beta_{j}[y_{t-j} -\sum_{i=1}^p\phi_{i}y_{t-j-i}]^{2}}, \hspace{5mm}0<\tau<1,
\]
which, however, has a more complicated form than model \eqref{darquantile}.
This makes it impossible to reach an easy-to-interpret quantile model.
In addition, comparing with AR-ARCH models, the double AR has many other advantages; see Section 1.
\end{remark}

Let $\bm y_t=(y_t,y_{t-1}, \ldots, y_{t-p+1})^{\prime}$, where $\{y_t\}$ is a quantile double AR process generated by model \eqref{qdarprocess}.
It can be verified that $\{\bm y_t\}$ is a homogeneous Markov chain on the state space $(\mathbb{R}^p, \mathcal{B}^p, \nu^p)$, where $\mathcal{B}^p$ is the class of Borel sets of $\mathbb{R}^p$, and $\nu^p$ is the Lebesgue measure on $(\mathbb{R}^p, \mathcal{B}^p)$.
Denote by $F(\cdot)$ and $f(\cdot)$ the distribution and density functions of $b(u_t)$, respectively.
It holds that $F(\cdot)=b^{-1}(\cdot)$ and $E|b(u_t)|^{\delta}=\int_0^1|b(\tau)|^{\delta}d\tau$ for any $\delta>0$.

\begin{assum}\label{assum0}
The density function $f(\cdot)$ is positive and continuous on $\{x\in\mathbb{R}: 0<F(x)<1\}$, and $\int_0^1|b(\tau)|^{\kappa/2}d\tau<\infty$ for some $0<\kappa\leq 1$.
\end{assum}

\begin{thm}\label{thm0}
	Under Assumption \ref{assum0}, if
	\begin{equation*}
\sum_{i=1}^{p}\max\left\{\int_0^1\big|\phi_{i}(\tau)-\sqrt{|\beta_{i}(\tau)|}\big|^\kappa d\tau, \int_0^1\big|\phi_{i}(\tau)+\sqrt{|\beta_{i}(\tau)|}\big|^\kappa d\tau\right\}<1,
	\end{equation*}
	then there exists a strictly stationary solution $\{y_t\}$ to the quantile double AR model at \eqref{qdarprocess}, and this solution is unique and geometrically ergodic with $E|y_t|^\kappa<\infty$.
\end{thm}

Since $u_t$ is a standard uniform random variable, the condition in the above theorem is equivalent to
$\sum_{i=1}^{p}\max\left\{E\big|\phi_{i}(u_{t})-\sqrt{|\beta_{i}(u_{t})|}\big|^\kappa, E\big|\phi_{i}(u_{t})+\sqrt{|\beta_{i}(u_{t})|}\big|^\kappa\right\}<1
$.
Moreover, for the double AR model at \eqref{dar0}, \cite{Ling2007} derived its stationarity condition for the case with $\varepsilon_t$ being a standard normal random variable, while it is still unknown for the other distributions of $\varepsilon_t$. From Theorem \ref{thm0} and Remark \ref{remark1-special cases}, we can obtain a stationarity condition for a general double AR model below.

\begin{cor}\label{cor1}
Suppose that $\varepsilon_t$ has a positive and continuous density on its support, and $E|\varepsilon_t|^{\kappa}<\infty$ for some $0<\kappa\leq 1$.  If $\sum_{i=1}^{p} \max\left\{E|\phi_{i}-\varepsilon_t\sqrt{\beta_{i}}|^\kappa, E|\phi_{i}+\varepsilon_t\sqrt{\beta_{i}}|^\kappa\right\}<1$, then there exists a strictly stationary solution $\{y_t\}$ to the double AR model at \eqref{dar0}, and this solution is unique and geometrically ergodic with $E|y_t|^\kappa<\infty$.
\end{cor}

For the case with the order of $p=1$, when $\varepsilon_t$ is symmetrically distributed, the above condition can be simplified to $E|\phi_1+\varepsilon_t\sqrt{\beta_1}|^\kappa<1$.
Moreover, if the normality of $\varepsilon_t$ is further assumed, the necessary and sufficient condition for the strict stationarity is then $E\ln|\phi_1+\varepsilon_t\sqrt{\beta_1}|<0$ \citep{Ling2007}.
The comparison of the above stationarity regions is illustrated in the left panel of Figure \ref{StationaryRegion}.
It can be seen that a larger value of $\kappa$ in Corollary \ref{cor1} leads to a higher moment of $y_t$,  and hence results in a narrower stationarity region.
Figure \ref{StationaryRegion} also gives the stationarity regions with different distributions of $\varepsilon_t$. As expected, the parameter space of model \eqref{dar0} becomes smaller as $\varepsilon_t$ gets more heavy-tailed.


\section{Self-weighted conditional quantile estimation}\label{Section_estimation}

Let $\bm y_{1,t}=(y_{t},...,y_{t-p+1})^{\prime}$ and $\bm y_{2,t}=(y_{t}^2,...,y_{t-p+1}^2)^{\prime}$.
For the proposed quantile double AR model at \eqref{qdar}, denote by $\bm\theta_{\tau}=(\bm \phi^{\prime}(\tau),b(\tau),\bm \beta^{\prime}(\tau))^{\prime}$ the parameter vector, where $\bm\phi(\tau)=(\phi_{1}(\tau),\ldots,\phi_{p}(\tau))^{\prime}$ and $\bm\beta(\tau)=(\beta_{1}(\tau),\ldots,\beta_{p}(\tau))^{\prime}$.
We then can define the conditional quantile function below,
\[
q_t(\bm\theta_{\tau}) =\bm y_{1,t-1}^{\prime}\bm \phi(\tau) +S_{\textrm{Q}}\left(b(\tau)+\bm y_{2,t-1}^{\prime}\bm \beta(\tau)\right),
\]
and this paper considers a self-weighted conditional quantile estimation (CQE) for model \eqref{qdar},
\begin{equation}\label{WCQE}
\widehat{\bm\theta}_{\tau n}=(\widehat{\bm \phi}_{n}^{\prime}(\tau), \widehat{b}_{n}(\tau), \widehat{\bm \beta}_{n}^{\prime}(\tau))^{\prime} =\argmin_{\bm\theta_{\tau}}\sum_{t=p+1}^nw_{t}\rho_{\tau} \left(y_t-q_t(\bm\theta_{\tau})\right),
\end{equation}
where $\rho_{\tau}(x)=x[\tau-I(x<0)]$ is the check function, and $\{w_t\}$ are nonnegative random weights; see also \cite{Ling2005} and \cite{Zhu_Ling2011}. 

\begin{remark}
When $w_t=1$ for all $t$ with probability one, the self-weighted CQE will become the common CQE. 
A third order moments will be needed for the asymptotic normality, and this will lead to a much narrower stationarity region; see Figure \ref{StationaryRegion} for the illustration.
For heavy-tailed data, the least absolute deviation (LAD) estimator of infinite variance AR models is shown to have a faster convergence rate and a more complicated asymptotic distribution \citep{AN_Chen1982,Davis_Knight_Liu1992}, and the similar situation may be expected to the CQE. 
By adopting the self-weighting method in \cite{Ling2005}, this paper suggests the self-weighted CQE at \eqref{WCQE} to maintain the asymptotic normality even for heavy-tailed data with only a finite fractional moment.
\end{remark}

Denote the true parameter vector by $\bm\theta_{\tau 0}=(\bm \phi_0^{\prime}(\tau),b_0(\tau),\bm \beta_0^{\prime}(\tau))^{\prime}$, and it is assumed to be an interior point of the parameter space $\Theta\subset \mathbb{R}^{2p+1}$, which is a compact set.
Moreover, let $F_{t-1}(\cdot)$ and $f_{t-1}(\cdot)$ be the distribution and density functions of $y_t$ conditional on $\mathcal{F}_{t-1}$, respectively.

\begin{assum}\label{assum2}
	$\{y_t\}$ is strictly stationary and ergodic with $E|y_t|^\kappa<\infty$ for some $0<\kappa\leq 1$.
\end{assum}
\begin{assum}\label{assum3}
	$\{w_{t}\}$ is strictly stationary and ergodic, and $w_{t}$ is nonnegative and measurable with respect to $\mathcal{F}_{t-1}$ with $E[w_{t}\|\bm y_{1,t-1}\|^3]<\infty$.
\end{assum}
\begin{assum}\label{assum4}
	With probability one, $f_{t-1}(\cdot)$ and its derivative function $\dot{f}_{t-1}(\cdot)$ are uniformly bounded, and $f_{t-1}(\cdot)$ is positive on the support $\{x:0<F_{t-1}(x)<1\}$.	
\end{assum}

Theorem \ref{thm0} in Section \ref{Section_process} provides a sufficient condition for Assumption \ref{assum2}, and the proposed self-weighted CQE can handle very heavy-tailed data since only a fractional moment is needed. Some discussions on random weights $\{w_t\}$ are delayed to Remark \ref{weights}.

\begin{remark}
Assumption \ref{assum4} is commonly used in the literature; see, e.g, Assumption 3 of \cite{Ling2005} for the self-weighted LAD estimator of infinite variance AR models, Assumption (A2) of \cite{Lee_Noh2013} for the CQE of GARCH models, and Assumption 4 of \cite{Zhu_Li_Xiao2020QRGARCHX} for the CQE of linear models with GARCH-X errors. 
Specifically, the strong consistency requires the positiveness and continuity of $f_{t-1}(\cdot)$, while the boundedness of $f_{t-1}(\cdot)$ and $\dot{f}_{t-1}(\cdot)$ is needed for the $\sqrt{n}$-consistency and asymptotic normality. 
Assumption \ref{assum4} is mainly to simplify the technical proofs, and certainly can be weakened. For example, we may replace the boundedness of $\dot{f}_{t-1}(\cdot)$ by assuming the uniform continuity of $f_{t-1}(\cdot)$ on a close interval. However, a lengthy technical proof will be required.
\end{remark}

\begin{thm}\label{thm1}
	Suppose that $\min\{|b(\tau)|,|\beta_1(\tau)|,\ldots,|\beta_p(\tau)|\}\geq C$ for a constant $C>0$. 
	If Assumptions \ref{assum2}-\ref{assum4} hold, then $\widehat{\bm\theta}_{\tau n}\rightarrow \bm\theta_{\tau 0}$ almost surely as $n\rightarrow\infty$.
\end{thm}

\begin{remark}\label{add1}
The nonzero restriction on $b(\tau)$ and $\beta_j(\tau)$ with $1\leq j\leq p$ is due to two reasons: (i) the first order derivative of $S_Q(x)$ does not exist at $x=0$, and we need to bound the term of $b(\tau)+\bm y_{2,t-1}^{\prime}\bm \beta(\tau)$ away from zero such that the first derivative of $q_t(\bm\theta_{\tau})$ exists and then the asymptotic variance in Theorem \ref{thm2} is well defined; and (ii) this term will also be used to reduce the moment requirement on $\bm y_{1,t-1}$ or $\bm y_{2,t-1}$ in the technical proofs and, without it, a moment condition on $y_t$ will be required.
For many financial time series, at quantile levels around $\tau=0.5$, the values of $b(\tau)$ and $\beta_j(\tau)$'s are all close to zero, and hence we may fail to provide a reliable estimating result; see Figure \ref{FunctionPlots} for details.
This actually is a common problem for all CQE, and may be solved by assuming a parametric form to these coefficient functions.
We leave it for future research.
\end{remark}

Denote $h_{t}(\bm\theta_{\tau})=b(\tau)+\bm y_{2,t-1}^{\prime}\bm \beta(\tau)$. Let $\dot{q}_t(\bm\theta_{\tau})=(\bm y_{1,t-1}^{\prime},0.5|h_{t}(\bm\theta_{\tau})|^{-1/2},0.5|h_{t}(\bm\theta_{\tau})|^{-1/2}\bm y_{2,t-1}^{\prime})^{\prime}$ be the first derivative of $q_t(\bm\theta_{\tau})$. Define $(2p+1)\times(2p+1)$ symmetric matrices
\[
\Omega_0(\tau)=E\left[w_t^2\dot{q}_t(\bm\theta_{\tau 0})\dot{q}_t^{\prime}(\bm\theta_{\tau 0})\right],
\hspace{5mm}
\Omega_1(\tau)=E\left[f_{t-1}(F_{t-1}^{-1}(\tau))w_{t}\dot{q}_t(\bm\theta_{\tau 0})\dot{q}_t^{\prime}(\bm\theta_{\tau 0})\right]
\]
and $\Sigma(\tau)=\tau(1-\tau)\Omega_1^{-1}(\tau)\Omega_0(\tau)\Omega_1^{-1}(\tau)$.
\begin{thm}\label{thm2}
	Suppose that the conditions of Theorem \ref{thm1} hold. If $\Omega_1(\tau)$ is positive definite, then $\sqrt{n}(\widehat{\bm\theta}_{\tau n}-\bm\theta_{\tau 0})\rightarrow_d N\left(\bm 0,\Sigma(\tau)\right)$
	as $n\rightarrow\infty$.
\end{thm}

The technical proof of Theorem \ref{thm2} is nontrivial since the objective function of the self-weighted CQE is non-convex and non-differentiable, and 
the main difficulty is to prove the $\sqrt{n}$-consistency, i.e. $\sqrt{n}(\widehat{\bm\theta}_{\tau n}-\bm\theta_{\tau 0})=O_p(1)$.
We overcome it by adopting the bracketing method \citep{Pollard1985} as in \cite{Zhu_Ling2011}.

\begin{remark}\label{weights}
For random weights $\{w_t\}$, there are many choices satisfying Assumption \ref{assum3}.
For the double AR model at \eqref{darquantile} and Remark \ref{remark1-special cases}, by a method similar to that of proving Theorem 3 in \cite{Zhu_Zheng_Li2018}, we can show that the asymptotic variance $\Sigma(\tau)$ is minimized when $w_t=|b_0(\tau)+\bm y_{2,t-1}^{\prime}\bm \beta_0(\tau)|^{-1/2}$ and $Ey_t^2<\infty$.
However, for a general quantile double AR, the asymptotic variance depends on the weights in a complicated way, and the optimal weights do not have an explicit form.
In fact, the selection of optimal weights is still an open problem for AR models \citep{Ling2005}, and it becomes more complicated for a quantile model with an additional conditional scale structure in \eqref{qdar}.
Note that the key step of using the weights in technical proofs is to bound the term of $w_ty_{t-j}$ by $O(|y_{t-j}|^{\delta})$ for some fractional value $\delta>0$.
As a result, \cite{Ling2005} heuristically used $w_t=I(a_t=0)+C^3a_t^{-3}I(a_t\neq 0)$, where $a_t=\sum_{i=1}^{p}|y_{t-i}|I(|y_{t-i}|>C)$ for some constant $C>0$. 
Following \cite{Zhu_Zheng_Li2018} and \cite{Jiang2020} for AR-type conditional heteroscedastic models, we simply employ the weights with $w_t=(1+\sum_{i=1}^{p}|y_{t-i}|^3)^{-1}$.
\end{remark}

\begin{remark}
To estimate the quantity of $f_{t-1}(F_{t-1}^{-1}(\tau))$ in the asymptotic variance at Theorem \ref{thm2}, we consider the difference quotient method \citep{Koenker2005},
\begin{align*}
\widehat{f}_{t-1}(F_{t-1}^{-1}(\tau))=\dfrac{2h}{\widehat{Q}_{\tau+h}(y_t|\mathcal{F}_{t-1})-\widehat{Q}_{\tau-h}(y_t|\mathcal{F}_{t-1})},
\end{align*}
for some appropriate choices of bandwidth $h$, where $\widehat{Q}_{\tau}(y_t|\mathcal{F}_{t-1})=q_t(\widehat{\bm\theta}_{\tau n})$ is the fitted $\tau$th conditional quantile. 
As in \cite{Koenker_Xiao2006} and \cite{Li_Li_Tsai2015}, we employ two commonly used choices for the bandwidth $h$ as follows
	\begin{equation}\label{bandwidths}
	h_{B}=n^{-1/5}\left\{\dfrac{4.5f_N^4(F_N^{-1}(\tau))}{[2(F_N^{-1}(\tau))^2+1]^2}\right\}^{1/5}
	\quad\text{and}\quad
	h_{HS}=n^{-1/3}z_{\alpha}^{2/3}\left\{\dfrac{1.5f_N^2(F_N^{-1}(\tau))}{2(F_N^{-1}(\tau))^2+1}\right\}^{1/3},
	\end{equation}
	where $f_N(\cdot)$ and $F_N(\cdot)$ are the standard normal density and distribution functions, respectively, and $z_{\alpha}=F_N^{-1}(1-\alpha/2)$ with $\alpha$ being set to $0.05$; see \cite{Bofinger1975} and \cite{Hall_Sheather1988}.
	The bandwidth $h_{B}$ is selected by minimizing the mean square error of Gaussian density estimation, while $h_{HS}$ is obtained based on the Edgeworth expansion for studentized quantiles. Simulation results in Section \ref{Section_simulation} indicate that these two bandwidths have similar performance.   
As a result, the two matrices $\Omega_0(\tau)$ and $\Omega_1(\tau)$ can then be approximated by the sample averages,
\begin{align*}
\widehat{\Omega}_0(\tau)=\dfrac{1}{n}\sum_{t=p+1}^{n}w_t^2\dot{q}_t(\widehat{\bm\theta}_{\tau n})\dot{q}_t^{\prime}(\widehat{\bm\theta}_{\tau n}) \hspace{2mm}\text{and}\hspace{2mm} \widehat{\Omega}_1(\tau)=\dfrac{1}{n}\sum_{t=p+1}^{n}\widehat{f}_{t-1}(F_{t-1}^{-1}(\tau))w_{t}\dot{q}_t(\widehat{\bm\theta}_{\tau n})\dot{q}_t^{\prime}(\widehat{\bm\theta}_{\tau n}),
\end{align*}
where $\dot{q}_t(\widehat{\bm\theta}_{\tau n})=(\bm y_{1,t-1}^{\prime},0.5|h_{t}(\widehat{\bm\theta}_{\tau n})|^{-1/2},0.5|h_{t}(\widehat{\bm\theta}_{\tau n})|^{-1/2}\bm y_{2,t-1}^{\prime})^{\prime}$. Consequently, a consistent estimator $\widehat{\Sigma}(\tau)$ of the asymptotic variance matrix $\Sigma(\tau)$ can be constructed.
\end{remark}

From the self-weighted CQE $\widehat{\bm\theta}_{\tau n}$, the $\tau$th quantile of $y_t$ conditional on $\mathcal{F}_{t-1}$ can be estimated by $q_{t}(\widehat{\bm\theta}_{\tau n})$.
The following corollary provides the theoretical justification for one-step ahead forecasting, and it is a direct result from Taylor expansion and Theorem \ref{thm2}. 
\begin{cor}\label{cor3}
	Under the conditions of Theorem \ref{thm2} and $E(|y_t|)<\infty$, it holds that
	\begin{align*}
	\widehat{Q}_{\tau}(y_{n+1}|\mathcal{F}_{n})-Q_{\tau}(y_{n+1}|\mathcal{F}_{n})=\dot{q}_{n+1}^{\prime}(\bm\theta_{\tau 0})(\widehat{\bm\theta}_{\tau n}-\bm\theta_{\tau 0})+o_p(n^{-1/2}).
	\end{align*}
\end{cor}

In practice, we may consider multiple quantile levels simultaneously, say $\tau_1<\tau_2<\cdots<\tau_K$. Although $\{\widehat{Q}_{\tau_k}(y_{n+1}|\mathcal{F}_{n})\}_{k=1}^K$ from the proposed procedure may not be monotonically increasing in $k$, it is convenient to employ the rearrangement method in \cite{Chernozhukov2010} to solve the quantile crossing problem after the estimation.

\begin{remark}\label{remark3-order selection}
In real applications, we do not know the order of $p$ at model \eqref{qdar}, and an information criterion is thus introduced.
At each quantile level of $\tau$, we may consider the following Bayesian information criterion (BIC),
\begin{equation*}\label{BICk}
\text{BIC}_{\tau}(p)=2(n-p_{\max})\log(L_n(\widehat{\bm\theta}_{\tau n}^{p})) + (2p+1)\log(n-p_{\max}),
\end{equation*}
where $p_{\max}$ is a predetermined positive integer, $\widehat{\bm\theta}_{\tau n}^{p}$ is the self-weighted CQE defined in \eqref{WCQE} at order $p$, and $L_n(\widehat{\bm\theta}_{\tau n}^p)=(n-p_{\max})^{-1}\sum_{t=p_{\max}+1}^{n}w_t\rho_{\tau}(y_t-q_t(\widehat{\bm\theta}_{\tau n}^p))$ with $w_t=(1+\sum_{i=1}^{p_{\max}}|y_{t-i}|^3)^{-1}$; see also \cite{machado1993robust} and \cite{Zhu_Zheng_Li2018}.
Note that the selected orders will depend on $\tau$, while we may need an order uniformly for all quantile levels. As a result, we suggest a combined version below,
\begin{equation*}\label{combinedBIC}
\text{BIC}(p)= \dfrac{1}{K}\sum_{k=1}^K\text{BIC}_{\tau_k}(p) =2(n-p_{\max}) \dfrac{1}{K}\sum_{k=1}^K \log(L_n(\widehat{\bm\theta}_{\tau n}^{p})) + (2p+1)\log(n-p_{\max}),
\end{equation*} 
where $\tau_k=k/(K+1)$ with $K$ being a fixed integer. Let $\widehat{p}_{n} = \arg\min_{1\leq p\leq p_{\max}}\text{BIC}(p)$ be the selected order. From simulation results in Section \ref{Section_simulation}, the proposed $\text{BIC}$ has satisfactory performance. 
Note that the likelihood function of model \eqref{qdar} or \eqref{qdarprocess} will involve the inverse function of $Q_{\tau}(y_t|\mathcal{F}_{t-1})$ with respect to $\tau$, which, however, has no close form.
It hence is not feasible to use the likelihood function to design an information criterion.
\end{remark}

\begin{remark}\label{remark2-model order}
It is of interest to consider a quantile double autoregression with orders different for the conditional location and scale parts,
\begin{equation*}
Q_{\tau}(y_t|\mathcal{F}_{t-1}) =\sum_{i=1}^{p_1}\phi_{i}(\tau)y_{t-i} +S_{\textrm{Q}}\left(b(\tau)+\sum_{j=1}^{p_2}\beta_{j}(\tau)y_{t-j}^{2}\right), 
\end{equation*}
for all $0<\tau<1$. The counterpart of model \eqref{qdarprocess} can be given similarly.
Let $p=\max\{p_1,p_2\}$ with $\phi_{i}(\cdot)=0$ for $i>p_1$ or $\beta_{j}(\cdot)=0$ for $j>p_2$.  Theorem \ref{thm0} and Corollary \ref{cor1} can then be used to establish the strict stationarity.
For the case with $p_1<p_2$, all theoretical results in this section still hold. However, when $p_1>p_2$, the conditional scale of $S_Q(b(u_t)+\sum_{j=1}^{p_2}\beta_{j}(u_t)y_{t-j}^{2})$ is not enough to reduce the moment requirement on $y_{t}$, and hence a higher order moment of $y_{t}$ will be needed for theoretical justification.
We prefer to the case with $p_1=p_2$, and leave the general setting for future research. 
\end{remark}

\section{Diagnostic checking for conditional quantiles}\label{Section_checking}

To check the adequacy of fitted conditional quantiles, we construct two portmanteau tests to detect possible misspecifications in the conditional location and scale, respectively, and a combined test is also considered.

Let $\eta_{t,\tau}=y_t-Q_{\tau}(y_t|\mathcal{F}_{t-1})=y_t-q_t(\bm\theta_{\tau 0})$ be the conditional quantile error. 
Instead of applying the QACF to quantile errors $\{\eta_{t,\tau}\}$ directly, we alternatively introduce the self-weighted QACF which naturally combines the QACF in \cite{Li_Li_Tsai2015} and the idea of self-weighting in \cite{Ling2005}. 
Specifically, the self-weighted QACF of $\{\eta_{t,\tau}\}$ at lag $k$ is defined as
\[
\rho_{k,\tau} =\textrm{qcor}_{\tau}\{\eta_{t,\tau},\eta_{t-k,\tau}\} =\dfrac{E[w_{t}\psi_{\tau}(\eta_{t,\tau})(\eta_{t-k,\tau}-\mu_{1,\tau})]} {\sqrt{(\tau-\tau^2)\sigma_{1,\tau}^2}},\quad k=1,2,\dots,
\]
where $\psi_{\tau}(x)=\tau-I(x<0)$, $\{w_{t}\}$ are random weights used in Section \ref{Section_estimation}, $\mu_{1,\tau}=E(\eta_{t,\tau})$ and $\sigma_{1,\tau}^2=\textrm{var}(\eta_{t,\tau})$. By replacing $\eta_{t-k,\tau}$ with $|\eta_{t-k,\tau}|$, a variant of $\rho_{k,\tau}$ can be defined as
\[
r_{k,\tau} =\textrm{qcor}_{\tau}\{\eta_{t,\tau},|\eta_{t-k,\tau}|\} =\dfrac{E[w_{t}\psi_{\tau}(\eta_{t,\tau})(|\eta_{t-k,\tau}|-\mu_{2,\tau})]} {\sqrt{(\tau-\tau^2)\sigma_{2,\tau}^2}},\quad k=1,2,\dots,
\]
where $\mu_{2,\tau}=E(|\eta_{t,\tau}|)$ and $\sigma_{2,\tau}^2=\textrm{var}(|\eta_{t,\tau}|)$.
Note that if $Q_{\tau}(y_t|\mathcal{F}_{t-1})$ is correctly specified by model \eqref{qdar}, then $\rho_{k,\tau}=0$ and $r_{k,\tau}=0$ for all $k\geq 1$. 

Accordingly, denote by $\{\widehat{\eta}_{t,\tau}\}$ the conditional quantile residuals, where $\widehat{\eta}_{t,\tau}=y_t-\widehat{Q}_{\tau}(y_t|\mathcal{F}_{t-1})=y_t-q_t(\widehat{\bm\theta}_{\tau n})$. The self-weighted residual QACFs at lag $k$ can then be defined as
\[
\widehat{\rho}_{k,\tau}=\dfrac{1}{\sqrt{(\tau-\tau^2)\widehat{\sigma}_{1,\tau}^2}}\dfrac{1}{n-p}\sum_{t=p+k+1}^n w_{t}\psi_{\tau}(\widehat{\eta}_{t,\tau})(\widehat{\eta}_{t-k,\tau}-\widehat{\mu}_{1,\tau})
\]
and
\[
\widehat{r}_{k,\tau}=\dfrac{1}{\sqrt{(\tau-\tau^2)\widehat{\sigma}_{2,\tau}^2}}\dfrac{1}{n-p}\sum_{t=p+k+1}^n w_{t}\psi_{\tau}(\widehat{\eta}_{t,\tau})(|\widehat{\eta}_{t-k,\tau}|-\widehat{\mu}_{2,\tau}),
\]
where $\widehat{\mu}_{1,\tau}=(n-p)^{-1}\sum_{t=p+1}^n\widehat{\eta}_{t,\tau}$, $\widehat{\mu}_{2,\tau}=(n-p)^{-1}\sum_{t=p+1}^n|\widehat{\eta}_{t,\tau}|$,  $\widehat{\sigma}_{1,\tau}^2=(n-p)^{-1}\sum_{t=p+1}^n(\widehat{\eta}_{t,\tau}-\widehat{\mu}_{1,\tau})^2$ and $\widehat{\sigma}_{2,\tau}^2=(n-p)^{-1}\sum_{t=p+1}^n(|\widehat{\eta}_{t,\tau}|-\widehat{\mu}_{2,\tau})^2$.
For a predetermined positive integer $K$, let $\widehat{\bm\rho}=(\widehat{\rho}_{1,\tau},\ldots,\widehat{\rho}_{K,\tau})^{\prime}$ and $\widehat{\bm r}=(\widehat{r}_{1,\tau},\ldots,\widehat{r}_{K,\tau})^{\prime}$. 
We first derive the asymptotic distribution of $\widehat{\bm\rho}$ and $\widehat{\bm r}$.

Denote $H_{1k}=E[w_{t}f_{t-1}(F_{t-1}^{-1}(\tau))\dot{q}_t(\bm\theta_{\tau 0})\eta_{t-k,\tau}]$ and $H_{2k}=E[w_{t}f_{t-1}(F_{t-1}^{-1}(\tau))\dot{q}_t(\bm\theta_{\tau 0})|\eta_{t-k,\tau}|]$. 
Let $\bm\epsilon_{1,t}=(\eta_{t,\tau},\ldots,\eta_{t-K+1,\tau})^{\prime}$ and   $\bm\epsilon_{2,t}=(|\eta_{t,\tau}|,\ldots,|\eta_{t-K+1,\tau}|)^{\prime}$. 
For $i=1$ and 2, denote the $K\times (2p+1)$ matrices $H_i(\tau)=(H_{i1},\ldots,H_{iK})^{\prime}$ and $M_i(\tau)=E[w_t^2\bm\epsilon_{i,t-1}\dot{q}_t^{\prime}(\bm\theta_{\tau 0})]$, and the $K\times K$ matrices $\Psi_i(\tau)=E(w_{t}^2\bm\epsilon_{i,t-1}\bm\epsilon_{i,t-1}^{\prime})$ and
\begin{equation*}\label{SeparateTestCov}
\Pi_i(\tau)=\sigma_{i,\tau}^{-2}[\Psi_i(\tau)+H_i(\tau)\Xi(\tau) H_i^{\prime}(\tau)-M_i(\tau)\Omega_1^{-1}(\tau)H_i^{\prime}(\tau)-H_i(\tau) \Omega_1^{-1}(\tau)M_i^{\prime}(\tau)],
\end{equation*}
where $\Xi(\tau)=\Omega_1^{-1}(\tau)\Omega_0(\tau)\Omega_1^{-1}(\tau)$. 
In addition, denote the $2K\times (2p+1)$ matrices $H(\tau)=(\sigma_{1,\tau}^{-1}H_1^{\prime}(\tau),\sigma_{2,\tau}^{-1}H_2^{\prime}(\tau))^{\prime}$ and $M(\tau)=E[w_t^2\bm\epsilon_{t-1}\dot{q}_t^{\prime}(\bm\theta_{\tau 0})]$, where $\bm\epsilon_{t}=(\sigma_{1,\tau}^{-1}\bm\epsilon_{1,t}^{\prime},\sigma_{2,\tau}^{-1}\bm\epsilon_{2,t}^{\prime})^{\prime}$. Define the $2K\times 2K$ matrices $\Psi(\tau)=E(w_{t}^2\bm\epsilon_{t-1}\bm\epsilon_{t-1}^{\prime})$ and 
	\begin{equation*}\label{JointTestCov}
	\Pi(\tau)= \Psi(\tau)+H(\tau)\Xi(\tau) H^{\prime}(\tau)-M(\tau)\Omega_1^{-1}(\tau)H^{\prime}(\tau)-H(\tau) \Omega_1^{-1}(\tau)M^{\prime}(\tau) :=\left(\begin{array}{cc}
	\Pi_1(\tau) & \Pi_3^{\prime}(\tau)\\
	\Pi_3(\tau) & \Pi_2(\tau)
	\end{array}\right).
	\end{equation*} 

\begin{thm}\label{thm3}
	Suppose that the conditions of Theorem \ref{thm2} hold and $E(y_t^2)<\infty$. It holds that $\sqrt{n}(\widehat{\bm \rho}^{\prime},\widehat{\bm r}^{\prime})^{\prime} \to_d N(0,\Pi(\tau))$ as $n\rightarrow\infty$.
\end{thm} 

The finite second moment of $y_t$ is required in Theorem \ref{thm3} since this condition is needed to show the consistency of $\widehat{\sigma}_{i,\tau}^2$ for $i=1$ and 2. We have tried many other approaches, such as transforming the residuals by a bounded and strictly increasing function \citep{Zhu_Zheng_Li2018} and then applying the self-weighted QACF to the transformed residuals, however, the condition of $E(y_t^2)<\infty$ is still unavoidable.  
	As a comparison, if we alternatively use the unweighted QACFs of residuals with $w_t=1$ for all $t$, then the condition of $E(|y_t|^3)<\infty$ is unavoidable for the asymptotic normality. 

As for estimating $\Sigma(\tau)$ in Section \ref{Section_estimation}, we can employ the difference quotient method to estimate the quantity of $f_{t-1}(F_{t-1}^{-1}(\tau))$, and then approximate the expectations in  $\Pi(\tau)$ by sample averages with $\eta_{t,\tau}$ being replaced by $\widehat{\eta}_{t,\tau}$. Consequently, a consistent estimator, denoted by $\widehat{\Pi}(\tau)$, for the asymptotic covariance in Theorem \ref{thm3} can be constructed. 
We then can check the significance of $\widehat{\rho}_{k,\tau}$'s and $\widehat{r}_{k,\tau}$'s individually by establishing their confidence intervals based on $\widehat{\Pi}(\tau)$.

Let $\bm z=(\bm z_1^{\prime},\bm z_2^{\prime})^{\prime}\in \mathbb{R}^{2K}$ be a multivariate normal random vector with zero mean vector and covariance matrix $\Pi(\tau)$, where $\bm z_1, \bm z_2\in \mathbb{R}^{K}$.
To check the first $K$ lags of $\widehat{\rho}_{k,\tau}$'s (or $\widehat{r}_{k,\tau}$'s) jointly, by Theorem \ref{thm3}, the Box-Pierce type test statistics can be designed below,
\[
Q_1(K)=n\sum_{k=1}^K\widehat{\rho}_{k,\tau}^2 \to_d \bm z_1^{\prime}\bm z_1 \;\;\text{and}\;\;
Q_2(K)=n\sum_{k=1}^K\widehat{r}_{k,\tau}^2 \to_d \bm z_2^{\prime}\bm z_2 ,
\]
as $n\to\infty$.
It is also of interest to consider a combined test, $Q(K)=Q_1(K)+Q_2(K) \to_d \bm z^{\prime}\bm z$ as $n\to\infty$.
To calculate the critical value or $p$-value of the three tests, we generate a sequence of, say $B=10000$, multivariate random vectors with the same distribution of $\bm z$, and then use the empirical distributions to approximate the corresponding null distributions.
As expected, the simulation results in Section \ref{Section_simulation} suggest that $Q_1(K)$ is more powerful in detecting the misspecification in the conditional location, $Q_2(K)$ has a better performance in detecting the misspecification in the conditional scale, while $Q(K)$ has a performance between those of $Q_1(K)$ and $Q_2(K)$. 
Therefore, we may first use $Q(K)$ to check the overall adequacy of the fitted conditional quantiles, and then employ $Q_1(K)$ and $Q_2(K)$ to look for more details; see also \cite{Li_Li2008b} and \cite{Zhu_Zheng_Li2018}. 

\begin{remark}
There are two advantages to introduce the self-weights $\{w_{t}\}$ into the QACF. First, in line with the estimating procedure in Section \ref{Section_estimation}, the self-weights can reduce the moment restriction on $y_t$ in establishing the asymptotic normality of $\widehat{\rho}_{k,\tau}$ and $\widehat{r}_{k,\tau}$ (see Theorem \ref{thm3}) and thus makes the diagnostic tools applicable for heavy-tailed data. Secondly, as validated by the simulation results in Section \ref{Section_simulation}, the self-weighted residual QACFs $\widehat{\rho}_{k,\tau}$ and $\widehat{r}_{k,\tau}$ are more efficient than their unweighted counterparts with $w_t=1$ for all $t$. In contrast to the efficiency loss in estimation due to the self-weights, interestingly, self-weights can improve the efficiency of QACFs in terms of standard deviations.
\end{remark}

\section{Simulation studies}\label{Section_simulation}

\subsection{Self-weighted CQE and model selection}

The first experiment is to evaluate the self-weighted CQE $\widehat{\bm\theta}_{\tau n}$ in Section \ref{Section_estimation}. The data generating process is
\begin{align}\label{sim1}
y_t=& \phi(u_t)y_{t-1}+S_{\textrm{Q}}(b(u_t)+\beta(u_t) y_{t-1}^2),
\end{align}
where $\{u_t\}$ are $i.i.d.$ standard uniform random variables, and we consider two sets of coefficient functions below, 
\begin{align} \label{sim1coef1}
\phi(\tau)=-0.2, \; b(\tau)=S_{\textrm{Q}}^{-1}(F_b^{-1}(\tau)) \hspace{2mm}\text{and}\hspace{2mm} \beta(\tau)=0.4b(\tau)
\end{align}
and
\begin{align}\label{sim1coef2}
\phi(\tau)=0.5\tau, \; b(\tau)=S_{\textrm{Q}}^{-1}(F_b^{-1}(\tau)) \hspace{2mm}\text{and}\hspace{2mm} \beta(\tau)=0.5\tau b(\tau),
\end{align}
where $S_{\textrm{Q}}^{-1}(\tau)=\tau^2\sgn(\tau)$ is the inverse function of $S_{\textrm{Q}}(\cdot)$, and $F_b(\cdot)$ is the distribution function of the standard normal, the Student's $t_5$ or the Student's $t_3$ random variable.
It holds that $S_{\textrm{Q}}(b(u_t)+\beta(u_t) y_{t-1}^2)=S_{\textrm{Q}}(b(u_t))\sqrt{1+ cy_{t-1}^2}$ when $\beta(\tau)=cb(\tau)$ with $c>0$ being a constant. Let $\varepsilon_t=S_{\textrm{Q}}(b(u_t))$, and then the coefficient functions at \eqref{sim1coef1} and \eqref{sim1coef2} correspond to
\[
y_t= -0.2y_{t-1}+\varepsilon_t \sqrt{1+ 0.4y_{t-1}^2} \hspace{5mm}\text{and}\hspace{5mm} y_t= 0.5u_ty_{t-1}+\varepsilon_t \sqrt{1+ 0.5u_ty_{t-1}^2},
\]
respectively.
We consider two sample sizes, $n=500$ and 1000, and there are 1000 replications for each sample size. 
The self-weighted CQE in \eqref{WCQE} is employed to fit the data, and the quasi-Newton algorithm can be used to obtain the estimate for each replication. 
To estimate the asymptotic standard deviation (ASD) of $\widehat{\bm\theta}_{\tau n}$, two bandwidths, $h_B$ and $h_{HS}$ at \eqref{bandwidths}, are used to estimate $f_{t-1}(F_{t-1}^{-1}(\tau))$ in the difference quotient method, and the resulting ASDs are denoted as ASD$_1$ and ASD$_2$, respectively.

Tables \ref{table1a} and \ref{table1b} present the bias, empirical standard deviation (ESD) and ASD of $\widehat{\bm\theta}_{\tau n}$ at quantile level $\tau=0.05$ or $0.25$ for settings \eqref{sim1coef1} and \eqref{sim1coef2}, respectively, and the corresponding values of $\bm\theta_{\tau 0}$ are also given. 
We only report the results for the standard normal and the Student's $t_5$ cases, and that for the Student's $t_3$ case are relegated to the supplementary material to save space. 
It can be seen that, as the sample size increases, the biases, ESDs and ASDs decrease, and the ESDs get closer to their corresponding ASDs.
Moreover, the biases, ESDs and ASDs get smaller as the quantile level increases from $\tau=0.05$ to $\tau=0.25$. It is obvious since there are more observations as the quantile level gets near to the center, although the true values of $b(\tau)$ and $\beta(\tau)$ also become smaller as $\tau$ approaches $0.5$.
Finally, most biases, ESDs and ASDs increase as the distribution of $F_b(\cdot)$ gets heavy-tailed, and the bandwidth $h_{HS}$ slightly outperforms $h_{B}$.

The second experiment is to evaluate the proposed BIC in Remark \ref{remark3-order selection}. The data generating process is
\begin{align}\label{sim3}
y_t=& \phi_1(u_t)y_{t-1}+\phi_2(u_t)y_{t-2}+S_{\textrm{Q}}(b(u_t)+\beta_1(u_t) y_{t-1}^2+\beta_2(u_t) y_{t-2}^2)
\end{align}
with three sets of coefficient functions,
\begin{align*} 
(i)&~\phi_1(\tau)=0.1, \; \phi_2(\tau)=0.3, \; b(\tau)=S_{\textrm{Q}}^{-1}(F_b^{-1}(\tau)), \; \beta_1(\tau)=0.1b(\tau) \;\text{and}\; \beta_2(\tau)=0.4b(\tau), \\
(ii)&~\phi_1(\tau)=0.1\tau, \; \phi_2(\tau)=0.3, \; b(\tau)=S_{\textrm{Q}}^{-1}(F_b^{-1}(\tau)), \; \beta_1(\tau)=0.1\tau b(\tau) \;\text{and}\; \beta_2(\tau)=0.4b(\tau), \\
(iii)&~\phi_1(\tau)=0.1\tau, \; \phi_2(\tau)=0.3, \; b(\tau)=S_{\textrm{Q}}^{-1}(F_b^{-1}(\tau)), \; \beta_1(\tau)=0.1\tau b(\tau) \;\text{and}\; \beta_2(\tau)=0.4\tau b(\tau),
\end{align*}
where the other settings are preserved as in the first experiment. 
The proposed BIC is used to select the order $p$ with $p_{\max}=5$. Since the true order is two, the cases of underfitting, correct selection and overfitting correspond to $\widehat{p}_{n}$ being 1, 2 and greater than 2, respectively. 
Table \ref{table-bic} gives the percentages of underfitted, correctly selected and overfitted models. It can be seen that the BIC performs well in general. It becomes better when the sample size increases, while  gets slightly worse as the distribution of $F_b(\cdot)$ gets more heavy-tailed.

\subsection{Three portmanteau tests}

The third experiment considers the self-weighted residual QACFs, $\widehat{\rho}_{k,\tau}$ and $\widehat{r}_{k,\tau}$, and the approximation of their asymptotic distributions.
The number of lags is set to $K=6$, and all the other settings are the same as in the first experiment.
For model \eqref{sim1} with coefficient functions at \eqref{sim1coef1}, Tables \ref{table2a} and \ref{table2c} give the biases, ESDs and ASDs of $\widehat{\rho}_{k,\tau}$ and $\widehat{r}_{k,\tau}$, respectively, with lags $k=2, 4$ and 6. 
The results for Student's $t_3$ case are relegated to the supplementary material to save space.  
We have four findings for both residual QACFs $\widehat{\rho}_{k,\tau}$ and $\widehat{r}_{k,\tau}$: (1) the biases, ESDs and ASDs decrease as the sample size increases or as the quantile level increases from $\tau=0.05$ to $0.25$; (2) the ESDs and ASDs are very close to each other; (3) the biases, ESDs and ASDs get smaller as $F_b(\cdot)$ becomes heavy-tailed; (4) two bandwidths $h_{HS}$ and $h_{B}$ perform similarly.
Similar observations can be found for coefficient functions at \eqref{sim1coef2}, and the results are relegated to the supplmentary material to save space.

The fourth experiment is to evaluate the size and power of three portmanteau tests in Section \ref{Section_checking}, and the data generating process is
\begin{align}\label{sim2coef1}
y_t=& c_1y_{t-2}+S_{\textrm{Q}}\left(b(u_t)+0.1b(u_t)y_{t-1}^2+c_2b(u_t)y_{t-2}^2\right)
\end{align}
with $b(\cdot)$ being defined as in previous experiments, while a quantile double AR model with order one is fitted to the generated sequences. As a result, the case of $c_1=c_2=0$ corresponds to the size, the case of $c_1\neq 0$ to the misspecification in the conditional location, and the case of $c_2\neq 0$ to the misspecification in the conditional scale.
Two departure levels, 0.1 and 0.3, are considered for both $c_1$ and $c_2$, and we calculate the critical values by generating $B=10000$ random vectors.
Table \ref{table3a} gives the rejection rates of $Q_{1}(6)$, $Q_{2}(6)$ and $Q(6)$, where the critical values of three tests are calculated based on estimated covariance matrix $\widehat{\Pi}(\tau)$ using the bandwidth $h_{HS}$. 
It can be seen that the size gets closer to the nominal rate as the sample size $n$ increases to 1000 or the quantile level $\tau$ increases to 0.25, and almost all powers increase as the sample size or departure level increases; 
Moreover, in general, $Q_{1}(K)$ is more powerful than $Q_{2}(K)$ in detecting the misspecification in the conditional location, $Q_{2}(K)$ is more powerful in detecting the misspecification in the conditional scale, and $Q(K)$ compromises between $Q_{1}(K)$ and $Q_{2}(K)$. 
Finally, when the data are more heavy-tailed, $Q_{1}(K)$ becomes less powerful in detecting the misspecification in the conditional location, while $Q_{2}(K)$ is more powerful in detecting the misspecification in the conditional scale.
This may be due to the mixture of two effects: the worse performance of the estimation and the larger value of $|c_2b(\tau)|$ in the conditional scale.
We also calculate the rejection rates of three tests using the bandwidth $h_{B}$ and another data generating process, and similar findings can be observed; see the supplmentary material for details.

\subsection{Random weights}

The last experiment is to compare the proposed inference tools with and without random weights in term of efficiency.
Note that $\widehat{\bm\theta}_{\tau n}$, $\widehat{\rho}_{k,\tau}$ and $\widehat{r}_{k,\tau}$ are the self-weighted CQE and two self-weighted QACFs, respectively, and we denote their unweighted counterparts by $\widetilde{\bm\theta}_{\tau n}$, $\widetilde{\rho}_{k,\tau}$ and $\widetilde{r}_{k,\tau}$.

The data are generated from the process \eqref{sim1}, and two sets of coefficient functions are considered: 
\begin{align*}
&(a)~\phi(\tau)=0.5\tau,\; b(\tau)=S_Q^{-1}(F_b^{-1}(\tau)),\; \beta(\tau)=0.5\tau b(\tau), \\
&(b)~\phi(\tau)=0.5\tau, \; b(\tau)=S_Q^{-1}(F_b^{-1}(\tau)),\; \beta(\tau)=0.8(\tau-0.5),
\end{align*}  
where $F_b(\cdot)$ is defined in the first experiment, and Set (a) is the same as in \eqref{sim1coef2}.
We generate data of size 2000 with 1000 replications, and fit a quantile double AR model with order one to each replication using weighted and unweighted CQE at quantile levels $\tau=0.05$ and 0.25. 
Figure \ref{BoxPlotEstimator} presents the box plots for the weighted and unweighted CQE estimators, and Figures \ref{BoxPlotRho} and \ref{BoxPlotR} give the box plots for residual QACFs of weighted and unweighted versions. 

On one hand, the interquartile range of the weighted estimator $\widehat{\bm\theta}_{\tau n}$ is larger than that of the unweighted counterpart $\widetilde{\bm\theta}_{\tau n}$, 
however, the efficiency loss due to the random weights seems smaller as $\tau$ increases from 0.05 to 0.25 or the distribution of $F_b(\tau)$ gets less heavy-tailed.   
On the other hand, the interquartile range of the weighted residual QACFs $\widehat{\rho}_{k,\tau}$ (or $\widehat{r}_{k,\tau}$) is smaller than that of the unweighted counterpart $\widetilde{\rho}_{k,\tau}$ (or $\widetilde{r}_{k,\tau}$), and the efficiency gain of residual QACFs owing to the random weights gets larger as the distribution of $F_b(\tau)$ becomes more heavy-tailed.   
In other words, the random weights play the opposite role for estimators and residual QACFs, i.e., the weighted estimator is less efficient while the weighted residual QACFs are more efficient in finite samples. 
The case with the sample size of 1000 is similar, and its results can be found in the supplementary material.

In sum, both the self-weighted CQE and diagnostic tools can be used to handle the heavy-tailed time series by introducing self-weights at the sacrifice of efficiency in estimation, but the weights can lead to efficiency gain for residual QACFs. In terms of portmanteau tests, the combined test $Q(K)$ is preferred to check the overall adequacy of fitted conditional quantiles, while the tests $Q_1(K)$ and $Q_2(K)$ can be used to detect the possible misspecification in the conditional location and scale, respectively.   
For the selection of bandwidth in estimating the quantity of $f_{t-1}(F_{t-1}^{-1}(\tau))$, 
although $h_{HS}$ and $h_{B}$ perform very similarly in diagnostic checking, we recommend $h_{HS}$ since it has more stable performance in estimation. Therefore, $h_{HS}$ is used to estimate the covariance matrices of the self-weighted estimator and residual QACFs in the next section.

\section{An empirical example}\label{Section_realdata}

This section analyzes the weekly closing prices of S\&P500 Index from January 10, 1997 to December 30, 2016. Figure \ref{TimePlotVaR} gives the time plot of log returns in percentage, denoted by $\{y_t\}$, and there are 1043 observations in total. The summary statistics for $\{y_t\}$ are listed in Table \ref{RealDateTab1}, and it can be seen that the data is negatively skewed and heavy-tailed.
\begin{table}[htb]
	\begin{center}
		\caption{\label{RealDateTab1}Summary statistics for log returns}\vspace{1mm}	
		\begin{tabular}{ccccccc}
			\hline\hline
			Min & Max & Mean & Median & Std. Dev. & Skewness & Kurtosis \\
			\hline	
			-20.189 & 11.251 & 0.000 & 0.097 & 2.488 & -0.746 & 6.119 \\
			\hline
		\end{tabular}	
	\end{center}
\end{table}

The autocorrelation can be found from the sample autocorrelation functions (ACFs) of both $\{y_t\}$ and $\{y_t^2\}$. We then consider a double AR model,
\begin{align}\label{fittedDAR}
y_t=& -0.091_{0.040}y_{t-1}+0.051_{0.036}y_{t-2}-0.013_{0.037}y_{t-3} \nonumber \\
&+\varepsilon_t\sqrt{2.583_{0.193}+0.255_{0.031}y_{t-1}^2+0.132_{0.039}y_{t-2}^2+0.197_{0.035}y_{t-3}^2},
\end{align}
where the Gaussian QMLE is employed, standard errors are given in the corresponding subscripts, and the order is selected by the Bayesian information criterion (BIC) with $p_{\textrm{max}}=10$.
It can be seen that, at the 5\% significance level, all fitted coefficients in the conditional mean are insignificant or marginally significant, while those in the conditional variance are significant.

We apply the quantile double AR model to the sequence, and the quantile levels are set to $\tau_k=k/20$ with $1\leq k\leq 19$.
The proposed BIC in Section \ref{Section_estimation} is employed with $p_{\textrm{max}}=10$, and the selected order is $\widehat{p}_n=3$.
The estimates of $\phi_i(\tau)$ for $1\leq i\leq 3$, together with their 95\% confidence bands, are plotted against the quantile level in Figure \ref{IntroFunction}, and those of $\phi_i$ in the fitted double AR model \eqref{fittedDAR} are also given for the sake of comparison.
The confidence bands of $\phi_i(\tau)$ and $\phi_i$ at lags 1 and 2 are not overlapped at the quantile levels around $\tau=0.8$, while those at lag 3 are significantly separated from each other around $\tau=0.2$.
We may conclude the $\tau$-dependence of $\phi_i(\tau)$'s, and the commonly used double AR model is limited in interpreting such type of financial time series.

We next attempt to compare the fitted coefficients in the conditional scale of our model with those of model \eqref{fittedDAR}.
Note that, from Remark \ref{remark1-special cases}, the quantity of $\beta_j(\tau)/b(\tau)$ for each $1\leq j\leq p$ in the quantile double AR model corresponds to $\beta_j/\omega$ in the double AR model.
Moreover, when the quantile level $\tau$ is near to 0.5, the estimate of $b(\tau)$ is very small, and this makes the value of $\widehat{\beta}_j(\tau)/\widehat{b}(\tau)$ abnormally large; see also Remark \ref{add1}.
As a result, Figure \ref{FunctionPlots} plots the fitted values of ${\beta}_j(\tau)/{b}(\tau)$ for $1\leq j\leq 3$, together with their 95\% confidence bands, against the quantile level with $|\tau-0.5|\geq 0.15$, and the fitted values of $\beta_j/\omega$ from model \eqref{fittedDAR} are also reported.
The confidence bands of $\beta_1(\tau)/b(\tau)$ and $\beta_1/\omega$ are separated from each other for the quantile levels around $\tau=0.25$.
Note that, for the double AR model \eqref{darquantile}, the coefficients in the conditional scale include $S_Q(b_{\tau}\beta_j)$ with $1\leq j\leq p$, which already depend on the quantile level.
We may argue that it is still not flexible enough in interpreting this series.

Since 5\% VaR is usually of interest for the practitioner, we give more details about the fitted conditional quantile at $\tau=0.05$,
\begin{align*}
\widehat{Q}_{0.05}(y_t|\mathcal{F}_{t-1}) =& 0.091_{0.198}y_{t-1}+0.379_{0.135}y_{t-2}+0.260_{0.139}y_{t-3} \nonumber \\
&+S_{\textrm{Q}}\left(-6.951_{1.773}-0.261_{0.698}y_{t-1}^2-0.367_{0.413}y_{t-2}^2-1.346_{0.501}y_{t-3}^2\right),
\end{align*}
where standard errors are given in the corresponding subscripts of the estimated coefficients.
Figure \ref{RhoRCI} plots the residual QACFs $\widehat{\rho}_{k,\tau}$ and $\widehat{r}_{k,\tau}$, and they slightly stand out the 95\% confidence bands only at lags 1 and 4.
The $p$-values of $Q_1(K)$, $Q_2(K)$ and $Q(K)$ are all larger than 0.717 for $K=10$, 20 and 30.
We may conclude that both residual QACFs are insignificant both individually and jointly, and the fitted conditional quantile is adequate.

We consider the one-step ahead conditional quantile prediction at level $\tau=0.05$, which is the negative values of 5\% VaR forecast, and a rolling forecasting procedure is performed.
Specifically, we begin with the forecast origin $n=501$, which corresponds to the date of August 11, 2006, and obtain the estimated coefficients of the quantile double AR model with order three and using the data from the beginning to the forecast origin (exclusive). For each fitted model, we calculate the one-step ahead conditional quantile prediction for the next trading week by  $\widehat{Q}_{0.05}(y_{n}|\mathcal{F}_{n-1})=\sum_{i=1}^3\widehat{\phi}_i(0.05)y_{n-i} +S_{\textrm{Q}}\left(\widehat{b}(0.05)+\sum_{j=1}^3\widehat{\beta}_j(0.05)y_{n-j}^2\right)$.
Then we advance the forecast origin by one and repeat the previous estimation and prediction procedure until all data are utilized.
These predicted values are displayed against the time plot in Figure \ref{TimePlotVaR}.
The magnitudes of VaRs become larger as the return becomes more volatile, and the returns fall below their calculated negative 5\% VaRs occasionally.

In addition, we compare the forecasting performance of the proposed quantile double AR (QDAR) model with two commonly used models in the literature: the quantile AR (QAR) model in \cite{Koenker_Xiao2006} and the two-regime threshold quantile AR (TQAR) model in \citep{galvao2011threshold}.
The order is fixed at three for all models. In estimating TQAR models, the delay parameter $d$ is searched among $\{1,2,3\}$, and the threshold parameter $r$ is searched among a compact grid with the empirical percentiles of $y_t$ from the $10$th quantile to the $90$th quantile.
The rolling forecasting procedure is employed again, and we consider VaRs at four levels, $\tau=5\%$, 10\%, 90\% and 95\%.
To evaluate the forecasting performance, the empirical coverage rate (ECR) is calculated as the percentage of observations that fall below the corresponding conditional quantile forecast for the last 543 data points.  
Besides the ECR, we also conducts two VaR backtests: the likelihood ratio test for correct conditional coverage (CC) in \cite{christoffersen1998evaluating} and the dynamic quantile (DQ) test in \cite{Engle2004}.
Specifically, the null hypothesis of CC test is that, conditional on $\mathcal{F}_{t-1}$, $\{H_t\}$ are $i.i.d.$ Bernoulli random variables with success probability being $\tau$, where the hit $H_t=I(y_t<Q_{\tau}(y_t | \mathcal{F}_{t-1}))$.
For the DQ test, following \cite{Engle2004}, we regress $H_t$ on a constant, four lagged hits $H_{t-\ell},\ell=1,2,3,4$, and the contemporaneous VaR forecast. The null hypothesis of DQ test is that the intercept is equal to the quantile level $\tau$ and four regression coefficients are zero. If the null hypothesis of each VaR backtest cannot be rejected, it then indicates that the VaR forecasts are satisfactory.

Table \ref{table_real_data_forecast} gives ECRs and $p$-values of two VaR backtests for the one-step-ahead forecasts by three fitted models. 
It can be seen that the ECRs of three models are all close to the corresponding nominal levels.
On the other hand, in terms of two backtests, the QDAR model performs satisfactorily at four quantile levels with all $p$-values greater than $0.1$. However, the QAR model performs poorly at all levels, and the TQAR model only performs well at the 90\% quantile level. 
This may be due to the fact that both QAR and TQAR model ignore the conditional heteroscedastic structure in stock prices.

In sum, it is necessary to consider the quantile double AR model to interpret these stock indices, and the proposed inference tools can also provide reliable results.

\section{Conclusion and discussion}\label{Section_conclusion}

This paper proposes a new conditional heteroscedastic model, which has varying structures at different quantile levels, and its necessity is illustrated by analyzing the weekly S\&P500 Index.
A simple but ingenious transformation $S_Q(\cdot)$ is introduced to the conditional scale, and this makes the AR coefficients free from nonnegative constraints.
The strict stationarity of the new model is derived, and inference tools, including a self-weighted CQE and self-weighted residual QACFs-based portmanteau tests, are also constructed. 

Our model can be extended in two directions. First, as far as we know, this is the first quantile conditional heteroscedastic model in the literature, and it is certainly of interest to consider other types of quantile conditional heteroscedastic time series models \citep{Francq_Zakoian2010}.
Moreover, based on the transformation of $S_Q(\cdot)$, the proposed quantile double AR model \eqref{qdarprocess} has a seemingly linear structure. As a result, we may consider a multivariate quantile double AR model, say with order one,
\[
y_{lt}=\sum_{m=1}^N\phi_{lm}(u_{lt})y_{m,t-1} + S_{\textrm{Q}}\left(b_l(u_{lt})+\sum_{m=1}^N\beta_{lm}(u_{lt})y_{m,t-1}^{2}\right), \hspace{5mm}1\leq l\leq N,
\]
where $\bm{y}_t=(y_{1t},...,y_{Nt})^{\prime}$ is an $N$-dimensional time series, $\bm{u}_t=(u_{1t},...,u_{Nt})^{\prime}$, $\bm\phi(\bm{u}_t)=(\phi_{lm}(u_{lt}))$ and $\bm\beta(\bm{u}_t)=(\beta_{lm}(u_{lt}))$ are $N\times N$ coefficient matrices, $\bm{b}(\bm{u}_t)=(b_1(u_{1t}),...,b_N(u_{Nt}))^{\prime}$, and the marginal distributions of $\bm{u}_t$ are all standard uniform; see \cite{Tsay2014}.
It may be even possible to be able to handle high-dimensional time series, when the coefficient matrices $\bm\phi(\bm{u}_t)$ and $\bm\beta(\bm{u}_t)$ are sparsed or have a low-rank structure \citep{zhu2017network}.
We leave it for future research.

\section*{Supplementary Material}
The supplementary material provides additional simulation results. 

\section*{Acknowledgment}
We are grateful to the co-editor, associate editor and three anonymous referees for their valuable comments and suggestions, which led to substantial improvement of this article.
Zhu's research was supported by a NSFC grant 12001355, Shanghai Pujiang Program 2019PJC051 and Shanghai Chenguang Program 19CG44. Li's research was partially supported by a Hong Kong RGC grant 17304617 and a NSFC grant 72033002.

\renewcommand{\thesection}{A}
\setcounter{equation}{0} 
\section*{Appendix: Technical proofs}

This appendix gives the technical proofs of Theorems \ref{thm0}-\ref{thm3} and Corollary \ref{cor3}. Moreover, Lemmas \ref{lem1}-\ref{lem2} with proofs are also included, and they provide some preliminary results for proving Theorem \ref{thm2}. 
Throughout the appendix, the notation $C$ is a generic constant which may take different values from lines to lines.
The norm of a matrix or column vector is defined as $\|A\|=\sqrt{\text{tr}(AA^\prime)}=\sqrt{\sum_{i,j}|a_{ij}|^2}$.

\begin{proof}[Proof of Theorem \ref{thm0}]	
	Let $\bm y_t=(y_t, y_{t-1}, \ldots, y_{t-p+1})^{\prime}$, $\mathcal{B}^p$ be the class of Borel sets of $\mathbb{R}^p$, and $\nu_p$ be the Lebesgue measure on $(\mathbb{R}^p, \mathcal{B}^p)$.
	Denote by $m: \mathbb{R}^p\rightarrow \mathbb{R}$ the projection map onto the first coordinate, i.e. $m(x)=x_1$ for $\bm x=(x_1, x_2, \ldots, x_p)^{\prime}$.
	Define the function $G(u; \bm v_1,\bm v_2)=\bm v_1^{\prime}\bm \phi(u)+S_{\textrm{Q}}\left(b(u)+\bm v_2^{\prime}\bm \beta(u)\right)$ for $u\in (0,1)$ and vectors $\bm v_1,\bm v_2 \in \mathbb{R}^p$, where $b(\cdot)$, $\bm \phi(\cdot)$ and $\bm \beta(\cdot)$ are defined as in Section \ref{Section_estimation}.
	As a result, $\{\bm y_t\}$ is a homogeneous Markov chain on the state space $(\mathbb{R}^p, \mathcal{B}^p,\nu_p)$, with transition probability
	\[
	P(\bm x,A)=\int_{m(A)}2|b(\widetilde{u})+\bm x_{\textrm{Q}}^{\prime}\bm \beta(\widetilde{u})|^{1/2}f\left(S_{\textrm{Q}}^{-1}\left(z-\bm x^{\prime}\bm \phi(\widetilde{u})\right)-\bm x_{\textrm{Q}}^{\prime}\bm \beta(\widetilde{u})\right)dz
	\]
	for $\bm x \in \mathbb{R}^p$ and $A \in \mathcal{B}^p$, where $\bm x_{\textrm{Q}}=(x_1^2, x_2^2, \ldots, x_p^2)^{\prime}$, $\widetilde{u}=G^{-1}(z; \bm x,\bm x_{\textrm{Q}})$ with $G^{-1}$ being the inverse function of $G(u; \bm x,\bm x_{\textrm{Q}})$, $f$ is the density function of $b(u_t)$, and $S_{\textrm{Q}}^{-1}(x)=x^2\sgn(x)$ is the inverse function of $S_{\textrm{Q}}(x)$.
	
	Let $\bm X_{1,i}=(z_i, \ldots, z_1, x_1, \ldots, x_{p-i})^{\prime}$ and $\bm X_{2,i}=(z_i^2, \ldots, z_1^2, x_1^2, \ldots, x_{p-i}^2)^{\prime}$. We can further show that the $p$-step transition probability of the Markov chain $\{\bm y_t\}$ is
	\begin{align}\label{ptran}
	P^p(\bm x, A)=&\int_{A}\prod_{i=1}^p 2|b(\widetilde{u}_i)+\bm X_{2,i-1}^{\prime}\bm \beta(\widetilde{u}_i)|^{1/2} \nonumber \\
	          & \hspace{10mm}\times f\left(S_{\textrm{Q}}^{-1}\left(z_i-\bm X_{1,i-1}^{\prime}\bm \phi(\widetilde{u}_i)\right)-\bm X_{2,i-1}^{\prime}\bm \beta(\widetilde{u}_i)\right)dz_1\ldots dz_p,
	\end{align}
	where $\widetilde{u}_i=G^{-1}(z_i; \bm X_{1,i-1}, \bm X_{2,i-1})$.  
	Observe that, by Assumption \ref{assum0}, 
	$P^p(\bm x, A)>0$ for any $\bm x\in \mathbb{R}^p$ and $A\in\mathcal{B}^p$ with $\nu_p(A)>0$, i.e., $\{\bm y_t\}$ is $\nu_p$-irreducible.	
	
	To show that $\{\bm y_{t}\}$ is geometrically ergodic, next we verify the Tweedie's drift criterion \citep[Theorem 4]{Tweedie1983}.
	Note that $(a+b)^\kappa\leq a^\kappa+b^\kappa$ for $a,b>0$ and $0<\kappa \leq 1$.
	Moreover, since $-|c|-|d|\leq c+d \leq |c|+|d|$ for any constants $c$ and $d$. Then for any $u\in (0,1)$, it follows that
	\[
	-\sqrt{|b(u)|}-\sum_{i=1}^p\sqrt{|\beta_{i}(u)|}|x_i| \leq \left|b(u)+\sum_{i=1}^p\beta_i(u)x_i^2\right|^{1/2} \leq \sqrt{|b(u)|}+\sum_{i=1}^p\sqrt{|\beta_{i}(u)|}|x_i|.
	\]
	 As a result, by Assumption \ref{assum0}, it can be verified that, for some $0<\kappa \leq 1$,
	\begin{align*}
	E(|y_{t+1}|^{\kappa}|\bm y_t=\bm x) =& E\left| \sum_{i=1}^p\phi_i(u_{t+1})x_i+S_{\textrm{Q}}\left(b(u_{t+1})+\sum_{i=1}^p\beta_i(u_{t+1})x_i^2\right) \right|^\kappa \\
	\leq& \sum_{i=1}^p\left[E|\phi_i(u_{t+1})\sgn(x_i)+\sqrt{|\beta_{i}(u_{t+1})|}|^{\kappa}\right]|x_i|^{\kappa} +E|b(u_{t+1})|^{\kappa/2} \\
	\leq& \sum_{i=1}^pa_i|x_i|^{\kappa} +E|b(u_{t+1})|^{\kappa/2},
	\end{align*}
	where $a_i=\max\{E|\phi_{i}(u_{t+1})-\sqrt{|\beta_{i}(u_{t+1})|}|^\kappa,E|\phi_{i}(u_{t+1})+\sqrt{|\beta_{i}(u_{t+1})|}|^\kappa
	\}$ for $1\leq i\leq p$.
	Note that $\sum_{i=1}^pa_i<1$, and then we can find positive values $\{r_1,\ldots,r_{p-1}\}$ such that
	\begin{equation}\label{coe}
	a_p<r_{p-1}<1-\sum_{i=1}^{p-1}a_i \quad\text{and}\quad a_{i+1}+r_{i+1}<r_i<1-\sum_{k=1}^ia_k \;
	\text{for } 1\leq i \leq p-2.
	\end{equation}
	Consider the test function $g(\bm x)=1+|x_1|^\kappa+\sum_{i=1}^{p-1}r_i |x_{i+1}|^\kappa$, and we have that
	\begin{align*}\label{B.3}
	E[g(\bm y_{t+1})&|\bm y_{t}=\bm x] \nonumber \\
	&\leq 1+\sum_{i=1}^{p}a_i|x_i|^{\kappa}+\sum_{i=1}^{p-1}r_i |x_i|^\kappa+E|b(u_{t+1})|^{\kappa/2} \nonumber \\
	&= 1+(a_1+r_1)|x_1|^\kappa+\sum_{i=2}^{p-1}\dfrac{a_i+r_i}{r_{i-1}}r_{i-1}|x_i|^\kappa +\dfrac{a_p}{r_{p-1}}r_{p-1}|x_p|^\kappa+E|b(u_{t+1})|^{\kappa/2} \nonumber \\
	&\leq \rho g(\bm x)+1-\rho+E|b(u_{t+1})|^{\kappa/2},
	\end{align*}
	where, from \eqref{coe},
	\[
	\rho=\max\left\{a_1+r_1,\dfrac{a_2+r_2}{r_{1}},\cdots, \dfrac{a_{p-1}+r_{p-1}}{r_{p-2}},\dfrac{a_p}{r_{p-1}}\right\}<1.
	\]
	
	Denote $\epsilon=1-\rho-(1-\rho+E|b(u_{t+1})|^{\kappa/2})/g(\bm x)$, and $K=\{\bm x: \|\bm x\|\leq L\}$, where $L$ is a positive constant such that $g(\bm x)>1+E|b(u_{t+1})|^{\kappa/2}/(1-\rho)$ as $\|\bm x\|>L$.
	We can verify that
	\begin{align*}
		E[g(\bm y_{t+1})|\bm y_{t}=\bm x]\leq (1-\epsilon)g(\bm x), \quad \bm x \notin K,
	\end{align*}
	and
	\begin{align*}
		E[g(\bm y_{t+1})|\bm y_{t}=\bm x]\leq C<\infty, \quad \bm x \in K,
	\end{align*}
	i.e. Tweedie's drift criterion \citep[Theorem 4]{Tweedie1983} holds.
	Moreover, $\{\bm y_{t}\}$ is a Feller chain since, for each bounded continuous function $g^*(\cdot)$, $E[g^*(\bm y_{t})|\bm y_{t-1}=\bm x]$ is continuous with respect to $\bm x$, and then $K$ is a small set.
	As a result, from Theorem 4(ii) in \cite{Tweedie1983} and Theorems 1 and 2 in \cite{Feigin_Tweedie1985}, $\{\bm y_{t}\}$ is geometrically ergodic with a unique stationary distribution $\pi(\cdot)$, and
	\begin{equation*}
	\int_{\mathbb{R}^p}g(\bm x)\pi(d\bm x)=1+\left(1+\sum_{i=1}^{p-1}r_i\right)E|y_t|^{\kappa}<\infty,
	\end{equation*}
	which implies that $E|y_t|^{\kappa}<\infty$.
	This accomplishes the proof.
\end{proof}

\begin{proof}[Proof of Theorem \ref{thm1}]
	Recall that $q_t(\bm\theta_{\tau})=\bm y_{1,t-1}^{\prime}\bm \phi(\tau) +S_{\textrm{Q}}\left(b(\tau)+\bm y_{2,t-1}^{\prime}\bm \beta(\tau)\right)$, where $\bm\theta_{\tau}=(\bm \phi^{\prime}(\tau),b(\tau),\bm \beta^{\prime}(\tau))^{\prime}$.
	Define $L_n(\bm\theta_{\tau})=n^{-1}\sum_{t=p+1}^{n}\omega_t\ell_t(\bm\theta_{\tau})$, where $\ell_t(\bm\theta_{\tau})=\rho_{\tau}(y_t-q_t(\bm\theta_{\tau}))$.
	To show the consistency, it suffices to verify the following claims:
	\begin{itemize}
		\item[(i)] $E[\sup\limits_{\Theta_{\tau}}\omega_t\ell_t(\bm\theta_{\tau})]<\infty$;
		\item[(ii)] $E[\omega_t\ell_t(\bm\theta_{\tau})]$ has a unique minimum at $\bm\theta_{\tau 0}$;
		\item[(iii)] For any $\bm\theta_{\tau}^{\dagger}\in\Theta_{\tau}$, $E[\sup\limits_{\bm\theta_{\tau} \in B_{\eta}(\bm\theta_{\tau}^{\dagger})}\omega_t|\ell_t(\bm\theta_{\tau})-\ell_t(\bm\theta_{\tau}^{\dagger})|]\rightarrow 0$ as $\eta \rightarrow 0$, where $B_{\eta}(\bm\theta_{\tau}^{\dagger})=\{\bm\theta_{\tau} \in \Theta_{\tau}: \|\bm\theta_{\tau}^{\dagger}-\bm\theta_{\tau}\|<\eta\}$ is an open neighborhood of $\bm\theta_{\tau}^{\dagger}$ with radius $\eta>0$.
	\end{itemize}
    We first prove Claim (i). By Assumptions \ref{assum2}-\ref{assum3}, the boundedness of $b(\tau)$, $\bm\phi(\tau)$ and $\bm\beta(\tau)$, and the fact that $|\rho_{\tau}(x)|\leq |x|$, it holds that
    \[E[\sup_{\Theta_{\tau}}\omega_t\ell_t(\bm\theta_{\tau})]\leq E[\sup_{\Theta_{\tau}}|\omega_ty_t|]+E[\sup_{\Theta_{\tau}}|\omega_tq_t(\bm\theta_{\tau})|]< \infty.\]
    Hence, (i) is verified.
    	
	We next prove (ii). For $x\neq 0$, it holds that
	\begin{align}\label{identity}
	\rho_{\tau}(x-y)-\rho_{\tau}(x)&=-y\psi_{\tau}(x)+y\int_{0}^{1}[I(x\leq ys)-I(x\leq 0)]ds \nonumber \\
	&=-y\psi_{\tau}(x)+(x-y)[I(0>x>y)-I(0<x<y)],
	\end{align}
	where $\psi_{\tau}(x)=\tau-I(x<0)$; see \cite{Knight1998}. Let $\nu_t(\bm\theta_{\tau})=q_t(\bm\theta_{\tau})-q_t(\bm\theta_{\tau 0})$ and $\eta_{t,\tau}=y_{t}-q_t(\bm\theta_{\tau 0})$. By \eqref{identity}, it follows that
	\begin{align*}
	&\ell_t(\bm\theta_{\tau})-\ell_t(\bm\theta_{\tau 0}) \\
	 =&-\nu_t(\bm\theta_{\tau})\psi_{\tau}(\eta_{t,\tau})+[\eta_{t,\tau}-\nu_t(\bm\theta_{\tau})]\left[I(0>\eta_{t,\tau}>\nu_t(\bm\theta_{\tau}))-I(0<\eta_{t,\tau}<\nu_t(\bm\theta_{\tau}))\right].
	\end{align*}
	This, together with $\omega_t\geq 0$ by Assumption \ref{assum3} and $E[\psi_{\tau}(\eta_{t,\tau})]=0$, 
	implies that
	\begin{align}\label{gammasign}
	& E[\omega_t\ell_t(\bm\theta_{\tau})]-E[\omega_t\ell_t(\bm\theta_{\tau 0})] \nonumber \\
	=& E\left\{\omega_t[\eta_{t,\tau}-\nu_t(\bm\theta_{\tau})]\left[I(0>\eta_{t,\tau}>\nu_t(\bm\theta_{\tau}))-I(0<\eta_{t,\tau}<\nu_t(\bm\theta_{\tau}))\right]\right\}  \geq 0.
	\end{align}
	By Assumption \ref{assum4}, $f_{t-1}(x)$ is continuous at a neighborhood of $q_t(\bm\theta_{\tau 0})$, then the above equality holds if and only if $\nu_t(\bm\theta_{\tau})=0$ a.s. for some $t \in \mathbb{Z}$. Then we have
   \[\bm y_{1,t-1}^{\prime}[\bm\phi_0(\tau)-\bm\phi(\tau)] = S_{\textrm{Q}}\left(b(\tau)+\bm y_{2,t-1}^{\prime}\bm\beta(\tau)\right)- S_{\textrm{Q}}\left(b_0(\tau)+\bm y_{2,t-1}^{\prime}\bm\beta_0(\tau)\right).\]
    Note that $b_0(\tau)\neq 0$ and, given $\mathcal{F}_{t-2}$, $y_{t-1}$ is independent of all the others. As a result, it holds that $\phi_1(\tau)=\phi_{10}(\tau)$ and $\beta_1(\tau)=\beta_{10}(\tau)$. Sequentially we can show that $\phi_i(\tau)=\phi_{i0}(\tau)$ and $\beta_i(\tau)=\beta_{i0}(\tau)$ for $i\geq 2$, and hence $b(\tau)=b_0(\tau)$.
	Therefore, $\bm\theta_{\tau}=\bm\theta_{\tau 0}$ and (ii) is verified.
	
	Finally, we show (iii). By Taylor expansion, it holds that
	\[|q_t(\bm\theta_{\tau})-q_t(\bm\theta_{\tau}^{\dagger})|\leq \|\bm\theta_{\tau}-\bm\theta_{\tau}^{\dagger}\|\left\|\dot{q}_t(\bm\theta_{\tau}^{*})\right\|,\]
	where $\bm\theta_{\tau}^{*}$ is between $\bm\theta_{\tau}$ and $\bm\theta_{\tau}^{\dagger}$.
	This together with the Lipschitz continuity of $\rho_{\tau}(x)$, by Assumption \ref{assum3}, we have
	\[
	E[\sup_{\bm\theta_{\tau} \in B_{\eta}(\bm\theta_{\tau}^{\dagger})}\omega_t|\ell_t(\bm\theta_{\tau})-\ell_t(\bm\theta_{\tau}^{\dagger})|] \leq C\eta E[\omega_t\dot{q}_t(\bm\theta_{\tau}^{*})]
	\]
	tends to 0 as $\eta \to 0$. Hence, Claim (iii) holds.
	
	Based on Claims (i)-(iii), by a method similar to that in \cite{Huber1973}, we next verify the consistency.
	Let $V$ be any open neighborhood of $\bm\theta_{\tau 0} \in \Theta_{\tau}$. By Claim (iii), for any $\bm\theta_{\tau}^{\dagger} \in V^c=\Theta_{\tau}/V$ and $\epsilon>0$, there exists an $\eta_0>0$ such that
	\begin{eqnarray}\label{thm1_ia}
	E[\inf_{\bm\theta_{\tau} \in B_{\eta_0}(\bm\theta_{\tau}^{\dagger})}\omega_t\ell_t(\bm\theta_{\tau})]\geq E[\omega_t\ell_t(\bm\theta_{\tau}^{\dagger})]-\epsilon.
	\end{eqnarray}
	From Claim (i), by the ergodic theorem, it follows that
	\begin{eqnarray}\label{thm1_ib}
	\dfrac{1}{n}\sum_{t=p+1}^{n}\inf_{\bm\theta_{\tau} \in B_{\eta_0}(\bm\theta_{\tau}^{\dagger})}\omega_t\ell_t(\bm\theta_{\tau}) \geq E[\inf_{\bm\theta_{\tau} \in B_{\eta_0}(\bm\theta_{\tau}^{\dagger})}\omega_t\ell_t(\bm\theta_{\tau})]-\epsilon
	\end{eqnarray}
	as $n$ is large enough. Since $V^c$ is compact, we can choose $\{B_{\eta_0}(\bm\theta_{\tau i}): \bm\theta_{\tau i} \in V^c, i=1,\ldots, k\}$ to be a finite covering of $V^c$. Then by \eqref{thm1_ia} and \eqref{thm1_ib}, we have
	\begin{eqnarray}\label{thm1_ic}
	\inf_{\bm\theta_{\tau} \in V^c}L_n(\bm\theta_{\tau})&=& \min_{1\leq i \leq k} \inf_{\bm\theta_{\tau} \in B_{\eta_0}(\bm\theta_{\tau i})}L_n(\bm\theta_{\tau}) \nonumber\\
	&\geq&\min_{1\leq i \leq k}\dfrac{1}{n}\sum_{t=p+1}^{n}\inf_{\bm\theta_{\tau} \in B_{\eta_0}(\bm\theta_{\tau i})}\omega_t\ell_t(\bm\theta_{\tau}) \nonumber\\
	&\geq&\min_{1\leq i \leq k}E[\inf_{\bm\theta_{\tau} \in B_{\eta_0}(\bm\theta_{\tau i})}\omega_t\ell_t(\bm\theta_{\tau})]-\epsilon
	\end{eqnarray}
	as $n$ is large enough. Moreover, for each $\bm\theta_{\tau i} \in V^c$, by Claim (ii), there exists an $\epsilon_0>0$ such that
	\begin{eqnarray}\label{thm1_id}
	E[\inf_{\bm\theta_{\tau} \in B_{\eta_0}(\bm\theta_{\tau i})}\omega_t\ell_t(\bm\theta_{\tau})] \geq E[\omega_t\ell_t(\bm\theta_{\tau 0})]+3\epsilon_0.
	\end{eqnarray}
	Therefore, by \eqref{thm1_ic} and \eqref{thm1_id}, taking $\epsilon=\epsilon_0$, it holds that
	\begin{eqnarray}\label{thm1_ie}
	\inf_{\bm\theta_{\tau} \in V^c}L_n(\bm\theta_{\tau})\geq E[\omega_t\ell_t(\bm\theta_{\tau 0})]+2\epsilon_0.
	\end{eqnarray}	
	Furthermore, by the ergodic theorem, it follows that
	\begin{eqnarray}\label{thm1_if}
	\inf_{\bm\theta_{\tau} \in V}L_n(\bm\theta_{\tau})\leq L_n(\bm\theta_{\tau 0})= \dfrac{1}{n}\sum_{t=p+1}^{n}\omega_t \ell_t(\bm\theta_{\tau 0}) \leq E[\omega_t\ell_t(\bm\theta_{\tau 0})]+\epsilon_0.
	\end{eqnarray}	
	Combing \eqref{thm1_ie} and \eqref{thm1_if}, we have
	\begin{eqnarray}\label{thm1_ig}
	\inf_{\bm\theta_{\tau} \in V^c}L_n(\bm\theta_{\tau})\geq E[\omega_t\ell_t(\bm\theta_{\tau 0})]+2\epsilon_0 > E[\omega_t\ell_t(\bm\theta_{\tau 0})]+\epsilon_0 \geq \inf_{\bm\theta_{\tau} \in V}L_n(\bm\theta_{\tau}),
	\end{eqnarray}	
	which implies that
	\[\widehat{\bm\theta}_{\tau n} \in V \hspace{5mm} \text{a.s.} \hspace{2mm} \text{for} \hspace{2mm} \forall V, \hspace{2mm} \text{as} \hspace{2mm} n \hspace{2mm}\text{is large enough}.\]
	By the arbitrariness of $V$, it implies that $\widehat{\bm\theta}_{\tau n}\rightarrow \bm\theta_{\tau 0}$ a.s.
	The proof of this theorem is complete.	
\end{proof}

\begin{proof}[Proof of Theorem \ref{thm2}]	
	For $\bm u\in \mathbb{R}^{2p+1}$, define $H_n(\bm u)=n[L_n(\bm\theta_{\tau 0}+\bm u)-L_n(\bm\theta_{\tau 0})]$, where $L_n(\bm\theta_{\tau})=n^{-1}\sum_{t=p+1}^{n}\omega_t\rho_{\tau}(y_t-q_t(\bm\theta_{\tau}))$. Denote $\widehat{\bm u}_n=\widehat{\bm\theta}_{\tau n}-\bm\theta_{\tau 0}$. By Theorem \ref{thm1}, it holds that $\widehat{\bm u}_n=o_p(1)$.	
	Note that $\widehat{\bm u}_n$ is the minimizer of $H_n(\bm u)$, since $\widehat{\bm\theta}_{\tau n}$ minimizes $L_n(\bm\theta_{\tau})$. Define $J=\Omega_1(\tau)/2$. By Assumption \ref{assum3} and the ergodic theorem, $J_n=J+o_p(1)$, where $J_n$ is defined in Lemma \ref{lem2}.	
	Moreover, from Lemma \ref{lem2}, it follows that
	\begin{align}\label{Hn}
	H_n(\widehat{\bm u}_n)=&-\sqrt{n}\widehat{\bm u}_n^{\prime}\bm T_n+\sqrt{n}\widehat{\bm u}_n^{\prime}J\sqrt{n}\widehat{\bm u}_n+o_p(\sqrt{n}\|\widehat{\bm u}_n\|+n\|\widehat{\bm u}_n\|^2) \\
	\geq &-\sqrt{n}\|\widehat{\bm u}_n\|[\|\bm T_n\|+o_p(1)]+n\|\widehat{\bm u}_n\|^2[\lambda_{\min}+o_p(1)], \nonumber
	\end{align}
	where $\lambda_{\min}$ is the smallest eigenvalue of $J$, and $\bm T_n$ is defined in Lemma \ref{lem2}.
	Note that, as $n\rightarrow \infty$, $\bm T_n$ converges in distribution to a normal random variable with mean zero and variance matrix $\tau(1-\tau)\Omega_0(\tau)$.
	
	Since $H_n(\widehat{\bm u}_n)\leq 0$, by Assumption \ref{assum3}, it holds that
	\begin{align}\label{root-n}
	\sqrt{n}\|\widehat{\bm u}_n\|\leq [\lambda_{\min}+o_p(1)]^{-1}[\|\bm T_n\|+o_p(1)]=O_p(1).
	\end{align}
	This together with Theorem \ref{thm1}, verifies the $\sqrt{n}$-consistency. Hence, Statement (i) holds.
	
	Let $\sqrt{n}\bm u_n^*=J^{-1}\bm T_n/2=\Omega_1^{-1}(\tau)\bm T_n$, then we have
	\begin{align*}
	\sqrt{n}\bm u_n^*\rightarrow N\left(\bm 0,\tau(1-\tau)\Omega_1^{-1}(\tau)\Omega_0(\tau)\Omega_1^{-1}(\tau)\right)
	\end{align*}
	in distribution as $n\rightarrow \infty$. Therefore, it suffices to show that $\sqrt{n}\bm u_n^*-\sqrt{n}\widehat{\bm u}_n=o_p(1)$.
	By \eqref{Hn} and \eqref{root-n}, we have
	\begin{align*}
	H_n(\widehat{\bm u}_n)=&-\sqrt{n}\widehat{\bm u}_n^{\prime}\bm T_n+\sqrt{n}\widehat{\bm u}_n^{\prime}J\sqrt{n}\widehat{\bm u}_n+o_p(1) \\
	=&-2\sqrt{n}\widehat{\bm u}_n^{\prime}J\sqrt{n}\bm u_n^*+\sqrt{n}\widehat{\bm u}_n^{\prime}J\sqrt{n}\widehat{\bm u}_n+o_p(1),
	\end{align*}
	and
	\begin{align*}
	H_n(\bm u_n^*)=&-\sqrt{n}\bm u_n^{*\prime}\bm T_n+\sqrt{n}\bm u_n^{*\prime}J\sqrt{n}\bm u_n^*+o_p(1)
	=-\sqrt{n}\bm u_n^{*\prime}J\sqrt{n}\bm u_n^*+o_p(1).
	\end{align*}
	It follows that
	\begin{align}\label{normality}
	H_n(\widehat{\bm u}_n)-H_n(\bm u_n^*)=&(\sqrt{n}\widehat{\bm u}_n-\sqrt{n}\bm u_n^*)^{\prime}J(\sqrt{n}\widehat{\bm u}_n-\sqrt{n}\bm u_n^*)+o_p(1)  \nonumber \\
	\geq & \lambda_{\min}\|\sqrt{n}\widehat{\bm u}_n-\sqrt{n}\bm u_n^*\|^2+o_p(1).
	\end{align}
	Since $H_n(\widehat{\bm u}_n)-H_n(\bm u_n^*)=n[L_n(\bm\theta_{\tau 0}+\widehat{\bm u}_n)-L_n(\bm\theta_{\tau 0}+\bm u_n^*)]\leq 0$ a.s., then \eqref{normality} implies that $\|\sqrt{n}\widehat{\bm u}_n-\sqrt{n}\bm u_n^*\|=o_p(1)$. We verify the asymptotic normality in Statement (ii), the proof is hence accomplished.
	
\end{proof}

\begin{proof}[Proof of Corollary \ref{cor3}]	
Recall that $q_t(\bm\theta_{\tau}) =\bm y_{1,t-1}^{\prime}\bm \phi(\tau) +S_{\textrm{Q}}\left(b(\tau)+\bm y_{2,t-1}^{\prime}\bm \beta(\tau)\right)$, $Q_{\tau}(y_{n+1}|\mathcal{F}_{n})=q_{t+1}(\bm\theta_{\tau 0})$ and $\widehat{Q}_{\tau}(y_{n+1}|\mathcal{F}_{n})=q_{t+1}(\widehat{\bm\theta}_{\tau n})$. 
Since $\min\{|b(\tau)|,|\beta_1(\tau)|,\ldots,|\beta_p(\tau)|\}\geq C$ for a constant $C>0$ is assumed for $\bm\theta_{\tau}$, then the first and second derivative functions of $q_t(\bm\theta_{\tau})$ are well-defined. Specifically, the first derivative function is $\dot{q}_t(\bm\theta_{\tau})=(\bm y_{1,t-1}^{\prime},0.5|h_{t}(\bm\theta_{\tau})|^{-1/2},0.5|h_{t}(\bm\theta_{\tau})|^{-1/2}\bm y_{2,t-1}^{\prime})^{\prime}$, and the second derivative function is
\begin{align}\label{ddotqt}
\ddot{q}_t(\bm\theta_{\tau})=\begin{pmatrix} \bm 0 & -\dfrac{\text{sgn}(h_{t}(\bm\theta_{\tau}))}{|h_{t}(\bm\theta_{\tau})|^{3/2}} & -\dfrac{\text{sgn}(h_{t}(\bm\theta_{\tau}))}{|h_{t}(\bm\theta_{\tau})|^{3/2}}\bm y_{2,t-1}^{\prime} \\
\bm 0 & \bm 0 & \bm 0 \\
\bm 0 & -\dfrac{\text{sgn}(h_{t}(\bm\theta_{\tau}))}{|h_{t}(\bm\theta_{\tau})|^{3/2}}\bm y_{2,t-1} & -\dfrac{\text{sgn}(h_{t}(\bm\theta_{\tau}))}{|h_{t}(\bm\theta_{\tau})|^{3/2}}\bm y_{2,t-1}\bm y_{2,t-1}^{\prime}	
\end{pmatrix}
\end{align}
with $h_{t}(\bm\theta_{\tau})=b(\tau)+\bm y_{2,t-1}^{\prime}\bm \beta(\tau)$. 
By Taylor expansion, we have
\begin{align}\label{qtExpansion}
q_{t+1}(\widehat{\bm\theta}_{\tau n})-q_{t+1}(\bm\theta_{\tau 0})=\dot{q}_{t+1}^{\prime}(\bm\theta_{\tau 0})(\widehat{\bm\theta}_{\tau n}-\bm\theta_{\tau 0})+\dfrac{1}{2}(\bm\theta_{\tau}^*-\bm\theta_{\tau 0})^{\prime}\ddot{q}_{t+1}(\bm\theta_{\tau}^*)(\bm\theta_{\tau}^*-\bm\theta_{\tau 0}),
\end{align}	
where $\bm\theta_{\tau}^*$ is between $\bm\theta_{\tau 0}$ and $\widehat{\bm\theta}_{\tau n}$.  
By Theorem \ref{thm2}, we have $\widehat{\bm\theta}_{\tau n}-\bm\theta_{\tau 0}=O_p(n^{-1/2})$. 
Then both $\dot{q}_t(\bm\theta_{\tau 0})$ and $\ddot{q}_t(\bm\theta_{\tau}^*)$ are well-defined. 
To show the representation of Corollary \ref{cor3}, it is sufficient to show that $\ddot{q}_{t+1}(\bm\theta_{\tau}^*)=O_p(1)$. The definition of $\ddot{q}_t(\cdot)$ in \eqref{ddotqt}  together with $E(|y_t|)<\infty$ implies that $\ddot{q}_{t+1}(\bm\theta_{\tau}^*)=O_p(1)$. The proof of this corollary is complete. 
\end{proof}

\begin{proof}[Proof of Theorem \ref{thm3}]	
We first show that $\widehat{\mu}_{i,\tau}=\mu_{i,\tau}+o_p(1)$ and $\widehat{\sigma}_{i,\tau}^2=\sigma_{i,\tau}^2+o_p(1)$ for $i=1$ and 2. 	
Recall that $\eta_{t,\tau}=y_t-Q_{\tau}(y_t|\mathcal{F}_{t-1})=y_t-q_t(\bm\theta_{\tau 0})$ and $\widehat{\eta}_{t,\tau}=y_t-\widehat{Q}_{\tau}(y_t|\mathcal{F}_{t-1})=y_t-q_t(\widehat{\bm\theta}_{\tau n})$.
By the Taylor expansion, we have
\[\widehat{\eta}_{t,\tau}-\eta_{t,\tau}=-[q_t(\widehat{\bm\theta}_{\tau n})-q_t(\bm\theta_{\tau 0})]=-\dot{q}_t^{\prime}(\bm\theta_{\tau}^{*})(\widehat{\bm\theta}_{\tau n}-\bm\theta_{\tau 0}),\]
and
\[\widehat{\eta}_{t,\tau}^2-\eta_{t,\tau}^2=-[2y_t-q_t(\bm\theta_{\tau 0})-q_t(\widehat{\bm\theta}_{\tau n})]\dot{q}_t^{\prime}(\bm\theta_{\tau}^{*})(\widehat{\bm\theta}_{\tau n}-\bm\theta_{\tau 0}),\]
where $\bm\theta_{\tau}^{*}$ is between $\bm\theta_{\tau 0}$ and $\widehat{\bm\theta}_{\tau n}$.
Then by the law of large numbers, $E(y_t^2)<\infty$ and the fact that $\widehat{\bm\theta}_{\tau n}-\bm\theta_{\tau 0}=O_p(n^{-1/2})$, it holds that
\begin{align}\label{mu1}
\widehat{\mu}_{1,\tau}=\dfrac{1}{n-p}\sum_{t=p+1}^n\eta_{t,\tau}+\dfrac{1}{n-p}\sum_{t=p+1}^n(\widehat{\eta}_{t,\tau}-\eta_{t,\tau})=\mu_{1,\tau}+o_p(1),
\end{align}
and
\begin{align}\label{sigma1}
\widehat{\sigma}_{1,\tau}^2=\dfrac{1}{n-p}\sum_{t=p+1}^n(\widehat{\eta}_{t,\tau}^2-\eta_{t,\tau}^2)+\dfrac{1}{n-p}\sum_{t=p+1}^n\eta_{t,\tau}^2-\widehat{\mu}_{1,\tau}^2=\sigma_{1,\tau}^2+o_p(1).
\end{align}
Similarly, we can show that
\begin{align*}
\widehat{\mu}_{2,\tau}&=\dfrac{1}{n-p}\sum_{t=p+1}^n|\eta_{t,\tau}|+\dfrac{1}{n-p}\sum_{t=p+1}^n(|\widehat{\eta}_{t,\tau}|-|\eta_{t,\tau}|) \\
&\leq \dfrac{1}{n-p}\sum_{t=p+1}^n|\eta_{t,\tau}|+\dfrac{1}{n-p}\sum_{t=p+1}^n|\widehat{\eta}_{t,\tau}-\eta_{t,\tau}|=\mu_{2,\tau}+o_p(1),
\end{align*}
and
\[\widehat{\sigma}_{2,\tau}^2=\dfrac{1}{n-p}\sum_{t=p+1}^n(|\widehat{\eta}_{t,\tau}|^2-|\eta_{t,\tau}|^2)+\dfrac{1}{n-p}\sum_{t=p+1}^n|\eta_{t,\tau}|^2-\widehat{\mu}_{2,\tau}^2=\sigma_{2,\tau}^2+o_p(1).\]

Since $|\sum_{t=p+1}^n\psi_{\tau}(\widehat{\eta}_{t,\tau})|<1$, by an elementary calculation, we have
\begin{align}\label{decomposition}
&\dfrac{1}{\sqrt{n}}\sum_{t=p+k+1}^n w_t\psi_{\tau}(\widehat{\eta}_{t,\tau})[\widehat{\eta}_{t-k,\tau}-\widehat{\mu}_{1,\tau}] \nonumber\\
=&\dfrac{1}{\sqrt{n}}\sum_{t=p+k+1}^nw_t\psi_{\tau}(\widehat{\eta}_{t,\tau})\widehat{\eta}_{t-k,\tau}+O_p(n^{-1/2}) \nonumber\\
=&\dfrac{1}{\sqrt{n}}\sum_{t=p+k+1}^nw_t\psi_{\tau}(\eta_{t,\tau})\eta_{t-k,\tau}+A_{n1}+A_{n2}+A_{n3}+O_p(n^{-1/2}),
\end{align}
where
\begin{align*}
A_{n1}&=\dfrac{1}{\sqrt{n}}\sum_{t=p+k+1}^nw_t[\psi_{\tau}(\widehat{\eta}_{t,\tau})-\psi_{\tau}(\eta_{t,\tau})]\eta_{t-k,\tau},\\
A_{n2}&=\dfrac{1}{\sqrt{n}}\sum_{t=p+k+1}^nw_t\psi_{\tau}(\eta_{t,\tau})(\widehat{\eta}_{t-k,\tau}-\eta_{t-k,\tau}),\\
A_{n3}&=\dfrac{1}{\sqrt{n}}\sum_{t=p+k+1}^nw_t[\psi_{\tau}(\widehat{\eta}_{t,\tau})-\psi_{\tau}(\eta_{t,\tau})](\widehat{\eta}_{t-k,\tau}-\eta_{t-k,\tau}).
\end{align*}
First, we consider $A_{n1}$. For any $\bm\nu\in \mathbb{R}^{2p+1}$, denote
\[\zeta_t(\bm\nu)=w_t[\psi_{\tau}(y_t-q_t(\bm\theta_{\tau 0}+n^{-1/2}\bm\nu))-\psi_{\tau}(y_t-q_t(\bm\theta_{\tau 0}))]\eta_{t-k,\tau}\quad\text{and} \]
\[\phi_n(\bm\nu)=\dfrac{1}{\sqrt{n}}\sum_{t=p+k+1}^n\{\zeta_t(\bm\nu)-E[\zeta_t(\bm\nu)|\mathcal{F}_{t-1}]\}.\]
Note that $w_t\in\mathcal{F}_{t-1}$. Then by the Taylor expansion and the Cauchy-Schiwarz inequality, together with $E[w_{t}\|\bm y_{1,t-1}\|^3]<\infty$ by Assumption \ref{assum3} and the fact that $f_{t-1}(x)$ is bounded by Assumption \ref{assum4}, it holds that
\begin{align*}
E[\zeta_t^2(\bm\nu)]&=E\{w_t|I(y_t<q_t(\bm\theta_{\tau 0}+n^{-1/2}\bm\nu))-I(y_t<q_t(\bm\theta_{\tau 0}))|\cdot w_t\eta_{t-k,\tau}^2\} \\
&\leq \left\{E[w_t|F_{t-1}(q_t(\bm\theta_{\tau 0}+n^{-1/2}\bm\nu))-F_{t-1}(q_t(\bm\theta_{\tau 0}))|]^2\right\}^{1/2}[E(w_t^2\eta_{t-k,\tau}^4)]^{1/2} \\
&\leq C \left\{E[f_{t-1}^2(q_t(\bm\theta_{\tau}^{*}))|n^{-1/2}\bm\nu^{\prime}w_t^2\dot{q}_t(\bm\theta_{\tau}^{*})\dot{q}_t^{\prime}(\bm\theta_{\tau}^{*}) n^{-1/2}\bm\nu|]\right\}^{1/2} \\
&\leq C n^{-1/2}\|\bm\nu\|\left\{\|E[f_{t-1}^2(\bm\theta_{\tau}^{*})w_t^2\dot{q}_t(\bm\theta_{\tau}^{*})\dot{q}_t^{\prime}(\bm\theta_{\tau}^{*})]\|\right\}^{1/2}=o(1).
\end{align*}
where $\bm\theta_{\tau}^{*}$ is between $\bm\theta_{\tau 0}$ and $\widehat{\bm\theta}_{\tau n}$. This together with $E[\phi_n(\bm\nu)]=0$, implies that
\begin{align}\label{phinsquare}
E[\phi_n^2(\bm\nu)]\leq \dfrac{1}{n}\sum_{t=p+k+1}^n E[\zeta_t^2(\bm\nu)]=o(1).
\end{align}
For any $\bm\nu_1,\bm\nu_2\in \mathbb{R}^{2p+1}$ and $\delta>0$, it holds that
\begin{align*}
  &I(y_t<q_t(\bm\theta_{\tau 0}+n^{-1/2}\bm\nu_2))-I(y_t<q_t(\bm\theta_{\tau 0}+n^{-1/2}\bm\nu_1))  \\
 = & I(q_t(\bm\theta_{\tau 0}+n^{-1/2}\bm\nu_1)<y_t<q_t(\bm\theta_{\tau 0}+n^{-1/2}\bm\nu_2)) \\
 &-I(q_t(\bm\theta_{\tau 0}+n^{-1/2}\bm\nu_1)>y_t>q_t(\bm\theta_{\tau 0}+n^{-1/2}\bm\nu_2)).
\end{align*}
Then by the Taylor expansion, for any $\bm\nu_1,\bm\nu_2\in \mathbb{R}^{2p+1}$ and $\delta>0$, it can be verified that
\begin{align*}
&E\sup_{\|\bm\nu_1-\bm\nu_2\|\leq \delta}|\zeta_t(\bm\nu_1)-\zeta_t(\bm\nu_2)| \\
\leq &\sup_{\|\bm\nu_1-\bm\nu_2\| \leq \delta}|w_t[I(y_t<q_t(\bm\theta_{\tau 0}+n^{-1/2}\bm\nu_2))-I(y_t<q_t(\bm\theta_{\tau 0}+n^{-1/2}\bm\nu_1))]\eta_{t-k,\tau}| \\
\leq &E \left\{w_tI(|y_t|\leq \sup_{\|\bm\nu_1-\bm\nu_2\| \leq \delta} |q_t(\bm\theta_{\tau 0}+n^{-1/2}\bm\nu_2)-q_t(\bm\theta_{\tau 0}+n^{-1/2}\bm\nu_1)|)|\eta_{t-k,\tau}|\right\} \\
\leq & E \left\{w_t\text{Pr}\left(|y_t|\leq \sup_{\|\bm\nu_1-\bm\nu_2\| \leq \delta} n^{-1/2}\|\bm\nu_1-\bm\nu_2\|\|\dot{q}_t(\bm\theta_{\tau}^{*})\|\right)|\eta_{t-k,\tau}|\right\} \\
\leq & 2n^{-1/2}\delta E[\|w_t\sup_{x}f_{t-1}(x)\dot{q}_t(\bm\theta_{\tau}^{*})\eta_{t-k,\tau}\|],
\end{align*}
where $\bm\theta_{\tau}^{*}$ is between $\bm\theta_{\tau 0}+n^{-1/2}\bm\nu_1$ and $\bm\theta_{\tau 0}+n^{-1/2}\bm\nu_2$. This together with $E[w_{t}\|\bm y_{1,t-1}\|^3]<\infty$ by Assumption \ref{assum3}, implies that
\begin{align}\label{phindiffabs}
E\sup_{\|\bm\nu_1-\bm\nu_2\|\leq \delta}|\phi_n(\bm\nu_1)-\phi_n(\bm\nu_2)| \leq \dfrac{2}{\sqrt{n}}\sum_{t=p+k+1}^nE\sup_{\|\bm\nu_1-\bm\nu_2\|\leq \delta}|\zeta_t(\bm\nu_1)-\zeta_t(\bm\nu_2)| \leq C\delta.
\end{align}
Therefore, it follows from \eqref{phinsquare}, \eqref{phindiffabs} and the finite covering theorem that
\begin{align}\label{phinabs}
\sup_{\|\bm\nu\|\leq M}|\phi_n(\bm\nu)| = o_p(1).
\end{align}
Note that $E[\zeta_t(\bm\nu)|\mathcal{F}_{t-1}]=w_t[F_{t-1}(q_t(\bm\theta_{\tau 0}))-F_{t-1}(q_t(\bm\theta_{\tau 0}+n^{-1/2}\bm\nu))]\eta_{t-k,\tau}$.
Moreover, by the Taylor expansion and $E[w_{t}\|\bm y_{1,t-1}\|^3]<\infty$ by Assumption \ref{assum3}, we can show that
\begin{align*}
&\dfrac{1}{\sqrt{n}}\sum_{t=p+k+1}^nE[\zeta_t(\bm\nu)|\mathcal{F}_{t-1}] \\
=& -\dfrac{1}{n}\sum_{t=p+k+1}^n w_t f_{t-1}(F_{t-1}^{-1}(\tau))\dot{q}_t^{\prime}(\bm\theta_{\tau 0})\eta_{t-k,\tau} \cdot\bm\nu \\
&- \bm\nu^{\prime}\cdot \dfrac{1}{n}\sum_{t=p+k+1}^n w_t f_{t-1}(F_{t-1}^{-1}(\tau))\ddot{q}_t(\bm\theta_{\tau}^{*})\eta_{t-k,\tau} \cdot n^{-1/2}\bm\nu^* \\
&- \bm\nu^{\prime}\cdot \dfrac{1}{n}\sum_{t=p+k+1}^n w_t f_{t-1}(q_t(\bm\theta_{\tau}^{**}))\dot{q}_t(\bm\theta_{\tau}^{**})\dot{q}_t^{\prime}(\bm\theta_{\tau}^{\ddag})\eta_{t-k,\tau} \cdot n^{-1/2}\bm\nu^{\dag},
\end{align*}
where $\bm\nu^*$ and $\bm\nu^{\dag}$ are between $\bm 0$ and $\bm\nu$, and $\bm\theta_{\tau}^{*}$, $\bm\theta_{\tau}^{**}$, $\bm\theta_{\tau}^{\ddag}$ are between $\bm\theta_{\tau 0}$ and $\bm\theta_{\tau 0}+n^{-1/2}\bm\nu$.
Then by the law of large numbers, Assumptions \ref{assum3} and \ref{assum4}, it follows that
\begin{align*}
  \sup_{\|\bm\nu\|\leq \delta} \left|\dfrac{1}{\sqrt{n}}\sum_{t=p+k+1}^nE[\zeta_t(\bm\nu)|\mathcal{F}_{t-1}]+H_{1k}^{\prime}\bm\nu\right| \leq n^{-1/2}C\delta +o_p(1)=o_p(1).
\end{align*}
This together with \eqref{phinabs} and the fact that $\sqrt{n}(\widehat{\bm\theta}_{\tau n}-\bm\theta_{\tau 0})=O_p(1)$, implies that
\begin{align}\label{An1}
A_{n1}=\dfrac{1}{\sqrt{n}}\sum_{t=p+k+1}^nE[\zeta_t(\bm\nu)|\mathcal{F}_{t-1}]+o_p(1)=-H_{1k}^{\prime}\sqrt{n}(\widehat{\bm\theta}_{\tau n}-\bm\theta_{\tau 0})+o_p(1).
\end{align}
Next, we consider $A_{n2}$. Since $E[\psi_{\tau}(\eta_{t,\tau})]=0$, it holds that $E(A_{n2})=0$. By the Taylor expansion, the law of large numbers and the fact that $\sqrt{n}(\widehat{\bm\theta}_{\tau n}-\bm\theta_{\tau 0})=O_p(1)$, we can verify that $E(A_{n2}^2)=o_p(1)$ under Assumption \ref{assum3}. Hence,
\begin{align}\label{An2}
A_{n2}=o_p(1).
\end{align}
Finally, we consider $A_{n3}$. For any $\bm\nu\in \mathbb{R}^{2p+1}$, denote $\eta_{t,\tau}(\bm\nu)=y_t-q_t(\bm\theta_{\tau 0}+n^{-1/2}\bm\nu)$ and
\[\varsigma_t(\bm\nu)=w_t[\psi_{\tau}(y_t-q_t(\bm\theta_{\tau 0}+n^{-1/2}\bm\nu))-\psi_{\tau}(y_t-q_t(\bm\theta_{\tau 0}))][\eta_{t-k,\tau}(\bm\nu)-\eta_{t-k,\tau}].\]
By a method similar to the proof of \eqref{phinsquare} and \eqref{phinabs}, we can show that, for any $\delta>0$,
\[\sup_{\|\bm\nu\|\leq \delta}\left|\dfrac{1}{\sqrt{n}}\sum_{t=p+k+1}^n\{\varsigma_t(\bm\nu)-E[\varsigma_t(\bm\nu)|\mathcal{F}_{t-1}]\}\right|=o_p(1)\]
and
\[\sup_{\|\bm\nu\|\leq \delta}\left|\dfrac{1}{\sqrt{n}}\sum_{t=p+k+1}^nE[\varsigma_t(\bm\nu)|\mathcal{F}_{t-1}]\right|=o_p(1).\]
As a result,
\[\sup_{\|\bm\nu\|\leq \delta}\left|\dfrac{1}{\sqrt{n}}\sum_{t=p+k+1}^nw_t[\psi_{\tau}(y_t-q_t(\bm\theta_{\tau 0}+n^{-1/2}\bm\nu))-\psi_{\tau}(y_t-q_t(\bm\theta_{\tau 0}))][\eta_{t-k,\tau}(\bm\nu)-\eta_{t-k,\tau}]\right|=o_p(1),\]
which together with $\sqrt{n}(\widehat{\bm\theta}_{\tau n}-\bm\theta_{\tau 0})=O_p(1)$, implies that
\begin{align}\label{An3}
A_{n3}=o_p(1).
\end{align}
Combing \eqref{mu1}-\eqref{decomposition}, \eqref{An1}-\eqref{An3} and Theorem \ref{thm2}, we have
\[\widehat{\rho}_{k,\tau}=\dfrac{1}{\sqrt{(\tau-\tau^2)\sigma_{1,\tau}^2}}\dfrac{1}{n}\sum_{t=p+k+1}^n [w_t\eta_{t-k,\tau}-H_{1k}^{\prime}\Omega_1^{-1}(\tau)w_t\dot{q}_t(\bm\theta_{\tau 0})]\psi_{\tau}(\eta_{t,\tau})+o_p(n^{-1/2}).\]
Therefore, for $\widehat{\bm\rho}=(\widehat{\rho}_{1,\tau},\ldots,\widehat{\rho}_{K,\tau})^{\prime}$, we have
\[\widehat{\bm\rho}=\dfrac{1}{\sqrt{(\tau-\tau^2)\sigma_{1,\tau}^2}}\dfrac{1}{n}\sum_{t=p+k+1}^n [w_t\bm\epsilon_{1,t-1}-H_{1}(\tau)\Omega_1^{-1}(\tau)w_t\dot{q}_t(\bm\theta_{\tau 0})]\psi_{\tau}(\eta_{t,\tau})+o_p(n^{-1/2}),\]
where $\bm\epsilon_{1,t}=(\eta_{t,\tau},\ldots,\eta_{t-K+1,\tau})^{\prime}$ and $H_{1}(\tau)=(H_{11},\ldots, H_{1K})^{\prime}$.

Let $\widehat{\bm r}=(\widehat{r}_{1,\tau},\ldots,\widehat{r}_{K,\tau})^{\prime}$. Similar to the proof of $\widehat{\bm\rho}$, we can show that
\[\widehat{\bm r}=\dfrac{1}{\sqrt{(\tau-\tau^2)\sigma_{2,\tau}^2}}\dfrac{1}{n}\sum_{t=p+k+1}^n [w_t\bm\epsilon_{2,t-1}-H_{2}(\tau)\Omega_1^{-1}(\tau)w_t\dot{q}_t(\bm\theta_{\tau 0})]\psi_{\tau}(\eta_{t,\tau})+o_p(n^{-1/2}),\]
where $\bm\epsilon_{2,t}=(|\eta_{t,\tau}|,\ldots,|\eta_{t-K+1,\tau}|)^{\prime}$ and $H_{2}(\tau)=(H_{21},\ldots, H_{2K})^{\prime}$. 
Therefore, it follows that
\[
\sqrt{n}(\widehat{\bm \rho}^\prime,\widehat{\bm r}^\prime)^\prime=\dfrac{1}{\sqrt{\tau-\tau^2}}\dfrac{1}{\sqrt{n}}\sum_{t=p+k+1}^n[w_t\bm\epsilon_{t-1}-H(\tau)\Omega_1^{-1}(\tau)w_t\dot{q}_t(\bm\theta_{\tau 0})]\psi_{\tau}(\eta_{t,\tau})+o_p(n^{-1/2}),
\]
where $\bm\epsilon_{t}=(\sigma_{1,\tau}^{-1}\bm\epsilon_{1,t}^\prime,\sigma_{2,\tau}^{-1}\bm\epsilon_{2,t}^\prime)^\prime$ and $H(\tau)=(\sigma_{1,\tau}^{-1}H_{1}^\prime(\tau),\sigma_{2,\tau}^{-1}H_{2}^\prime(\tau))^\prime$. Note that $E[\psi_{\tau}(\eta_{t,\tau})]=0$ and $\var[\psi_{\tau}(\eta_{t,\tau})]=\tau-\tau^2$. 
These together with the law of iterated expectations, we complete the proof by the central limit theorem and the Cram\'{e}r-Wold device. 
\end{proof}


The following two preliminary lemmas are used to prove Theorem \ref{thm2}.
Specifically, Lemma \ref{lem1} verifies the stochastic differentiability condition defined by
\cite{Pollard1985}, and the bracketing method in \cite{Pollard1985} is used for its proof. Lemma \ref{lem2} is used to obtain the $\sqrt{n}$-consistency and
the asymptotic normality of $\widehat{\bm\theta}_{\tau n}$, and its proof needs Lemma \ref{lem1}.

\begin{lemma}\label{lem1}
	Under Assumptions \ref{assum2}-\ref{assum4}, then for $\bm\theta_{\tau}-\bm\theta_{\tau 0}=o_p(1)$, it holds that	
	\begin{align*}
	\zeta_n(\bm\theta_{\tau})=o_p(\sqrt{n}\|\bm\theta_{\tau}-\bm\theta_{\tau 0}\|+n\|\bm\theta_{\tau}-\bm\theta_{\tau 0}\|^2),
	\end{align*}
	where $\zeta_n(\bm\theta_{\tau})=\sum_{t=p+1}^{n}\omega_tq_{1t}(\bm\theta_{\tau})\left\{\xi_{1t}(\bm\theta_{\tau})-E[\xi_{1t}(\bm\theta_{\tau})|\mathcal{F}_{t-1}]\right\}$ with $q_{1t}(\bm\theta_{\tau})=(\bm\theta_{\tau}-\bm\theta_{\tau 0})^{\prime}\dot{q}_t(\bm\theta_{\tau 0})$, and
    \begin{align*}
	\xi_{1t}(\bm\theta_{\tau})&=\int_{0}^{1}\left[I(y_t\leq F_{t-1}^{-1}(\tau)+q_{1t}(\bm\theta_{\tau})s)-I(y_t\leq F_{t-1}^{-1}(\tau))\right]ds.
    \end{align*}
\end{lemma}

\begin{proof}
Note that
\begin{align*}
|\zeta_n(\bm\theta_{\tau})| \leq \sqrt{n}\|\bm\theta_{\tau}-\bm\theta_{\tau 0}\|\sum_{j=1}^{2p+1}\left|\dfrac{1}{\sqrt{n}}\sum_{t=p+1}^{n}m_{t,j}\left\{\xi_{1t}(\bm\theta_{\tau})-E[\xi_{1t}(\bm\theta_{\tau})|\mathcal{F}_{t-1}]\right\}\right|,
\end{align*}
where $m_{t,j}=\omega_t\partial q_t(\bm\theta_{\tau 0})/\partial \theta_{\tau, j}$ with $\theta_{\tau, j}$ being the $j$th element of $\bm\theta_{\tau}$. For $1\leq j\leq 2p+1$, define $g_t=\max_{j}\{m_{t,j},0\}$ or $g_t=\max_j\{-m_{t,j},0\}$. Denote $\bm u=\bm\theta_{\tau}-\bm\theta_{\tau 0}$. Let $f_t(\bm u)=g_t\xi_{1t}(\bm\theta_{\tau})$ and define
\begin{align*}
D_n(\bm u)=\dfrac{1}{\sqrt{n}}\sum_{t=p+1}^{n}\left\{f_t(\bm u)-E\left[f_t(\bm u)|\mathcal{F}_{t-1}\right]\right\}.
\end{align*}
To establish Lemma \ref{lem1}, it suffices to show that, for any $\delta>0$,
\begin{align}\label{diffcond}
\sup_{\|\bm u\|\leq \delta}\dfrac{|D_n(\bm u)|}{1+\sqrt{n}\|\bm u\|}=o_p(1).
\end{align}

We follow the method in Lemma 4 of \cite{Pollard1985} to verify \eqref{diffcond}.
Let $\mathfrak{F}=\{f_t(\bm u): \|\bm u\|\leq \delta\}$ be a collection of functions indexed by $\bm u$. First, we verify that $\mathfrak{F}$ satisfies the bracketing condition defined on page 304 of \cite{Pollard1985}. 	
Let $B_{r}(\bm v)$ be an open neighborhood of $\bm v$ with radius $r>0$, and define a constant $C_0$ to be selected later. For any $\epsilon>0$ and $0< r\leq \delta$, there exists a sequence of small cubes $\{B_{\epsilon r/C_0}(\bm u_{i})\}_{i=1}^{K(\epsilon)}$ to cover $B_r(\bm 0)$, where $K(\epsilon)$ is an integer less than $C\epsilon^{-(2p+1)}$, and the constant $C$ is not depending on $\epsilon$ and $r$; see \cite{Huber1967}, page 227.
Denote $V_i(r)=B_{\epsilon r/C_0}(\bm u_{i})\bigcap B_r(\bm0)$, and let $U_1(r)=V_1(r)$ and $U_i(r)=V_i(r)-\bigcup_{j=1}^{i-1}V_j(r)$ for $i\geq 2$. Note that $\{U_i(r)\}_{i=1}^{K(\epsilon)}$ is a partition of $B_r(\bm0)$.
For each $\bm u_i\in U_i(r)$ with $1\leq i \leq K(\epsilon)$, define the following bracketing functions	
\begin{align*}
f_t^L(\bm u_i)&= g_t\int_{0}^{1}\left[I\left(y_t\leq F_{t-1}^{-1}(\tau)+\bm u_i^{\prime}\dot{q}_t(\bm\theta_{\tau 0})s-\dfrac{\epsilon r}{C_0}\|\dot{q}_t(\bm\theta_{\tau 0})\|\right)-I(y_t\leq F_{t-1}^{-1}(\tau))\right]ds, \\
f_t^U(\bm u_i)&= g_t\int_{0}^{1}\left[I\left(y_t\leq F_{t-1}^{-1}(\tau)+\bm u_i^{\prime}\dot{q}_t(\bm\theta_{\tau 0})s+\dfrac{\epsilon r}{C_0}\|\dot{q}_t(\bm\theta_{\tau 0})\|\right)-I(y_t\leq F_{t-1}^{-1}(\tau))\right]ds.
\end{align*}
Since the indicator function $I(\cdot)$ is non-decreasing and $g_t\geq 0$, for any $\bm u \in U_i(r)$, we have
\begin{align}\label{brac1}
f_t^L(\bm u_i)\leq f_t(\bm u)\leq f_t^U(\bm u_i).
\end{align}
Furthermore, by Taylor expansion, it holds that
\begin{align}\label{uncond}
E\left[f_t^U(\bm u_i)-f_i^L(\bm u_i)|\mathcal{F}_{t-1}\right]\leq \dfrac{\epsilon r}{C_0}\cdot2\sup_{x}f_{t-1}(x) \omega_t\left\|\dot{q}_t(\bm\theta_{\tau 0})\right\|^2.
\end{align}
Denote $\Delta_t=2\sup_{x}f_{t-1}(x)\omega_t\left\|\dot{q}_t(\bm\theta_{\tau 0})\right\|^2$. By Assumption \ref{assum4}, we have $\sup_{x}f_{t-1}(x)<\infty$. Choose $C_0=E(\Delta_t)$. Then by iterated-expectation and Assumption \ref{assum3}, it follows that
\begin{align*}
E\left[f_t^U(\bm u_i)-f_t^L(\bm u_i)\right]=E\left\{E\left[f_t^U(\bm\theta_i)-f_t^L(\bm\theta_i)|\mathcal{F}_{t-1}\right]\right\}\leq \epsilon r.
\end{align*}	
This together with \eqref{brac1}, implies that the family $\mathfrak{F}$ satisfies the bracketing condition.	

Put $r_k=2^{-k}\delta$. Let $B(k)=B_{r_k}(\bm0)$ and $A(k)$ be the annulus $B(k)\setminus B(k+1)$. From the bracketing condition, for fixed $\epsilon>0$, there is a partition $U_1(r_k), U_2(r_k), \ldots, U_{K(\epsilon)}(r_k)$ of $B(k)$. First, consider the upper tail case. For $\bm u \in U_i(r_k)$, by \eqref{uncond}, it holds that
\begin{align}\label{upper}
D_n(\bm u) \leq &\dfrac{1}{\sqrt{n}}\sum_{t=p+1}^{n}\left\{f_t^U(\bm u_i)-E\left[f_t^U(\bm u_i)|\mathcal{F}_{t-1}\right]\right\}+\dfrac{1}{\sqrt{n}}\sum_{t=p+1}^{n}E\left[f_t^U(\bm u_i)-f_t^L(\bm u_i)|\mathcal{F}_{t-1}\right] \nonumber \\
\leq & D_n^U(\bm u_i)+\sqrt{n}\epsilon r_k\dfrac{1}{nC_0}\sum_{t=p+1}^{n}\Delta_t,
\end{align}	
where \[D_n^U(\bm u_i)=\dfrac{1}{\sqrt{n}}\sum_{t=p+1}^{n}\left\{f_t^U(\bm u_i)-E\left[f_t^U(\bm u_i)|\mathcal{F}_{t-1}\right]\right\}.\]
Define the event
\begin{align*}
E_n=\left\{\omega: \dfrac{1}{nC_0}\sum_{t=p+1}^{n}\Delta_t(\omega) < 2 \right\}.
\end{align*}

For $\bm u \in A(k)$, $1+\sqrt{n}\|\bm u\|>\sqrt{n}r_{k+1}=\sqrt{n}r_{k}/2$. Then by \eqref{upper} and the Chebyshev's inequality, we have
\begin{align}\label{Ak0}
 \text{P}\left(\sup_{\bm u \in A(k)}\dfrac{D_n(\bm u)}{1+\sqrt{n}\|\bm u\|}>6\epsilon, E_n\right)
\leq & \text{P}\left(\max_{1 \leq i \leq K(\epsilon)}\sup_{\bm u \in U_i(r_k) \cap A(k)}D_n(\bm u)>3\sqrt{n}\epsilon r_k, E_n\right) \nonumber\\
\leq & K(\epsilon)\max_{1 \leq i \leq K(\epsilon)}\text{P}\left(D_n^U(\bm u_i)>\sqrt{n}\epsilon r_k\right) \nonumber\\
\leq & K(\epsilon)\max_{1 \leq i \leq K(\epsilon)}\dfrac{E\{[D_n^U(\bm u_i)]^2\}}{n\epsilon^2 r_k^2}.
\end{align}
Moreover, by iterated-expectation, Taylor expansion, Assumption \ref{assum4} and $\|\bm u_i\|\leq r_k$ for $\bm u_i \in U_i(r_k)$, we have
\begin{align*}	
& E\left\{[f_t^U(\bm u_i)]^2\right\}
= E\left\{E\left\{[f_t^U(\bm u_i)]^2|\mathcal{F}_{t-1}\right\}\right\} \nonumber \\
\leq & 2E\left\{g_t^2\left| \int_{0}^{1}\left[F_{t-1}\left(F_{t-1}^{-1}(\tau)+\bm u_i^{\prime}\dot{q}_t(\bm\theta_{\tau 0})s+\dfrac{\epsilon r_k}{C_0}\|\dot{q}_t(\bm\theta_{\tau 0})\|\right)-F_{t-1}\left(F_{t-1}^{-1}(\tau)\right)\right]ds\right|\right\} \nonumber\\
\leq & C\sup_{x}f_{t-1}(x)r_kE\left[\omega_t^2\left\|\dot{q}_t(\bm\theta_{\tau 0})\right\|^3\right].
\end{align*}
This, together with $E[w_{t}\|\bm y_{1,t-1}\|^3]<\infty$ by Assumption \ref{assum3}, $\sup_{x}f_{t-1}(x)<\infty$ by Assumption \ref{assum4} and the fact that $f_t^U(\bm u_i)-E[f_t^U(\bm u_i)|\mathcal{F}_{t-1}]$ is a martingale difference sequence, implies that
\begin{align}\label{ED}
E\{[D_n^U(\bm u_i)]^2\}&=\dfrac{1}{n}\sum_{t=p+1}^{n}E\{\{f_t^U(\bm u_i)-E[f_t^U(\bm u_i)|\mathcal{F}_{t-1}]\}^2\} \nonumber \\
&\leq \dfrac{1}{n}\sum_{t=p+1}^{n}E\{[f_t^U(\bm u_i)]^2\} \nonumber \\
&\leq \dfrac{Cr_k}{n}\sum_{t=p+1}^{n}E\left[\omega_t^2\left\|\dot{q}_t(\bm\theta_{\tau 0})\right\|^3\right]:=\Delta(r_k).
\end{align}
Combining \eqref{Ak0} and \eqref{ED}, we have
\begin{align*}
\text{P}\left(\sup_{\bm u \in A(k)}\dfrac{D_n(\bm u)}{1+\sqrt{n}\|\bm u\|}>6\epsilon, E_n\right)
\leq \dfrac{K(\epsilon)\Delta(r_k)}{n\epsilon^2r_k^2}.
\end{align*}
Similar to the proof of the upper tail case, we can obtain the same bound for the lower tail case. Therefore,
\begin{align}\label{Ak}
& \text{P}\left(\sup_{\bm u \in A(k)}\dfrac{|D_n(\bm u)|}{1+\sqrt{n}\|\bm u\|}>6\epsilon, E_n\right) \leq \dfrac{2K(\epsilon)\Delta(r_k)}{n\epsilon^2r_k^2}.
\end{align}

Note that $\Delta(r_k)\to 0$ as $k\to \infty$, we can choose $k_{\epsilon}$ such that $2K(\epsilon)\Delta(r_k)/(\epsilon^2\delta^2)<\epsilon$ for $k\geq k_{\epsilon}$. Let $k_n$ be the integer such that $n^{-1/2}\delta \leq r_{k_n} \leq 2n^{-1/2}\delta$, and split $B_{\delta}(\bm 0)$ into two events $B:=B(k_n+1)$ and $B^c:=B(0)-B(k_n+1)$. Note that $B^c=\bigcup_{k=0}^{k_n}A(k)$ and $\Delta(r_k)$ is bounded by Assumption \ref{assum3}. Then by \eqref{Ak}, it holds that
\begin{align}\label{Bc}
\text{P}\left(\sup_{\bm u \in B^c}\dfrac{|D_n(\bm u)|}{1+\sqrt{n}\|\bm u\|}>6\epsilon\right)
\leq & \sum_{k=0}^{k_n}\text{P}\left(\sup_{\bm u \in A(k)}\dfrac{|D_n(\bm u)|}{1+\sqrt{n}\|\bm u\|}>6\epsilon, E_n\right) + \text{P}(E_n^c)\nonumber \\
\leq & \dfrac{1}{n}\sum_{k=0}^{k_{\epsilon}-1}\dfrac{CK(\epsilon)}{\epsilon^2\delta^2}2^{2k}+ \dfrac{\epsilon}{n}\sum_{k=k_{\epsilon}}^{k_n}2^{2k}+ \text{P}(E_n^c) \nonumber \\
\leq & O\left(\dfrac{1}{n}\right) + 4\epsilon + \text{P}(E_n^c).
\end{align}

Furthermore, for $\bm u \in B$, we have $1+\sqrt{n}\|\bm u\|\geq 1$ and $r_{k_n+1}\leq n^{-1/2}\delta<n^{-1/2}$. Similar to the proof of \eqref{Ak0} and \eqref{ED}, we can show that
\begin{align*}
\text{P}\left(\sup_{\bm u \in B}\dfrac{D_n(\bm u)}{1+\sqrt{n}\|\bm u\|}>3\epsilon, E_n\right) \leq \text{P}\left(\max_{1 \leq i \leq K(\epsilon)}D_n^U(\bm u_i)>\epsilon, E_n\right) \leq \dfrac{K(\epsilon)\Delta(r_{k_n+1})}{\epsilon^2}.
\end{align*}
We can obtain the same bound for the lower tail. Therefore, we have
\begin{align}\label{B}
\text{P}\left(\sup_{\bm u \in B}\dfrac{|D_n(\bm u)|}{1+\sqrt{n}\|\bm u\|}>3\epsilon\right)
= &  \text{P}\left(\sup_{\bm u \in B}\dfrac{|D_n(\bm u)|}{1+\sqrt{n}\|\bm u\|}>3\epsilon, E_n\right)+ \text{P}(E_n^c) \nonumber \\
\leq & \dfrac{2K(\epsilon)\Delta(r_{k_n+1})}{\epsilon^2} + \text{P}(E_n^c).
\end{align}
Note that $\Delta(r_{k_n+1})\to 0$ as $n\to \infty$. Moreover, by the ergodic theorem, $\text{P}(E_n)\rightarrow 1$ and thus $\text{P}(E_n^c)\rightarrow 0$ as $n\rightarrow \infty$. \eqref{B} together with \eqref{Bc} asserts \eqref{diffcond}. The proof of this lemma is accomplished.	
\end{proof}

\begin{lemma}\label{lem2}
	Suppose that Assumptions \ref{assum2}-\ref{assum4} hold, then
	\begin{align*}
	n[L_n(\bm\theta_{\tau})-L_n(\bm\theta_{\tau 0})]=&-\sqrt{n}(\bm\theta_{\tau}-\bm\theta_{\tau 0})^{\prime}\bm T_n+\sqrt{n}(\bm\theta_{\tau}-\bm\theta_{\tau 0})^{\prime}J_n\sqrt{n}(\bm\theta_{\tau}-\bm\theta_{\tau 0}) \\
	&+o_p(\sqrt{n}\|\bm\theta_{\tau}-\bm\theta_{\tau 0}\|+n\|\bm\theta_{\tau}-\bm\theta_{\tau 0}\|^2)
	\end{align*}
	for $\bm\theta_{\tau}-\bm\theta_{\tau 0}=o_p(1)$, where $L_n(\bm\theta_{\tau})=n^{-1}\sum_{t=p+1}^{n}\omega_t\rho_{\tau}(y_t-q_t(\bm\theta_{\tau}))$, and	
	\[\bm T_n=\dfrac{1}{\sqrt{n}}\sum_{t=p+1}^{n}\omega_t\dot{q}_t(\bm\theta_{\tau 0})\psi_{\tau}(\eta_{t,\tau}) \hspace{2mm}\text{and}\hspace{2mm} J_n=\dfrac{1}{2n}\sum_{t=p+1}^{n}f_{t-1}(F_{t-1}^{-1}(\tau))\omega_t\dot{q}_t(\bm\theta_{\tau 0})\dot{q}_t^{\prime}(\bm\theta_{\tau 0})\]
	with $\eta_{t,\tau}=y_t-q_t(\bm\theta_{\tau 0})$.	
\end{lemma}


\begin{proof}
Denote $\bm u=\bm\theta_{\tau}-\bm\theta_{\tau 0}$. Let $\nu_t(\bm u)=q_t(\bm\theta_{\tau})-q_t(\bm\theta_{\tau 0})$,
and define the function
$\xi_t(\bm u)=\int_{0}^{1}\left[I(y_{t}\leq F_{t-1}^{-1}(\tau)+\nu_t(\bm u)s)-I(y_{t}\leq F_{t-1}^{-1}(\tau))\right]ds.$
Recall that $L_n(\bm\theta_{\tau})=n^{-1}\sum_{t=p+1}^{n}\omega_t\rho_{\tau}(y_t-q_t(\bm\theta_{\tau}))$ and $q_t(\bm\theta_{\tau 0})=F_{t-1}^{-1}(\tau)$.
By the Knight identity \eqref{identity}, it can be verified that
\begin{align}\label{Gnrep}
n[L_n(\bm\theta_{\tau})-L_n(\bm\theta_{\tau 0})] =&\sum_{t=p+1}^{n}\omega_t\left[\rho_{\tau}\left(\eta_{t,\tau}-\nu_t(\bm u)\right)-\rho_{\tau}\left(\eta_{t,\tau}\right)\right] \nonumber \\
=& K_{1n}(\bm u)+K_{2n}(\bm u),
\end{align}	
where
\begin{align*}
K_{1n}(\bm u)=-\sum_{t=p+1}^{n}\omega_t\nu_t(\bm u)\psi_{\tau}(\eta_{t,\tau}) \hspace{2mm}\text{and}\hspace{2mm}  K_{2n}(\bm u)=\sum_{t=p+1}^{n}\omega_t\nu_t(\bm u)\xi_t(\bm u).
\end{align*}
By Taylor expansion, we have $\nu_t(\bm u)=q_{1t}(\bm u)+q_{2t}(\bm u)$, where $q_{1t}(\bm u)=\bm u^{\prime}\dot{q}_t(\bm\theta_{\tau 0})$ and $q_{2t}(\bm u)=\bm u^{\prime}\ddot{q}_t(\bm\theta_{\tau}^*)\bm u/2$ for $\bm\theta_{\tau}^*$ between $\bm\theta_{\tau}$ and $\bm\theta_{\tau 0}$, and $\ddot{q}_t(\bm\theta_{\tau})$ is defined as \eqref{ddotqt}.
Then it follows that
\begin{align}\label{K1rep}
K_{1n}(\bm u)&=-\sum_{t=p+1}^{n}\omega_tq_{1t}(\bm u)\psi_{\tau}(\eta_{t,\tau})-\sum_{t=p+1}^{n}\omega_tq_{2t}(\bm u)\psi_{\tau}(\eta_{t,\tau}) \nonumber \\
&=-\sqrt{n}\bm u^{\prime}\bm T_n-\sqrt{n}\bm u^{\prime}R_{1n}(\bm\theta_{\tau}^*)\sqrt{n}\bm u,
\end{align}
where
\[\bm T_n=\dfrac{1}{\sqrt{n}}\sum_{t=p+1}^{n}\omega_t\dot{q}_t(\bm\theta_{\tau 0})\psi_{\tau}(\eta_{t,\tau}) \hspace{2mm}\text{and}\hspace{2mm} R_{1n}(\bm\theta_{\tau}^*)=\dfrac{1}{2n}\sum_{t=p+1}^{n}\omega_t\ddot{q}_t(\bm\theta_{\tau}^*)\psi_{\tau}(\eta_{t,\tau}).\]
From $E[w_{t}\|\bm y_{1,t-1}\|^3]<\infty$ by Assumption \ref{assum3} and the fact that $|\psi_{\tau}(\eta_{t,\tau})|\leq 1$, we have
\[E[\sup_{\bm\theta_{\tau}^*\in\Theta_{\tau}}\left\|\omega_t \ddot{q}_t(\bm\theta_{\tau}^*)\psi_{\tau}(\eta_{t,\tau})\right\|]\leq CE[\sup_{\bm\theta_{\tau}^*\in\Theta_{\tau}}\left\|\omega_t \ddot{q}_t(\bm\theta_{\tau}^*)\right\|]<\infty.\]
Moreover, by iterated-expectation and the fact that $E[\psi_{\tau}(\eta_{t,\tau})]=0$, it follows that
\[E\left[\omega_t \ddot{q}_t(\bm\theta_{\tau}^*)\psi_{\tau}(\eta_{t,\tau})\right]=0.\]
Then by Theorem 3.1 in \cite{Ling_McAleer2003}, we can show that
\[\sup_{\bm\theta_{\tau}^*\in\Theta_{\tau}}\|R_{1n}(\bm\theta_{\tau}^*)\|=o_p(1).\]
This together with \eqref{K1rep}, implies that
\begin{align}\label{Kn1}
K_{1n}(\bm u)=-\sqrt{n}\bm u^{\prime}\bm T_n+o_p(n\|\bm u\|^2).
\end{align}
Denote $\xi_{t}(\bm u)=\xi_{1t}(\bm u)+\xi_{2t}(\bm u)$, where
\begin{align*}
\xi_{1t}(\bm u)&=\int_{0}^{1}\left[I(y_{t}\leq F_{t-1}^{-1}(\tau)+q_{1t}(\bm u)s)-I(y_{t}\leq F_{t-1}^{-1}(\tau))\right]ds \hspace{2mm}\text{and}\hspace{2mm} \\
\xi_{2t}(\bm u)&=\int_{0}^{1}\left[I(y_{t}\leq F_{t-1}^{-1}(\tau)+\nu_t(\bm u)s)-I(y_{t}\leq F_{t-1}^{-1}(\tau)+q_{1t}(\bm u)s)\right]ds.
\end{align*}
Then for $K_{2n}(\bm\theta_{\tau})$, it holds that
\begin{align}\label{K2rep}
K_{2n}(\bm u)=R_{2n}(\bm u)+R_{3n}(\bm u)+R_{4n}(\bm u)+R_{5n}(\bm u),
\end{align}
where
\begin{align*}
R_{2n}(\bm u)&=\bm u^{\prime}\sum_{t=p+1}^{n}\omega_t\dot{q}_t(\bm\theta_{\tau 0})E[\xi_{1t}(\bm u)|\mathcal{F}_{t-1}], \\
R_{3n}(\bm u)&=\bm u^{\prime}\sum_{t=p+1}^{n}\omega_t\dot{q}_t(\bm\theta_{\tau 0})\{\xi_{1t}(\bm u)-E[\xi_{1t}(\bm u)|\mathcal{F}_{t-1}]\}, \\
R_{4n}(\bm u)&=\bm u^{\prime}\sum_{t=p+1}^{n}\omega_t\dot{q}_t(\bm\theta_{\tau 0})\xi_{2t}(\bm u) \hspace{2mm}\text{and}\hspace{2mm}
R_{5n}(\bm u)=\dfrac{\bm u^{\prime}}{2}\sum_{t=p+1}^{n}\omega_t\ddot{q}_t(\bm\theta_{\tau}^*)\xi_{t}(\bm u)\bm u.
\end{align*}
Note that
\begin{align}\label{xi1}
E[\xi_{1t}(\bm u)|\mathcal{F}_{t-1}]=\int_{0}^{1}[F_{t-1}(F_{t-1}^{-1}(\tau)+q_{1t}(\bm u)s)-F_{t-1}(F_{t-1}^{-1}(\tau))]ds.
\end{align}
Then by Taylor expansion, together with Assumption \ref{assum4}, it follows that
\begin{align*}
E[\xi_{1t}(\bm u)|\mathcal{F}_{t-1}]=&\dfrac{1}{2}f_{t-1}(F_{t-1}^{-1}(\tau))q_{1t}(\bm u) \\
&+q_{1t}(\bm u)\int_{0}^{1}[f_{t-1}(F_{t-1}^{-1}(\tau)+q_{1t}(\bm u)s^*)-f_{t-1}(F_{t-1}^{-1}(\tau))]sds,
\end{align*}
where $s^*$ is between 0 and $s$. Therefore, it follows that
\begin{align}\label{R2nrep}
R_{2n}(\bm u)=\sqrt{n}\bm u^{\prime}J_n\sqrt{n}\bm u+\sqrt{n}\bm u^{\prime}\Pi_{1n}(\bm u)\sqrt{n}\bm u,
\end{align}
where
$J_n=(2n)^{-1}\sum_{t=p+1}^{n}f_{t-1}(F_{t-1}^{-1}(\tau))\omega_t\dot{q}_t(\bm\theta_{\tau 0})\dot{q}_t^{\prime}(\bm\theta_{\tau 0})$ and
\begin{align*}
\Pi_{1n}(\bm u)&=\dfrac{1}{n}\sum_{t=p+1}^{n}\omega_t\dot{q}_t(\bm\theta_{\tau 0})\dot{q}_t^{\prime}(\bm\theta_{\tau 0})\int_{0}^{1}[f_{t-1}(F_{t-1}^{-1}(\tau)+q_{1t}(\bm u)s^*)-f_{t-1}(F_{t-1}^{-1}(\tau))]sds.
\end{align*}
By Taylor expansion and Assumptions \ref{assum3}-\ref{assum4}, for any $\eta>0$, it holds that
\begin{align*}
E\left(\sup_{\|\bm u\|\leq \eta}\|\Pi_{1n}(\bm u)\|\right)&\leq \dfrac{1}{n}\sum_{t=p+1}^{n}E\left[\sup_{\|\bm u\|\leq \eta} \|\omega_t\dot{q}_t(\bm\theta_{\tau 0})\dot{q}_t^{\prime}(\bm\theta_{\tau 0})\sup_{x}|\dot{f}_{t-1}(x)|\bm u^{\prime}\dot{q}_t(\bm\theta_{\tau 0})\|\right] \\
&\leq C\eta\sup_{x}|\dot{f}_{t-1}(x)| E[\omega_t\|\dot{q}_t(\bm\theta_{\tau 0})\|^3]
\end{align*}
tends to $0$ as $\eta \to 0$.
Therefore, for any $\epsilon$, $\delta>0$, there exists $\eta_0=\eta_0(\epsilon)>0$ such that
\begin{align}\label{epsdelta1}
\text{Pr}\left(\sup_{\|\bm u\|\leq \eta_0}\|\Pi_{1n}(\bm u)\|> \delta\right)<\dfrac{\epsilon}{2}
\end{align}
for all $n\geq 1$. Since $\bm u=o_p(1)$, it follows that
\begin{align}\label{epsdelta2}
\text{Pr}\left(\|\bm u\|> \eta_0\right)<\dfrac{\epsilon}{2}
\end{align}
as $n$ is large enough. From \eqref{epsdelta1} and \eqref{epsdelta2}, we have
\begin{align*}
\text{Pr}\left(\|\Pi_{1n}(\bm u)\|> \delta\right)&\leq \text{Pr}\left(\|\Pi_{1n}(\bm u)\|> \delta, \|\bm u\|\leq \eta_0\right)+\text{Pr}\left(\|\bm u\|> \eta_0\right) \\
&\leq \text{Pr}\left(\sup_{\|\bm u\|\leq \eta_0}\|\Pi_{1n}(\bm u)\|> \delta\right)+\dfrac{\epsilon}{2}<\epsilon
\end{align*}
as $n$ is large enough. Therefore, $\Pi_{1n}(\bm u)=o_p(1)$. This together with \eqref{R2nrep}, implies that
\begin{align}\label{R2n}
R_{2n}(\bm u)=\sqrt{n}\bm u^{\prime}J_n\sqrt{n}\bm u+o_p(n\|\bm u\|^2).
\end{align}
For $R_{3n}(\bm u)$, by Lemma \ref{lem1}, it holds that
\begin{align}\label{R3n}
R_{3n}(\bm u)=o_p(\sqrt{n}\|\bm u\|+n\|\bm u\|^2).
\end{align}	
Note that
\begin{align}\label{xi2}
E[\xi_{2t}(\bm u)|\mathcal{F}_{t-1}]=\int_{0}^{1}\left[F_{t-1}(F_{t-1}^{-1}(\tau)+\nu_t(\bm u)s)-F_{t-1}(F_{t-1}^{-1}(\tau)+q_{1t}(\bm u)s)\right]ds.
\end{align}
Then by iterated-expectation, Taylor expansion and the Cauchy-Schiwarz inequality, together with $E[w_{t}\|\bm y_{1,t-1}\|^3]<\infty$ by Assumption \ref{assum3} and $\sup_{x}|\dot{f}_{t-1}(x)|<\infty$ by Assumption \ref{assum4}, for any $\eta>0$, it holds that
\begin{align*}
E\bigg(\sup_{\|\bm u\|\leq \eta}\dfrac{|R_{4n}(\bm u)|}{n\|\bm u\|^2}\bigg)
\leq & \dfrac{\eta}{n}\sum_{t=p+1}^{n}E\left\{\omega_t\left\|\dot{q}_t(\bm\theta_{\tau 0})\right\|\dfrac{1}{2}\sup_{x}f_{t-1}(x)\sup_{\Theta_{\tau}}\left\|\ddot{q}_t(\bm\theta_{\tau})\right\| \right\} \\
\leq &  C\eta E\left\{\left\|\sqrt{\omega_t}\dot{q}_t(\bm\theta_{\tau 0})\right\|\sup_{\Theta_{\tau}}\left\|\sqrt{\omega_t}\ddot{q}_t(\bm\theta_{\tau})\right\| \right\} \\
\leq & C\eta \left[E\left(\omega_t\left\|\dot{q}_t(\bm\theta_{\tau 0})\right\|^2\right)\right]^{1/2}\left[E\left(\sup_{\Theta_{\tau}}\omega_t\left\|\ddot{q}_t(\bm\theta_{\tau})\right\|^2\right)\right]^{1/2}
\end{align*}
tends to $0$ as $\eta \to 0$. Similar to \eqref{epsdelta1} and \eqref{epsdelta2}, we can show that
\begin{align}\label{R4n}
R_{4n}(\bm u)=o_p(n\|\bm u\|^2).
\end{align}
For $R_{5n}(\bm u)$, it follows that
\begin{align}\label{R5nrep}
R_{5n}(\bm u)=\sqrt{n}\bm u^{\prime}\Pi_{2n}(\bm u)\sqrt{n}\bm u+\sqrt{n}\bm u^{\prime}\Pi_{3n}(\bm u)\sqrt{n}\bm u.
\end{align}
where
\[\Pi_{2n}(\bm u)=\dfrac{1}{2n}\sum_{t=p+1}^{n}\omega_t\ddot{q}_t(\bm\theta_{\tau}^*)\xi_{1t}(\bm u)  \hspace{2mm}\text{and}\hspace{2mm} \Pi_{3n}(\bm u)=\dfrac{1}{2n}\sum_{t=p+1}^{n}\omega_t\ddot{q}_t(\bm\theta_{\tau}^*)\xi_{2t}(\bm u).\]
Similar to the proof of $R_{4n}(\bm u)$, by \eqref{xi1} we can show that, for any $\eta>0$,
\begin{align*}
E\left(\sup_{\|\bm u\|\leq \eta}\|\Pi_{2n}(\bm u)\|\right)
\leq & \dfrac{\eta}{n}\sum_{t=p+1}^{n} E\left\{\omega_t\sup_{\Theta_{\tau}}\left\|\ddot{q}_t(\bm\theta_{\tau})\right\|\dfrac{1}{2}\sup_{x}f_{t-1}(x)\left\|\dot{q}_t(\bm\theta_{\tau 0})\right\| \right\} \\
\leq & C\eta \left[E\left(\sup_{\Theta_{\tau}}\omega_t\left\|\ddot{q}_t(\bm\theta_{\tau})\right\|^2\right)\right]^{1/2}\left[E\left(\omega_t\left\|\dot{q}_t(\bm\theta_{\tau 0})\right\|^2\right)\right]^{1/2}
\end{align*}
tends to $0$ as $\eta \to 0$. And for $\Pi_{3n}(\bm u)$, by \eqref{xi2} we have
\begin{align*}
E\bigg(\sup_{\|\bm u\|\leq \eta}\|\Pi_{3n}(\bm u)\|\bigg)
\leq & \eta^2 E\left\{\dfrac{1}{n}\sum_{t=p+1}^{n}\omega_t\sup_{\Theta_{\tau}}\left\|\ddot{q}_t(\bm\theta_{\tau})\right\|\dfrac{1}{2}\sup_{x\in\mathbb{R}}f_{t-1}(x)\sup_{\Theta_{\tau}}\left\|\ddot{q}_t(\bm\theta_{\tau})\right\| \right\} \\
\leq &  C\eta^2 E\left(\sup_{\Theta_{\tau}}\omega_t\left\|\ddot{q}_t(\bm\theta_{\tau})\right\|^2 \right)
\end{align*}
tends to $0$ as $\eta \to 0$. 	
Similar to \eqref{epsdelta1} and \eqref{epsdelta2}, we can show that $\Pi_{2n}(\bm u)=o_p(1)$ and $\Pi_{3n}(\bm u)=o_p(1)$.
Therefore, together with \eqref{R5nrep}, it follows that
\begin{align}\label{R5n}
R_{5n}(\bm u)=o_p(n\|\bm u\|^2).
\end{align}
From \eqref{K2rep}, \eqref{R2n}-\eqref{R4n} and \eqref{R5n}, we have
\begin{align}\label{Kn2}
K_{2n}(\bm u)=\sqrt{n}\bm u^{\prime}J_n\sqrt{n}\bm u+o_p(\sqrt{n}\|\bm u\|+n\|\bm u\|^2).
\end{align}
In view of \eqref{Gnrep}, \eqref{Kn1} and \eqref{Kn2}, we accomplish the proof of this lemma.
\end{proof}

\bibliography{QuantileDAR}

\renewcommand{\baselinestretch}{1.2}
\clearpage
\begin{table}[h!]
\begin{center}
\caption{\label{table1a}Biases, empirical standard deviations (ESDs) and asymptotic standard deviations (ASDs) of $\widehat{\bm\theta}_{\tau n}$ at quantile level $\tau=0.05$ or $0.25$ for model \eqref{sim1} with coefficient functions \eqref{sim1coef1}, where $\text{ASD}_{1}$ and $\text{ASD}_{2}$ correspond to the bandwidths $h_{B}$ and $h_{HS}$, respectively. The $F(\cdot)$ is the distribution function of the normal, the Student's $t_5$ or the Student's $t_3$ distribution.}\vspace{2mm}	
\begin{tabular}{clrrrrrrrrrrr}
\hline\hline
&&\multicolumn{5}{c}{$\tau=0.05$}&&\multicolumn{5}{c}{$\tau=0.25$}\\
\cline{3-7}\cline{9-13}	
$n$&&\multicolumn{1}{c}{True}&\multicolumn{1}{c}{Bias}&\multicolumn{1}{c}{ESD}&\multicolumn{1}{c}{$\text{ASD}_{1}$}&\multicolumn{1}{c}{$\text{ASD}_{2}$}&&\multicolumn{1}{c}{True}&\multicolumn{1}{c}{Bias}&\multicolumn{1}{c}{ESD}&\multicolumn{1}{c}{$\text{ASD}_{1}$}&\multicolumn{1}{c}{$\text{ASD}_{2}$}\\
\hline
&&\multicolumn{11}{c}{Normal distribution}          \\
$500$	&	$b$	&	-2.706	&	-0.015	&	0.489	&	0.743	&	0.529	&&	-0.455	&	-0.004	&	 0.134	&	0.137	&	0.133	\\
	&	$\phi$	&	-0.200	&	-0.003	&	0.143	&	0.215	&	0.155	&&	-0.200	&	-0.003	&	 0.085	&	0.095	&	0.093	\\
	&	$\beta$	&	-1.082	&	-0.052	&	0.525	&	1.066	&	0.578	&&	-0.182	&	-0.011	&	 0.134	&	0.140	&	0.136	\\
$1000$	&	$b$	&	-2.706	&	-0.003	&	0.350	&	0.379	&	0.409	&&	-0.455	&	-0.007	&	 0.094	&	0.097	&	0.095	\\
	&	$\phi$	&	-0.200	&	-0.002	&	0.098	&	0.109	&	0.116	&&	-0.200	&	0.000	&	 0.064	&	0.066	&	0.065	\\
	&	$\beta$	&	-1.082	&	-0.029	&	0.370	&	0.398	&	0.495	&&	-0.182	&	-0.004	&	 0.096	&	0.099	&	0.096	\\
&&\multicolumn{11}{c}{Student's $t_5$ distribution}          \\
$500$	&	$b$	&	-4.060	&	-0.067	&	1.030	&	1.189	&	1.184	&&	-0.528	&	-0.008	&	 0.169	&	0.179	&	0.172	\\
	&	$\phi$	&	-0.200	&	-0.004	&	0.235	&	0.279	&	0.274	&&	-0.200	&	-0.001	&	 0.097	&	0.111	&	0.107	\\
	&	$\beta$	&	-1.624	&	-0.163	&	0.996	&	1.146	&	1.211	&&	-0.211	&	-0.017	&	 0.151	&	0.162	&	0.157	\\
$1000$	&	$b$	&	-4.060	&	-0.020	&	0.726	&	0.857	&	0.847	&&	-0.528	&	-0.009	&	 0.119	&	0.126	&	0.121	\\
	&	$\phi$	&	-0.200	&	-0.005	&	0.162	&	0.193	&	0.194	&&	-0.200	&	0.000	&	 0.073	&	0.077	&	0.074	\\
	&	$\beta$	&	-1.624	&	-0.099	&	0.683	&	0.855	&	0.879	&&	-0.211	&	-0.008	&	 0.108	&	0.113	&	0.109	\\
\hline
\end{tabular}
\end{center}
\end{table}

\begin{table}[h!]
\begin{center}
\caption{\label{table1b}Biases, empirical standard deviations (ESDs) and asymptotic standard deviations (ASDs) of $\widehat{\bm\theta}_{\tau n}$ at quantile level $\tau=0.05$ or $0.25$ for model \eqref{sim1} with coefficient functions \eqref{sim1coef2}, where $\text{ASD}_{1}$ and $\text{ASD}_{2}$ correspond to the bandwidths $h_{B}$ and $h_{HS}$, respectively.  The $F(\cdot)$ is the distribution function of the normal, the Student's $t_5$ or the Student's $t_3$ distribution. }\vspace{2mm}	
\begin{tabular}{clrrrrrrrrrrr}
\hline\hline
&&\multicolumn{5}{c}{$\tau=0.05$}&&\multicolumn{5}{c}{$\tau=0.25$}\\
\cline{3-7}\cline{9-13}	
$n$&&\multicolumn{1}{c}{True}&\multicolumn{1}{c}{Bias}&\multicolumn{1}{c}{ESD}&\multicolumn{1}{c}{$\text{ASD}_{1}$}&\multicolumn{1}{c}{$\text{ASD}_{2}$}&&\multicolumn{1}{c}{True}&\multicolumn{1}{c}{Bias}&\multicolumn{1}{c}{ESD}&\multicolumn{1}{c}{$\text{ASD}_{1}$}&\multicolumn{1}{c}{$\text{ASD}_{2}$}\\
\hline
&&\multicolumn{11}{c}{Normal distribution}          \\
$500$	&	$b$	&	-2.706	&	0.011	&	0.425	&	0.558	&	0.676	&&	-0.455	&	-0.001	&	 0.122	&	0.124	&	0.123	\\
	&	$\phi$	&	0.025	&	-0.006	&	0.117	&	0.146	&	0.166	&&	0.125	&	-0.004	&	 0.074	&	0.081	&	0.080	\\
	&	$\beta$	&	-0.068	&	-0.090	&	0.294	&	0.520	&	0.751	&&	-0.057	&	-0.013	&	 0.085	&	0.095	&	0.098	\\
$1000$	&	$b$	&	-2.706	&	0.014	&	0.291	&	0.349	&	0.326	&&	-0.455	&	-0.004	&	 0.082	&	0.089	&	0.085	\\
	&	$\phi$	&	0.025	&	-0.005	&	0.081	&	0.093	&	0.089	&&	0.125	&	-0.003	&	 0.054	&	0.057	&	0.055	\\
	&	$\beta$	&	-0.068	&	-0.046	&	0.185	&	0.275	&	0.245	&&	-0.057	&	-0.006	&	 0.060	&	0.068	&	0.064	\\
&&\multicolumn{11}{c}{Student's $t_5$ distribution}          \\
$500$	&	$b$	&	-4.060	&	0.010	&	0.869	&	1.316	&	1.028	&&	-0.528	&	-0.003	&	 0.153	&	0.157	&	0.155	\\
	&	$\phi$	&	0.025	&	-0.010	&	0.194	&	0.268	&	0.230	&&	0.125	&	-0.004	&	 0.082	&	0.091	&	0.090	\\
	&	$\beta$	&	-0.102	&	-0.238	&	0.538	&	1.222	&	0.748	&&	-0.066	&	-0.018	&	 0.091	&	0.099	&	0.105	\\
$1000$	&	$b$	&	-4.060	&	0.025	&	0.588	&	0.759	&	0.729	&&	-0.528	&	-0.005	&	 0.102	&	0.112	&	0.116	\\
	&	$\phi$	&	0.025	&	-0.009	&	0.131	&	0.164	&	0.158	&&	0.125	&	-0.002	&	 0.060	&	0.063	&	0.068	\\
	&	$\beta$	&	-0.102	&	-0.128	&	0.311	&	0.543	&	0.530	&&	-0.066	&	-0.008	&	 0.062	&	0.072	&	0.085	\\
\hline
\end{tabular}
\end{center}
\end{table}

\begin{table}[h!]
	\caption{\label{table-bic} Percentages of underfitted, correctly selected and  overfitted models by BIC based on 1000 replications for three sets of coefficient functions.  The $F(\cdot)$ is the distribution function of the normal, the Student's $t_5$ or the Student's $t_3$ distribution.}
	\begin{center}
		\begin{tabular}{crrrrrrrrrrrr}
			\hline\hline
			&&\multicolumn{3}{c}{Set (i)}&&\multicolumn{3}{c}{Set (ii)}&&\multicolumn{3}{c}{Set (iii)}\\
			\cline{3-5}\cline{7-9}\cline{11-13}	
			& $n$	&	Under &	Exact	& Over && Under &	Exact	& Over &&	Under &	Exact	& Over\\\hline
			Normal	&	500	&	3.7	&	95.3	&	1.0	&&	2.9	&	95.3	&	1.8	&&	4.3	&	93.7	&	2.0	\\
			&	1000	&	0.0	&	99.3	&	0.7	&&	0.0	&	99.7	&	0.3	&&	0.0	&	98.9	&	1.1	\\
			$t_5$	&	500	&	6.9	&	90.0	&	3.1	&&	5.3	&	92.1	&	2.6	&&	11.7	&	86.0	&	2.3	\\
			&	1000	&	0.3	&	98.1	&	1.6	&&	0.1	&	99.1	&	0.8	&&	0.1	&	98.5	&	1.4	\\
			$t_3$	&	500	&	11.6	&	81.7	&	6.7	&&	9.5	&	84.6	&	5.9	&&	17.8	&	80.4	&	1.8	\\
			&	1000	&	0.5	&	96.1	&	3.4	&&	1.2	&	96.2	&	2.6	&&	1.1	&	97.6	&	1.3	\\			
			\hline
		\end{tabular}	
	\end{center}
\end{table}

\begin{table}
\begin{center}
\caption{\label{table2a}Biases ($\times 100$), empirical standard deviations (ESDs) ($\times 100$) and asymptotic standard deviations (ASDs) ($\times 100$) of $\widehat{\rho}_{k,\tau}$ with $k=2,4$ or 6 for model \eqref{sim1coef1}. The quantile level is $\tau=0.05$ or $0.25$, and $\text{ASD}_{1}$ and $\text{ASD}_{2}$ correspond to the bandwidths $h_{B}$ and $h_{HS}$, respectively.  The $F(\cdot)$ is the distribution function of the normal, the Student's $t_5$ or the Student's $t_3$ distribution. }\vspace{2mm}	
\begin{tabular}{ccrrrrrrrrr}
\hline\hline
&&\multicolumn{4}{c}{$\tau=0.05$}&&\multicolumn{4}{c}{$\tau=0.25$}\\
\cline{3-6}\cline{8-11}	
$n$&\multicolumn{1}{c}{Lag}&\multicolumn{1}{c}{Bias}&\multicolumn{1}{c}{ESD}&\multicolumn{1}{c}{$\text{ASD}_{1}$}&\multicolumn{1}{c}{$\text{ASD}_{2}$}&&\multicolumn{1}{c}{Bias}&\multicolumn{1}{c}{ESD}&\multicolumn{1}{c}{$\text{ASD}_{1}$}&\multicolumn{1}{c}{$\text{ASD}_{2}$}\\
\hline
&&\multicolumn{9}{c}{Normal distribution}          \\
$500$	&	2	&	0.01	&	2.64	&	3.04	&	2.70	&&	-0.20	&	2.71	&	2.65	&	2.65	 \\
	&	4	&	0.04	&	3.03	&	5.60	&	3.01	&&	-0.11	&	3.06	&	3.00	&	3.00	 \\
	&	6	&	-0.01	&	3.07	&	3.33	&	3.06	&&	-0.17	&	3.10	&	3.04	&	3.04	 \\
$1000$	&	2	&	-0.06	&	1.94	&	1.89	&	2.07	&&	-0.13	&	1.94	&	1.87	&	1.87	 \\
	&	4	&	0.05	&	2.10	&	2.12	&	2.14	&&	-0.04	&	2.15	&	2.12	&	2.12	 \\
	&	6	&	-0.08	&	2.13	&	2.15	&	2.22	&&	-0.10	&	2.25	&	2.15	&	2.15	 \\
&&\multicolumn{9}{c}{Student's $t_5$ distribution}          \\
$500$	&	2	&	-0.01	&	2.15	&	2.13	&	2.17	&&	-0.12	&	2.11	&	2.04	&	2.05	 \\
	&	4	&	0.08	&	2.63	&	2.67	&	2.68	&&	-0.13	&	2.75	&	2.64	&	2.65	 \\
	&	6	&	0.00	&	2.84	&	2.81	&	2.83	&&	-0.14	&	2.89	&	2.80	&	2.81	 \\
$1000$	&	2	&	-0.01	&	1.50	&	1.53	&	1.50	&&	-0.06	&	1.44	&	1.42	&	1.42	 \\
	&	4	&	0.05	&	1.84	&	1.97	&	1.91	&&	-0.04	&	1.90	&	1.86	&	1.86	 \\
	&	6	&	-0.04	&	1.94	&	2.07	&	2.00	&&	-0.12	&	2.03	&	1.97	&	1.97	 \\
\hline
\end{tabular}
\end{center}
\end{table}

\begin{table}
\begin{center}
\caption{\label{table2c}Biases ($\times 100$), empirical standard deviations (ESDs) ($\times 100$) and asymptotic standard deviations (ASDs) ($\times 100$) of $\widehat{r}_{k,\tau}$ with $k=2,4$ or 6 for model \eqref{sim1coef1}. The quantile level is $\tau=0.05$ or $0.25$, and $\text{ASD}_{1}$ and $\text{ASD}_{2}$ correspond to the bandwidths $h_{B}$ and $h_{HS}$, respectively. The $F(\cdot)$ is the distribution function of the normal, the Student's $t_5$ or the Student's $t_3$ distribution. }\vspace{2mm}	
\begin{tabular}{ccrrrrrrrrr}
\hline\hline
&&\multicolumn{4}{c}{$\tau=0.05$}&&\multicolumn{4}{c}{$\tau=0.25$}\\
\cline{3-6}\cline{8-11}	
$n$&\multicolumn{1}{c}{Lag}&\multicolumn{1}{c}{Bias}&\multicolumn{1}{c}{ESD}&\multicolumn{1}{c}{$\text{ASD}_{1}$}&\multicolumn{1}{c}{$\text{ASD}_{2}$}&&\multicolumn{1}{c}{Bias}&\multicolumn{1}{c}{ESD}&\multicolumn{1}{c}{$\text{ASD}_{1}$}&\multicolumn{1}{c}{$\text{ASD}_{2}$}\\
\hline
&&\multicolumn{9}{c}{Normal distribution}          \\
$500$	&	2	&	0.02	&	2.67	&	3.09	&	2.73	&&	-0.10	&	2.72	&	2.69	&	2.69	 \\
	&	4	&	0.07	&	3.03	&	5.69	&	3.01	&&	-0.06	&	3.06	&	3.00	&	3.00	 \\
	&	6	&	0.03	&	3.09	&	3.33	&	3.06	&&	-0.04	&	3.03	&	3.03	&	3.03	 \\
$1000$	&	2	&	-0.06	&	1.96	&	1.91	&	2.09	&&	-0.07	&	1.99	&	1.90	&	1.90	 \\
	&	4	&	0.05	&	2.09	&	2.13	&	2.14	&&	-0.02	&	2.17	&	2.12	&	2.12	 \\
	&	6	&	-0.06	&	2.12	&	2.15	&	2.23	&&	-0.04	&	2.20	&	2.15	&	2.15	 \\
&&\multicolumn{9}{c}{Student's $t_5$ distribution}          \\
$500$	&	2	&	-0.01	&	2.22	&	2.17	&	2.21	&&	-0.06	&	2.06	&	2.03	&	2.04	 \\
	&	4	&	0.10	&	2.63	&	2.68	&	2.69	&&	-0.04	&	2.71	&	2.62	&	2.62	 \\
	&	6	&	0.04	&	2.85	&	2.81	&	2.83	&&	0.00	&	2.78	&	2.79	&	2.79	 \\
$1000$	&	2	&	-0.01	&	1.52	&	1.56	&	1.53	&&	-0.03	&	1.46	&	1.40	&	1.40	 \\
	&	4	&	0.04	&	1.84	&	1.97	&	1.91	&&	0.00	&	1.93	&	1.84	&	1.84	 \\
	&	6	&	-0.02	&	1.93	&	2.00	&	2.01	&&	-0.07	&	1.97	&	1.96	&	1.96	 \\
\hline
\end{tabular}
\end{center}
\end{table}

\begin{table}
\begin{center}
\caption{\label{table3a}Rejection rate (\%) of test statistics $Q_1(6)$, $Q_2(6)$ and $Q(6)$. The significance level is 5\%, and the quantile level is $\tau=0.05$ or $0.25$. The $F(\cdot)$ is the distribution function of the normal, the Student's $t_5$ or the Student's $t_3$ distribution.}\vspace{1mm}	
\begin{tabular}{clrrrrrrrr}
\hline\hline
&&&\multicolumn{3}{c}{$\tau=0.05$}&&\multicolumn{3}{c}{$\tau=0.25$}\\
\cline{4-6}\cline{8-10}	
$n$&$c_1$&$c_2$&\multicolumn{1}{c}{$Q_{1}$}&\multicolumn{1}{c}{$Q_{2}$}&\multicolumn{1}{c}{$Q$}&&\multicolumn{1}{c}{$Q_{1}$}&\multicolumn{1}{c}{$Q_{2}$}&\multicolumn{1}{c}{$Q$}\\
\hline
&&&\multicolumn{7}{c}{Normal distribution}          \\
500	&	0.0	&	0.0	&	3.5	&	3.7	&	3.6	&&	5.5	&	4.4	&	5.8	\\
&	0.1	&	0.0	&	8.3	&	7.3	&	7.4	&&	15.7	&	9.0	&	13.3	\\
&	0.3	&	0.0	&	47.3	&	41.8	&	45.6	&&	92.4	&	53.4	&	89.5	\\
&	0.0	&	0.1	&	6.3	&	6.2	&	6.6	&&	6.5	&	6.4	&	6.5	\\
&	0.0	&	0.3	&	13.7	&	14.1	&	13.7	&&	7.4	&	11.2	&	9.5	\\
1000	&	0.0	&	0.0	&	5.2	&	5.2	&	5.3	&&	4.7	&	6.2	&	5.7	\\
&	0.1	&	0.0	&	12.9	&	11.4	&	12.2	&&	28.1	&	16.0	&	25.1	\\
&	0.3	&	0.0	&	83.6	&	78.3	&	81.4	&&	99.9	&	90.2	&	99.8	\\
&	0.0	&	0.1	&	7.4	&	7.2	&	7.4	&&	5.7	&	6.5	&	5.7	\\
&	0.0	&	0.3	&	16.5	&	19.2	&	18.1	&&	6.2	&	16.3	&	11.8	\\
&&&\multicolumn{7}{c}{Student's $t_5$ distribution}          \\
500	&	0.0	&	0.0	&	5.6	&	5.5	&	5.6	&&	5.7	&	5.2	&	5.3	\\
&	0.1	&	0.0	&	7.8	&	6.1	&	6.8	&&	15.7	&	6.6	&	11.5	\\
&	0.3	&	0.0	&	32.4	&	19.1	&	26.9	&&	93.4	&	27.8	&	86.3	\\
&	0.0	&	0.1	&	11.8	&	12.3	&	12.2	&&	6.8	&	8.2	&	7.7	\\
&	0.0	&	0.3	&	31.1	&	32.0	&	32.9	&&	7.6	&	17.9	&	16.0	\\
1000	&	0.0	&	0.0	&	5.6	&	5.6	&	5.6	&&	4.6	&	6.5	&	5.2	\\
&	0.1	&	0.0	&	10.3	&	7.0	&	7.8	&&	30.0	&	11.0	&	23.4	\\
&	0.3	&	0.0	&	60.7	&	37.9	&	53.5	&&	99.8	&	56.8	&	99.7	\\
&	0.0	&	0.1	&	13.9	&	15.2	&	14.1	&&	6.4	&	10.5	&	9.1	\\
&	0.0	&	0.3	&	42.7	&	52.0	&	49.9	&&	8.6	&	30.0	&	23.2	\\
&&&\multicolumn{7}{c}{Student's $t_3$ distribution}          \\	
500	&	0.0	&	0.0	&	6.9	&	6.3	&	6.6	&&	5.3	&	4.6	&	5.2	\\
&	0.1	&	0.0	&	8.9	&	7.4	&	8.3	&&	14.8	&	6.1	&	10.7	\\
&	0.3	&	0.0	&	28.1	&	14.1	&	22.1	&&	91.9	&	14.4	&	82.8	\\
&	0.0	&	0.1	&	16.8	&	17.2	&	18.1	&&	7.5	&	9.0	&	8.5	\\
&	0.0	&	0.3	&	41.8	&	43.0	&	43.3	&&	7.2	&	20.6	&	17.5	\\
1000	&	0.0	&	0.0	&	5.8	&	6.3	&	6.5	&&	4.5	&	5.4	&	5.2	\\
&	0.1	&	0.0	&	10.0	&	7.2	&	9.0	&&	29.3	&	6.3	&	21.0	\\
&	0.3	&	0.0	&	50.5	&	22.3	&	39.8	&&	99.7	&	25.0	&	99.3	\\
&	0.0	&	0.1	&	24.5	&	27.4	&	27.0	&&	7.7	&	13.2	&	12.2	\\
&	0.0	&	0.3	&	60.3	&	63.8	&	63.6	&&	8.3	&	34.7	&	28.4	\\
\hline
\end{tabular}
\end{center}
\end{table}


\newpage

\begin{figure}[tbp]
\centering
\includegraphics[scale=0.45]{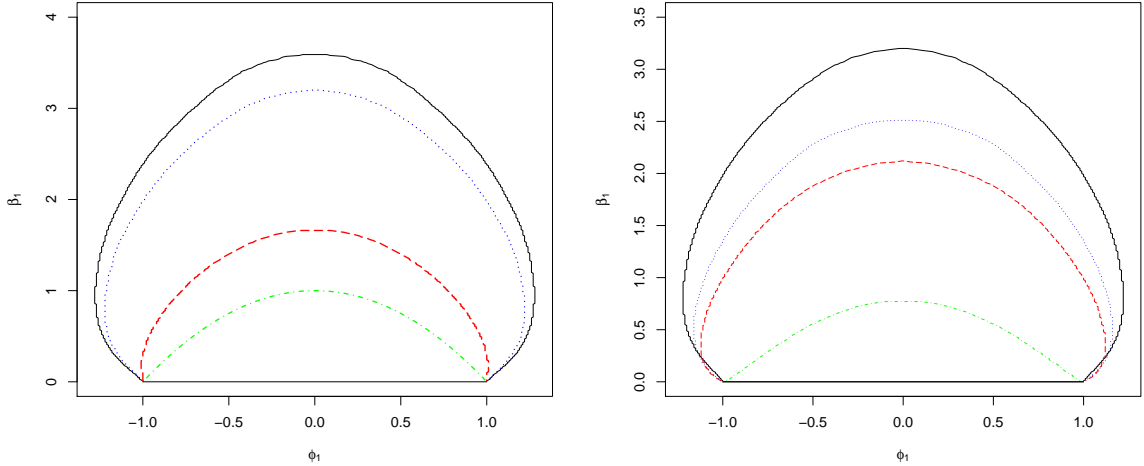}
\caption{\label{StationaryRegion} Left panel: stationarity regions of $E\ln|\phi_1+\varepsilon_t\sqrt{\beta_1}|<0$ (solid line), $E|\phi_1+\varepsilon_t\sqrt{\beta_1}|^\kappa<1$ with $\kappa=0.1$ (dotted line), $\kappa=0.9$ (longdash line), or $\kappa=2$ (dotdash line), respectively, where $\varepsilon_t$ follows a standard normal distribution; Right panel: stationarity regions of $E|\phi_1+\varepsilon_t\sqrt{\beta_1}|^\kappa<1$ with $\kappa=0.1$ and $\varepsilon_t$ being standard normal (solid line), Student's $t_5$ (dotted line), $t_3$ (longdash line) or standard Cauchy (dotdash line) random variable, respectively.}
\end{figure}

\begin{figure}
	\begin{center}
		\includegraphics[scale=0.55]{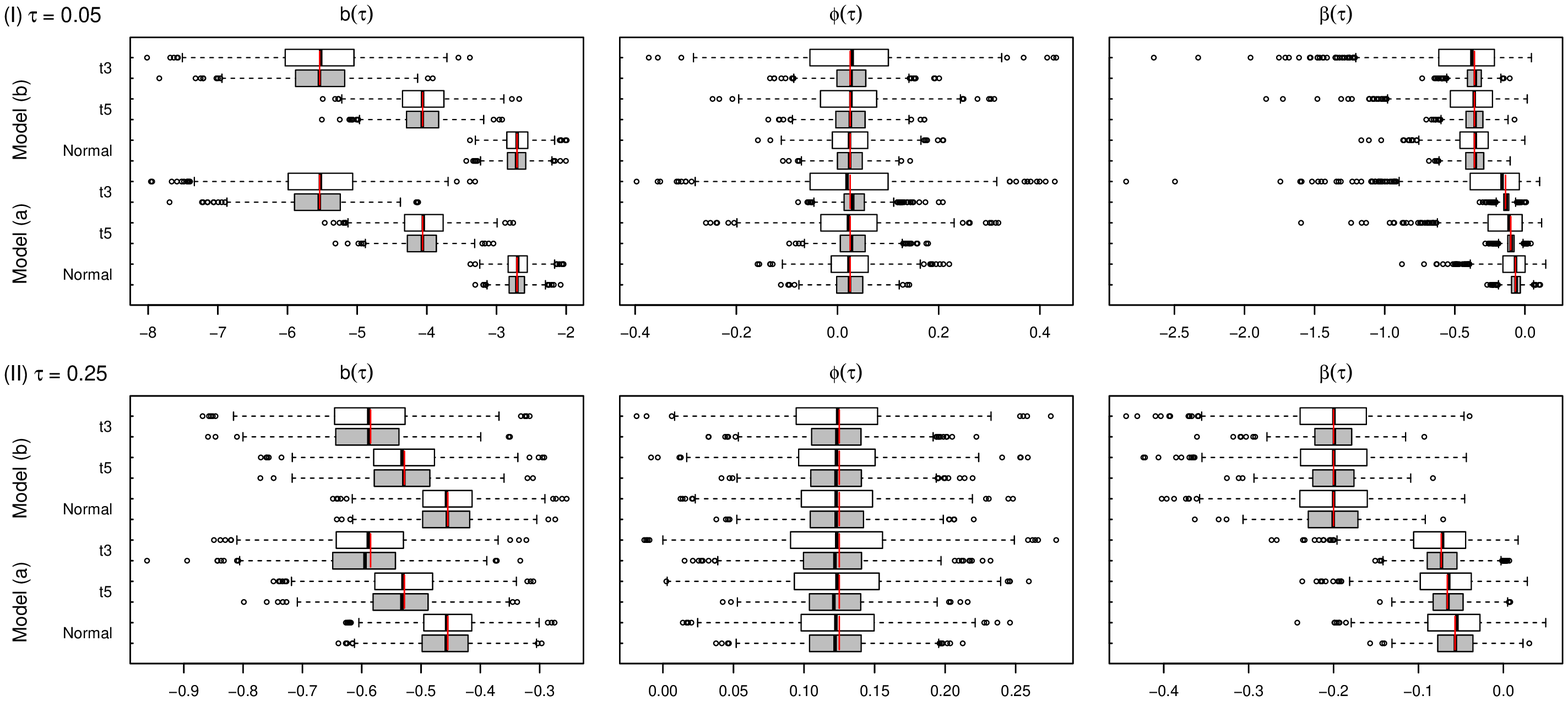}
		\caption{\label{BoxPlotEstimator} Box plots for the self-weighted estimator $\widehat{\bm\theta}_{\tau n}$ (white boxes) and the unweighted estimator $\widetilde{\bm\theta}_{\tau n}$ (grey boxes), at $\tau=0.05$ or 0.25, for two models with $F_b(\cdot)$ being the normal, Student's $t_5$ or Student's $t_3$ distribution function. Model (a): $b(\tau)=S_Q^{-1}(F_b^{-1}(\tau)), \phi(\tau)=0.5\tau, \beta(\tau)=0.5\tau b(\tau)$; Model (b): $b(\tau)=S_Q^{-1}(F_b^{-1}(\tau)), \phi(\tau)=0.5\tau, \beta(\tau)=0.8(\tau-0.5)$. The thick black line in the center of the box indicates the sample median, and the thin red line indicates the value of the corresponding element of the true parameter vector $\bm\theta_{\tau 0}$. The notations $b(\tau)$, $\phi(\tau)$ and $\beta(\tau)$ represent the corresponding elements of $\widehat{\bm\theta}_{\tau n}$ and $\widetilde{\bm\theta}_{\tau n}$.}
	\end{center}
\end{figure}

\begin{figure}
	\begin{center}
		\includegraphics[scale=0.55]{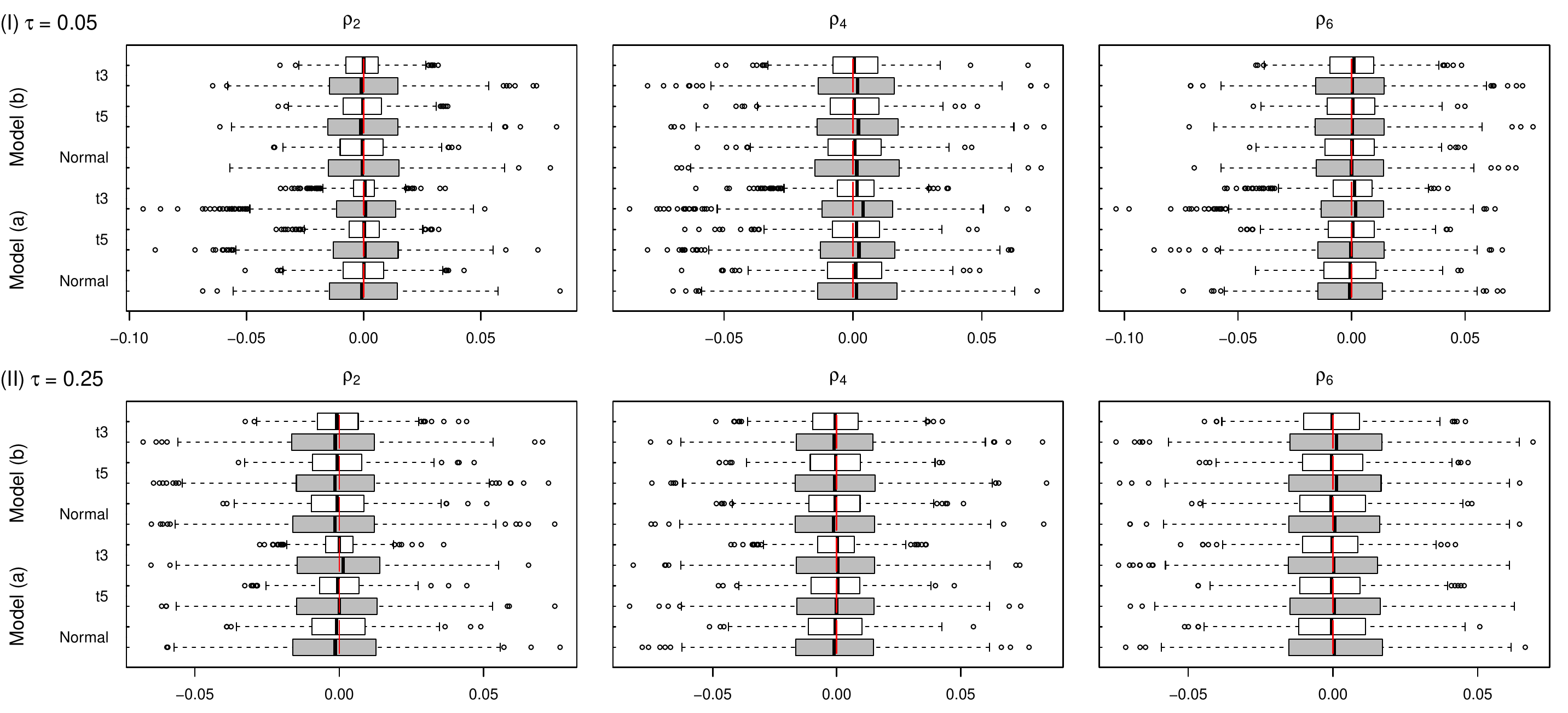}
		\caption{\label{BoxPlotRho} Box plots for the self-weighted residual QACFs $\widehat{\rho}_{k,\tau}$ (white boxes) and the unweighted residual QACFs $\widetilde{\rho}_{k,\tau}$ (grey boxes), at $\tau=0.05$ or 0.25, $k=2$, 4 or 6, for two models with $F_b(\cdot)$ being the normal, Student's $t_5$ or Student's $t_3$ distribution function. Model (a): $b(\tau)=S_Q^{-1}(F_b^{-1}(\tau)), \phi(\tau)=0.5\tau, \beta(\tau)=0.5\tau b(\tau)$; Model (b): $b(\tau)=S_Q^{-1}(F_b^{-1}(\tau)), \phi(\tau)=0.5\tau, \beta(\tau)=0.8(\tau-0.5)$. The thick black line in the center of the box indicates the sample median, and the thin red line indicates the true value zero of $\rho_{k,\tau}$ if $Q_{\tau}(y_t|\mathcal{F}_{t-1})$ is correctly specified. The notations $\rho_2$, $\rho_4$ and $\rho_6$ represent $\widehat{\rho}_{k,\tau}$ and $\widetilde{\rho}_{k,\tau}$ at $k=2$, 4 and 6.}
	\end{center}
\end{figure}

\begin{figure}
	\begin{center}
		\includegraphics[scale=0.55]{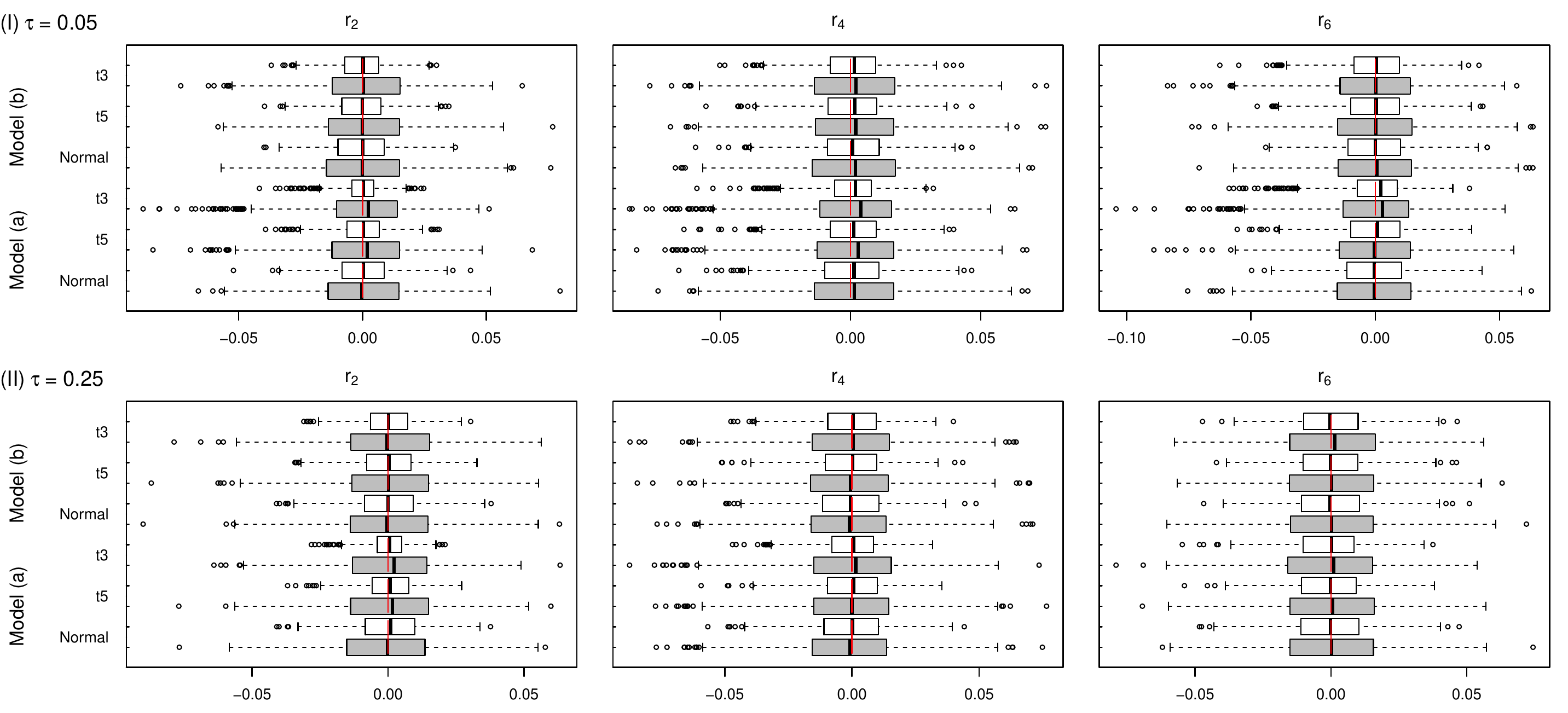}
		\caption{\label{BoxPlotR} Box plots for the self-weighted residual QACFs $\widehat{r}_{k,\tau}$ (white boxes) and the unweighted residual QACFs $\widetilde{r}_{k,\tau}$ (grey boxes), at $\tau=0.05$ or 0.25, $k=2$, 4 or 6, for two models with $F_b(\cdot)$ being the normal, Student's $t_5$ or Student's $t_3$ distribution function. Model (a): $b(\tau)=S_Q^{-1}(F_b^{-1}(\tau)), \phi(\tau)=0.5\tau, \beta(\tau)=0.5\tau b(\tau)$; Model (b): $b(\tau)=S_Q^{-1}(F_b^{-1}(\tau)), \phi(\tau)=0.5\tau, \beta(\tau)=0.8(\tau-0.5)$. The thick black line in the center of the box indicates the sample median, and the thin red line indicates the true value zero of $r_{k,\tau}$ if $Q_{\tau}(y_t|\mathcal{F}_{t-1})$ is correctly specified. The notations $r_2$, $r_4$ and $r_6$ represent $\widehat{r}_{k,\tau}$ and $\widetilde{r}_{k,\tau}$ at $k=2$, 4 and 6.}
	\end{center}
\end{figure}

\begin{figure}
\begin{center}
	\includegraphics[scale=0.55]{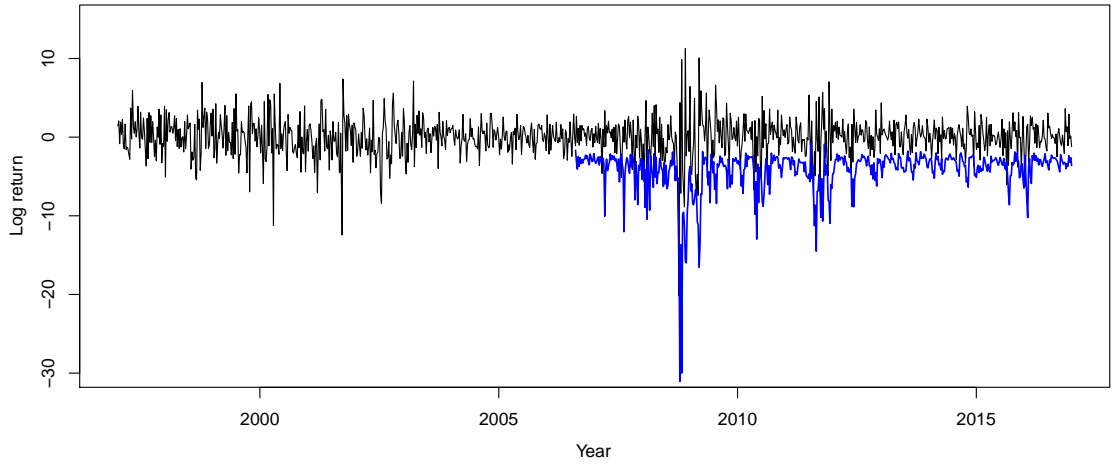}
	\caption{\label{TimePlotVaR} Time plot of weekly log returns (black line) of S\&P500 Index from January 10, 1997 to December 30, 2016, and negative $5\%$ VaR forecasts (blue line) from August 11, 2006 to December 30, 2016.}
\end{center}
\end{figure}

\begin{figure}[htb]
	\begin{center}
		\begin{tabular}{c}
			\includegraphics[scale=0.5]{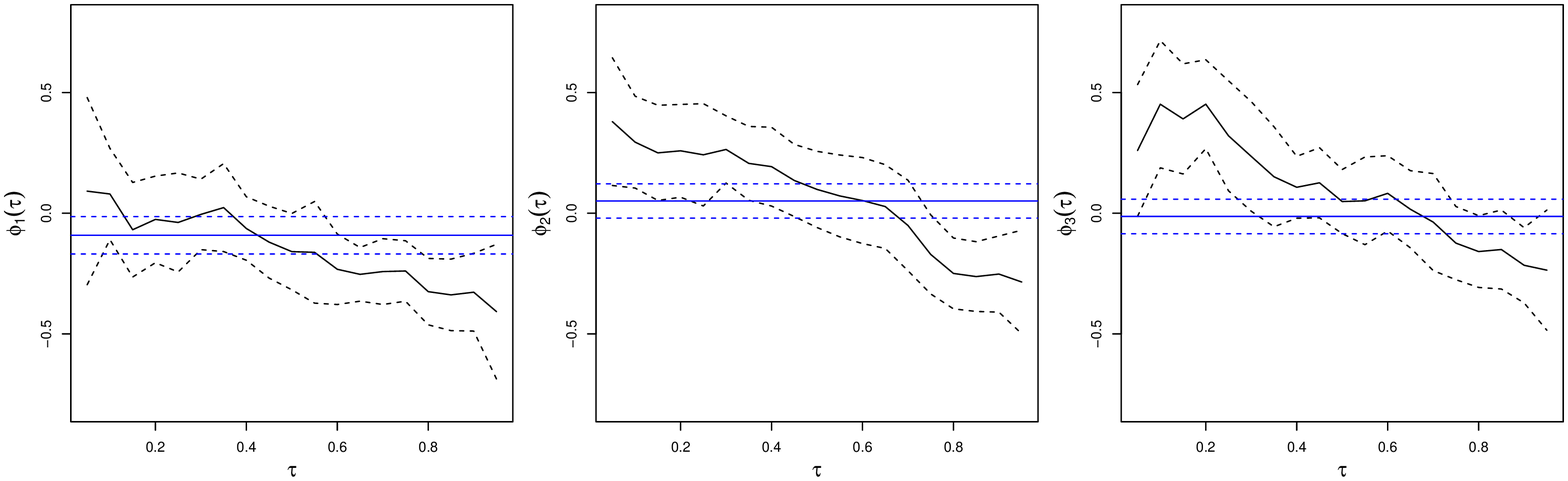}
		\end{tabular}
\caption{\label{IntroFunction} The estimates of $\phi_i(\tau_k)$ (black solid) from the fitted quantile double AR
model, together with their 95\% confidence band (black dotted), at $\tau_k=k/20$ with $1\leq k\leq 19$, and estimates of $\phi_i$ (blue solid) from the fitted double AR model, together with their 95\% confidence interval (blue dotted).}
	\end{center}
\end{figure}

\begin{figure}[htb]
	\begin{center}
		\begin{tabular}{c}
            \includegraphics[scale=0.5]{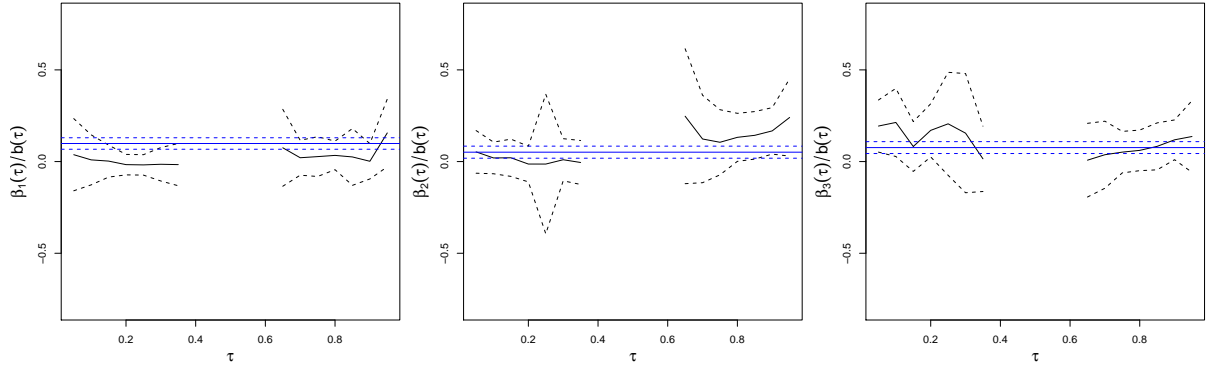}
		\end{tabular}
		\caption{\label{FunctionPlots}
The estimates of $\beta_j(\tau_k)/b(\tau_k)$ (black solid) from the fitted quantile double AR model, together with their 95\% confidence band (black dotted), at $\tau_k=k/20$ with $1\leq k \leq 7$ and $13\leq k\leq 19$, and estimates of $\beta_j/\omega$ (blue solid) from the fitted double AR model, together with their 95\% confidence interval (blue dotted).
}
	\end{center}
\end{figure}


\begin{figure}[htb]
 \begin{center}
  \includegraphics[scale=0.55]{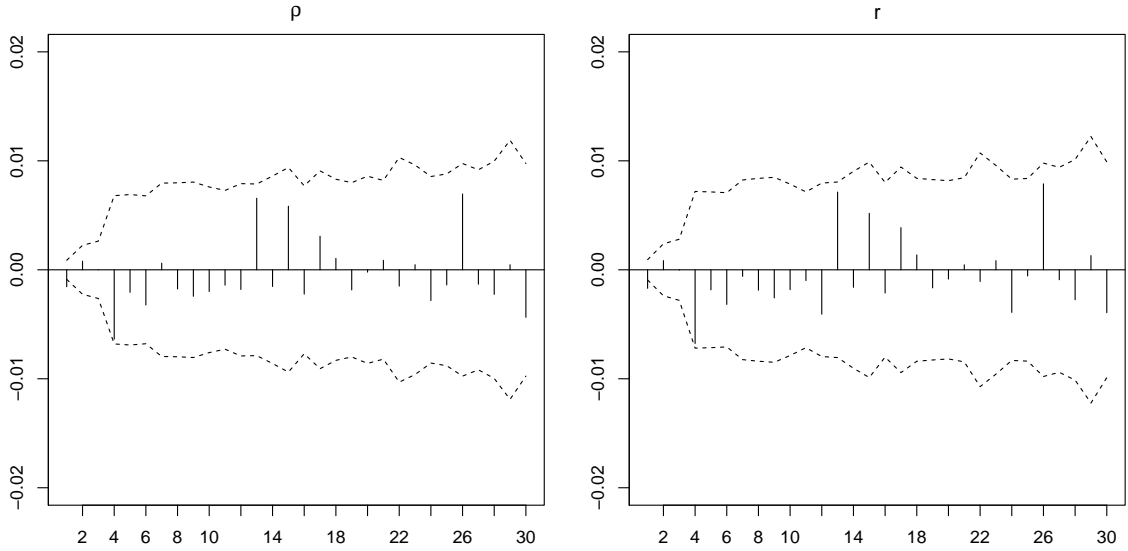}			
  \caption{\label{RhoRCI} Residual QACFs of the fitted quantile double AR model at $\tau=0.05$, with 95\% confidence bands.}
  \end{center}
 \end{figure}

\begin{table}
	\caption{\label{table_real_data_forecast} Empirical coverage rate (\%) and $p$-values of two VaR backtests of three models at the 5\%, 10\%, 90\% and 95\% conditional quantiles. M1, M2 and M3 represent the QDAR, QAR and TQAR models, respectively. }
	\begin{center}
		\begin{tabular}{crrrrrrrrrrrrrrr}
			\hline\hline 
			& \multicolumn{3}{c}{$\tau=5\%$} && \multicolumn{3}{c}{$\tau=10\%$} && \multicolumn{3}{c}{$\tau=90\%$} && \multicolumn{3}{c}{$\tau=95\%$}\\
			\cline{2-4}\cline{6-8}\cline{10-12}\cline{14-16}
			&\multicolumn{1}{c}{ECR}&\multicolumn{1}{c}{CC}&\multicolumn{1}{c}{DQ}&&\multicolumn{1}{c}{ECR}&\multicolumn{1}{c}{CC}&\multicolumn{1}{c}{DQ}&&\multicolumn{1}{c}{ECR}&\multicolumn{1}{c}{CC}&\multicolumn{1}{c}{DQ}&&\multicolumn{1}{c}{ECR}&\multicolumn{1}{c}{CC}&\multicolumn{1}{c}{DQ}\\
			\hline
			M1	&	5.34	&	0.88	&	0.33	&&	9.02	&	0.34	&	0.22	&&	91.53	&	0.25	&	0.11	&&	95.95	&	0.23	&	0.51	\\
			M2	&	5.16	&	0.17	&	0.00	&&	9.58	&	0.03	&	0.00	&&	92.45	&	0.08	&	0.00	&&	95.95	&	0.33	&	0.00	\\
			M3	&	6.45	&	0.20	&	0.01	&&	10.13	&	0.81	&	0.02	&&	91.34	&	0.19	&	0.12	&&	95.03	&	0.42	&	0.02	\\
			\hline
		\end{tabular}
	\end{center}
\end{table}

\end{document}